\title{Bounding the separable rank via polynomial optimization}

\author{Sander Gribling\thanks{Partially supported by SIRTEQ-grant QuIPP. \texttt{gribling@irif.fr}}
		\\
        IRIF, Universit\'e de Paris
        \and
		Monique Laurent\thanks{\texttt{M.Laurent@cwi.nl}} \\ 
		CWI and Tilburg University
		\and
        Andries Steenkamp\thanks{Supported by the European Union's EU Framework Programme for Research and Innovation Horizon 2020 under the Marie Sk\l{}odowska-Curie Actions Grant Agreement No 813211 (POEMA).\texttt{Andries.Steenkamp@cwi.nl} } \\
        CWI
}
\date{} 

\documentclass[11pt]{article}

\usepackage{fullpage}
\usepackage[utf8]{inputenc}
\usepackage[english]{babel}
\usepackage{amsmath}
\usepackage{amsthm}
\usepackage{amssymb}
\usepackage{blkarray}
\usepackage{mathtools}
\usepackage{mathrsfs}
\usepackage{xcolor}

\usepackage{booktabs}
\usepackage[bookmarksnumbered,hypertexnames=false,colorlinks=false,linkcolor=blue,urlcolor=blue,citecolor=blue,anchorcolor=green,breaklinks=true,pdfusetitle]{hyperref}
\usepackage[capitalize]{cleveref}

\usepackage{tikz,pgfplots}

\pgfplotsset{every axis/.style={scale only axis}}

\newtheorem{theorem}{Theorem}
\newtheorem{lemma}[theorem]{Lemma}
\newtheorem{corollary}[theorem]{Corollary}
\newtheorem{claim}[theorem]{Claim}
\newtheorem{remark}[theorem]{Remark}
\newtheorem{definition}[theorem]{Definition}
\newtheorem{proposition}[theorem]{Proposition}

\newcommand{\N}{\mathbb{N}} 		  
\newcommand{\R}{\mathbb{R}} 	      
\newcommand{\C}{\mathbb C} 			  

\newcommand{\calS}{\mathcal {S}} 
\newcommand{\calR}{{\mathcal R}}
\newcommand{\calH}{\mathcal{H}}
\renewcommand{\H}{\mathcal H} 
\newcommand{\calV}{{\mathcal V}}

\newcommand{\cone}{\mathrm{cone}}
\newcommand{\conv}{\mathrm{conv}}
\newcommand{\Span}[1]{\mathrm{Span}\left\{#1\right\}}

\newcommand{\sep}{\mathcal{SEP}}
\newcommand{\SEP}{\mathcal{SEP}}
\newcommand{\CP}{\mathcal{CP}}

\newcommand{\calM}{\mathcal {M}} 
\newcommand{\calW}{\mathcal{W}}    %
\newcommand{\calL}{\mathcal L}
\newcommand{\cL}{\mathcal L}
\newcommand{\sD}{\mathscr{D}} 
\newcommand{\bbS}{{\mathbb S}}

\newcommand{\hS}{\Sigma} 

\newcommand{\Sym}{{\text{\rm Sym}}}

\newcommand{\monoV}{[\bx,\by,\bbx,\bby]} 
\newcommand{\polyS}{\C[\bx,\by,\bbx,\bby]} 

\newcommand{\set}{{\mathcal{R}}}

\newcommand{\dps}{\mathcal{DPS}} %
\newcommand{\xisep}[1]{\xi_{#1}^{\mathrm{sep}}(\rho) } 
\newcommand{\xidsep}[1]{\xi_{#1}^{\mathrm{sep}}(\rho)}
\newcommand{\xib}[2]{\xi^{\mathrm{#1}}_{#2}}
\newcommand{\xidsepR}[1]{\xi_{#1}^{{\mathrm{sep}},\R}(\rho)}

\newcommand{\xr}{\bx_{\Re}}	   		  
\newcommand{\xim}{\bx_{\Im}}	   	  

\newcommand{\Cxox}{\C\mxox}   		  
\newcommand{\mxox}{[\bx,\bbx]} 		  
\newcommand{\mxrxi}{[\xr,\xim]} 	  
\newcommand{\Rxrxi}{\R\mxrxi}   	  
\newcommand{\Cxoxh}{\C\mxox^h}        
\newcommand{\Lr}{L^\R}   	  		  
\newcommand{\calLr}{\calL^\R} 		  

\newcommand{\Sre}{S_{\Re}} 			          
\newcommand{\calMr[1]}{\mathcal{M}_{#1}^{\R}} 
\newcommand{\sDr}{\mathscr{D}^{\R}} 	      

\newcommand{\qo}[1]{ \overline{#1} } 

\newcommand{\wh}{\widehat }
\newcommand{\wtl}{\widetilde }
\newcommand{\abs}[1]{\left|#1\right|}
\newcommand{\ol}{\overline}

\newcommand{\bx}{\mathbf{x}} 
\newcommand{\bbx}{\qo{\mathbf{x}}} 
\newcommand{\by}{\mathbf{y}} 
\newcommand{\bby}{\qo{\mathbf{y}}} 

\newcommand{\ba}{\mathbf{a}} 
\newcommand{\bb}{\mathbf{b}} 
\newcommand{\bi}{\mathbf{i}}
\renewcommand{\i}{\mathbf{i}}

\newcommand{\ox}{\overline{x}}

\newcommand{\cxbx}{[\bx,\bbx]} 
\newcommand{\cxbxs}{\cxbx^*}   
\newcommand{\Ccxbx}{\C\cxbx}   
\newcommand{\Ccxbxs}{\C\cxbxs} 
\newcommand{\Ccxbxh}{\Ccxbx^h} 

\newcommand{\avec}[1]{ \vec{{\mathbf{#1}}} } 

\newcommand{\tsep}{\tau_\mathrm{sep}}
\newcommand{\tausep}{\tau_{\mathrm{sep}}}

\newcommand{\taucp}{\tau_{\mathrm{cp}}}

\renewcommand{\i}{\mathbf{i}\,}
\renewcommand{\Re}{\mathrm{Re}}
\renewcommand{\Im}{\mathrm{Im}}

\newcommand{\haL}{\widehat{L}}

\newcommand{\seprank}{\mathrm{rank_{sep}}}
\newcommand{\cprank}{\mathrm{rank_{cp}}}
\newcommand{\rank}{\mathrm{rank}}
\newcommand{\birank}{\mathrm{birank}} 
\newcommand{\seprankR}{\mathrm{rank}_{\mathrm{sep}}^\R}

\newcommand{\Tr}{\mathrm{Tr}} 

\newcommand{\Mnc}{M^{\mathrm{nc}}} 
\newcommand{\Gr}{{G_{\rho}}} 	   
\DeclarePairedDelimiter\inip{\langle}{\rangle}  

\begin{document}
\maketitle

\begin{abstract}
We investigate questions related to the set $\sep_d$ consisting of the linear maps $\rho$ acting on $\C^d\otimes \C^d$ that can be written as a convex combination of rank one matrices of the form $xx^*\otimes yy^*$. Such maps are known in  quantum information theory as the separable bipartite states, while nonseparable states are called entangled. In particular we introduce
bounds for the separable rank $\seprank(\rho)$, defined as the  smallest number of rank one states  $xx^*\otimes yy^*$ entering the decomposition of a  separable state $\rho$. Our approach relies on the moment method and yields a hierarchy of semidefinite-based lower bounds, that converges to a parameter $\tausep(\rho)$, a natural convexification of the combinatorial parameter $\seprank(\rho)$.  A distinguishing feature is exploiting the positivity constraint $\rho-xx^*\otimes yy^* \succeq 0$ to impose positivity of a polynomial matrix localizing map, the dual notion of the notion of sum-of-squares polynomial matrices. Our approach extends naturally to the multipartite setting and to
the real separable rank, and it permits strengthening some known bounds for the completely positive rank.
In addition, we indicate how the moment approach also applies to define hierarchies of semidefinite relaxations for the set $\sep_d$ and permits to give new proofs,  using only tools from moment theory, for  convergence results on the DPS hierarchy from (A.C. Doherty, P.A. Parrilo and F.M. Spedalieri. Distinguishing separable and entangled states. Phys. Rev. Lett. 88(18):187904, 2002). 
\end{abstract} 

%

\section{Introduction}\label{Intro}

The main object of study in this paper is the following matrix cone
\begin{equation}\label{SEP}
\SEP_d := \cone\{xx^* \otimes yy^* \colon x\in \C^d,~y\in \C^d,~\|x\| = \|y\| = 1 \} \subseteq \H^d \otimes \H^d\simeq \H^{d^2},
\end{equation}
sometimes also denoted  as $\SEP$ when the dimension $d$ is not important. Throughout $\H^d$ denotes the cone of complex Hermitian $d\times d$  matrices and $\H^d_+$ is the subcone of Hermitian positive semidefinite matrices. Matrices in $\H^d_+$ are also known as {\em unnormalized states} and  matrices in $\H^d_+$ with trace~1 are called  {\em normalized states}.
The cone $\sep_d$ is of particular interest in the area of quantum information theory: its elements are known as the {\em (unnormalized, bipartite)  separable states} on  $\H^d \otimes \H^d$ and 
a positive semidefinite matrix $\rho \in \H^d \otimes \H^d$ that does not belong to $\SEP_d$ is said to be  \emph{entangled}. 
Entangled states can be used to observe quantum, non-classical behaviors that may be displayed by two physically separated quantum systems, as  already pointed out in the early work \cite{Einstein}.
Entanglement is  now recognized as an additional important resource that can be used in quantum information processing to carry out a great variety of tasks such as quantum computation, quantum communication, quantum cryptography and teleportation (see, e.g.,  \cite{Nielsen-Chuang,Watrous} and references therein).
Therefore, deciding whether a state is separable or entangled 
is a question of fundamental interest in quantum information theory.  Gurvits~\cite{Gurvits} has shown that 
the (weak) membership problem for the set $\sep_d\cap\{\rho: \Tr(\rho)=1\}$ is an NP-hard problem. In addition, the problem was shown to be  strongly NP-hard in~\cite{Gharibian10}.
Hence  it is important to have tractable criteria for  separability or entanglement of quantum states. Throughout, we restrict for simplicity to the case of bipartite states, acting on two copies of $\C^d$, but the treatment extends naturally to the case of $m$-partite states that act on $\C^{d_1}\otimes \ldots \otimes \C^{d_m}$ with $m\ge 2$ and $d_1,\ldots,d_m$ possibly distinct. We will return below to the question of testing separability, but first we introduce the relevant notion of separable rank, which plays a central role in this paper.

\subsubsection*{The separable rank} 
In this work we consider the following problem: given a state $\rho \in \SEP_d$, what is the smallest integer $r \in \N$ such that there exist vectors $a_1,\ldots,a_r, b_1,\ldots,b_r \in \C^d$ for which 
\begin{equation}\label{eqrho}
\rho = \sum_{\ell=1}^r a_\ell a_\ell^* \otimes b_\ell b_\ell^*.
\end{equation}
This smallest integer $r$ is called the \emph{separable rank} of $\rho$ and denoted as $\seprank(\rho)$. One  sets $\seprank(\rho)=\infty$ when $\rho$ is entangled. The separable rank has been previously studied, e.g., in~\cite{Uhlmann98,DTT00,Chen_2013a} (where it is called the \emph{optimal ensemble cardinality} or the  \emph{length} of $\rho$) and it can be seen as a `complexity measure'  of the state (with an infinite rank for entangled states).  Easy bounds on the separable rank are $\rank(\rho)\le \seprank(\rho)\le \rank(\rho)^2$, where the left most inequality can be strict (see \cite{DTT00}) and the right most one follows using Caratheodory's theorem \cite{Uhlmann98}.
We approach the problem of determining the separable rank from the moment perspective. 
We use the observation that, if $\seprank(\rho)=r$ and $\rho$ admits the decomposition (\ref{eqrho}), then 
 the sum of the $r$ atomic measures at the vectors $(a_\ell, b_\ell)\in \C^d\times \C^d$ is a measure $\mu$ whose expectation $\int 1d\mu$ is equal to $r$ and whose fourth-degree moments correspond to the entries of $\rho$. Moreover, as we will see later, this measure may  be assumed to be  supported on the semi-algebraic set
\begin{equation}\label{eqscaling}
\calV_\rho= \{(x,y) \in \C^d\times \C^d : \|x\|_\infty^2,\|y\|_\infty^2\leq \sqrt{\rho_{\max}},\ xx^* \otimes yy^* \preceq \rho\},
\end{equation}
where $\rho_{\max}$ denotes the largest diagonal entry of $\rho$. Then we obtain a lower bound on the separable rank of $\rho$, denoted $\tsep(\rho)$,  by minimizing the expectation $\int 1 \, d\mu$ over all  measures $\mu$  that are supported on $\calV_\rho$ and have fourth-degree moments corresponding to entries of $\rho$  (see~\cref{eq:tausep}).
Hence, here we view  the separable rank as a moment problem over the product of two balls. This view will enable us to design a hierarchy of tractable semidefinite based parameters, denoted  $\xidsep{t}$. These parameters provide lower bounds on the separable rank and converge to $\tsep(\rho)$ (see Section~\ref{sec: poly sep}).

In view of the definition of $\sep_d$ in \cref{SEP}, one may also view separability of a state $\rho$ as a moment problem on the bi-sphere $\bbS^{d-1}\times \bbS^{d-1}$, where  $\bbS^{d-1}=\{x\in\C^d: \|x\|=1\}$ denotes the (complex) unit sphere. 
However, this approach  does {\em not} (straightforwardly) lead to  bounds on the separable rank. Indeed, for a measure $\mu$ on the bi-sphere whose fourth-degree moments correspond to entries of $\rho$, we necessarily have~$\int 1 \, d\mu = \Tr(\rho)$. To get bounds on the separable rank, it is thus {crucial} to use another scaling for the points $(a_\ell,b_\ell)$ entering a separable decomposition of $\rho$ as, for instance, the scaling used in \cref{eqscaling}, but other scalings are possible as indicated in Section \ref{sec:tausep}. 

Our approach extends to several other settings, in particular, to the case of multipartite separable states (when $\rho$ acts on the tensor product of more than two spaces) and    to the case of real states (instead of complex valued ones). It can also be adapted to the notion of \emph{mixed separable rank}, where one tries to find factorizations of the form $\rho = \sum_{\ell=1}^r A_\ell \otimes B_\ell$ with $A_\ell, B_\ell$  Hermitian positive semidefinite matrices and $r$ as small as possible. In~\cite{de_las_Cuevas_2020} it was shown that, if $\rho$ is a diagonal matrix, then its  mixed separable rank is equal to the nonnegative rank of the associated $d \times d$ matrix consisting of the diagonal entries of $\rho$. Vavasis \cite{Vav09} has shown  that computing the nonnegative rank of a matrix is an NP-hard problem and, more recently, Shitov~\cite{Shi16b} showed   $\exists\R$-hardness of this problem. 
Hence computing the mixed separable rank has the same hardness complexity status as the nonnegative rank. 
Determining the complexity status of the separable rank remains open, but there is no reason to expect that it should be any easier than the mixed separable rank.

When using moment methods one typically works with measures supported on semi-algebraic sets, i.e., sets described by polynomial inequalities on the variables. In our approach this is also the case. Indeed  the set $\calV_\rho$ in \cref{eqscaling} is semi-algebraic since one can encode the condition $xx^* \otimes yy^* \preceq \rho$ by requiring all principal minors of $\rho-xx^* \otimes yy^*$ to be nonnegative. This would however lead to a description of the set $\calV_\rho$ with a number of polynomial constraints that is exponential in $d$. Instead, we will directly exploit the  constraint $\rho- xx^* \otimes yy^*\succeq 0$, which is of the form $G(x) \succeq 0$ for some polynomial matrix $G(x)$ (i.e., with entries polynomials in $x,\overline x$). This constraint enables us to impose positivity constraints on polynomial matrix localizing maps, a matrix analog of the usual scalar localizing maps used in the moment method (see Section \ref{sec:matrix-setting}).  
Such polynomial matrix localizing constraints can also be used to bound the completely positive rank of a completely positive matrix, and we will show that this permits to strengthen some known bounds on the completely positive rank from~\cite{GdLL17a} (see \cref{sec:CP}).

Our hierarchy of bounds $\xidsep{t}$ on the separable rank can also be used to detect entanglement. Indeed, 
as mentioned above, by Caratheodory's theorem,   the separable rank of a  state  $\rho \in \sep_d$ can be  upper bounded, e.g., by $\rank(\rho)^2\le d^4$.
We can leverage this fact and the asymptotic convergence of our hierarchy of lower bounds to detect entanglement: a state $\rho\in\H^d\otimes \H^d$ is entangled if and only if $\xidsep{t}>\rank(\rho)^2$ for some integer $t\ge 1$, i.e., there is a level of our hierarchy which is infeasible or   provides a  lower bound on $\seprank(\rho)$ which is strictly larger than Caratheodory's bound.  In addition, a certificate of entanglement is then provided by the dual semidefinite program. Hence our hierarchy of semidefinite parameters $\xidsep{t}$ can  also be used to provide a type of entanglement witnesses (see \cref{sec:mem}).

\subsubsection*{The Doherty-Parrilo-Spedalieri (DPS) hierarchy for $\sep$} 
As mentioned above a fundamental problem in quantum information theory is to have efficient criteria for checking separability or entanglement of quantum states. A second main contribution of our work concerns a hierarchy of outer approximations to the set $\sep$ that we describe now. Doherty, Parrilo, and Spedalieri~\cite{DPS04} designed what is now known as the \emph{DPS hierarchy}, a hierarchy of outer approximations $\dps_{1,t}$ ($t \geq 1$) for the set $\SEP$. It is based on the principle of {\em state extension}: if $\rho:=\sum_\ell \lambda_\ell x_\ell x_\ell^* \otimes y_\ell y^*_\ell\in\sep_d$ with $\lambda_\ell\ge 0$ then, for any integer $t\ge 1$,  $\rho$ admits an extension $\rho_{1,t} := \sum_\ell \lambda_\ell x_\ell x_\ell^* \otimes (y_\ell y^*_\ell)^{\otimes t}$ acting on $\C^d\otimes (\C^d)^{\otimes t}$. The state $\rho$ can be recovered from its extension $\rho_{1,t}$ by tracing out $t-1$ of the copies of the second space and the extension $\rho_{1,t}$ satisfies several natural   conditions such as symmetry (under permuting the $t$ copies of the second register) and the so-called {\em positive partial transpose (PPT)} criterion from \cite{Horodecki_1996} (which states that taking the transpose of some of the copies preserves positive semidefiniteness). The relaxation  $\dps_{1,t}$ consists of those $\rho$ for which a state $\rho_{1,t}$ exists satisfying these necessary conditions.  Here the state extension is one-sided (since one extends only in the $y$-direction); the two-sided analog (in both $x$- and $y$-directions) has also been considered, leading to the hierarchy $\dps_{t,t}\subseteq \dps_{1,t}$ (see Section~\ref{sec: state extension perspective} for details).   For fixed $t$, deciding membership in $\dps_{1,t}$ (or $\dps_{t,t}$) boils down to testing feasibility of a semidefinite program of size polynomial in $d$.
The DPS hierarchy is {\em complete}, in the sense that we have  equality: $\bigcap_{t\ge 1}\dps_{1,t}=\sep_d$ \cite{DPS04}.

One can also interpret the set $\sep$ in the language of moments of distributions on the bi-sphere: $\rho$ is separable if there exists an atomic measure on the bi-sphere whose fourth-degree moments agree with $\rho$ (see, e.g., \cite{dressler2020separability,HNW17,Li_2020}). Another main contribution in this paper will be to make the links between this moment approach and the DPS hierarchy more apparent. These links enable us to give an alternative proof of completeness for the DPS hierarchy that is based on the theory of positive-operator valued measures. In contrast, existing proofs rely on other tools such as quantum de Finetti theorems or sums of squares. Indeed one can also design approximation hierarchies  for $\sep_d$, starting from its definition in \cref{SEP} and applying the moment approach to the bi-sphere $\bbS^{d-1}\times \bbS^{d-1}$.
Depending on the degrees that are allowed in the $x,\ol x$ variables and the $y,\ol y$ variables, this  leads to  several possible variants of relaxations for $\sep_d$ that we explore in Section~\ref{sec:momentDPS},  denoted there as $\set_{t}$ (when the full degree is at most $2t$), $\set_{t,t}$ (when the degree in $x,\ol x$ is at most $2t$ and the same for the degree in $y,\ol y$) and $\set_{1,t}$ (when the degree in $x,\ol x$ is at most 2 and the degree in $y,\ol y$ is at most $2t$).
We provide a convergence proof for each of these hierarchies (i.e., show their completeness) using tools from the moment method (i.e., existence of an atomic representing measure under certain positivity conditions), which we apply to the setting of matrix polynomials for the hierarchy $\set_{1,t}$ (see \cref{sec:convergenceR1t}).
In addition we show that the hierarchy  $\set_{1,t}$ (resp., $\set_{t,t}$) coincides with the DPS hierarchy $\dps_{1,t}$ (resp., $\dps_{t,t}$). Therefore we offer a new convergence proof for the DPS hierarchy that is based on the moment method.

\subsubsection*{Related literature on approximation hierarchies for $\sep$}

There is a vast literature about the set $\sep$ of separable states and approximations thereof (such as the DPS hierarchy), so we only mention here some of the  results that are most relevant to this paper.
The PPT criterion, introduced in \cite{Peres1996,Horodecki_1996}, is a necessary condition for separability. While it was shown to be sufficient to ensure separability of bipartite states acting on $\C^2\otimes \C^3$ \cite{WORONOWICZ1976165}, it is in general not sufficient for separability of states acting on larger dimensional spaces (see, e.g.,  \cite{Horodecki_1997,WORONOWICZ1976165}). In fact it has been shown that no semidefinite representation exists for $\sep_d$ when $d\ge 3$ \cite{Fawzi2021}.
As mentioned above, the authors in \cite{DPS04} use symmetric state extensions and the PPT conditions to define the hierarchy $\dps_{1,t}$ ($t \geq 1$). They show it to be complete (i.e.,  $\cap_{t \geq 1} \dps_{1,t} = \sep$) using the quantum de Finetti theorem from  \cite{Caves-Fuchs-Schack} (note that this completeness proof in fact does not use the PPT conditions).

Navascues, Owari and Plenio \cite{NOP09} show a quantitative result on the convergence of the sets $\dps_{1,t}$ to $\sep_d$. Consider  $\rho\in \dps_{1,t}$, whose membership is certified by the extended state  $\rho_{1,t}$ acting on $\C^d\otimes (\C^d)^{\otimes t}$, and let  $\rho_1\in \H^d$ be obtained by tracing out the part of $\rho_{1,t}$ that acts on $(\C^d)^{\otimes t}$; then $\rho_1\otimes I_d$ is clearly separable. In \cite{NOP09} it is shown that 
\begin{equation}\label{eqrhot}
\tilde \rho:= (1-\epsilon)\rho+ \epsilon \big(\rho_1\otimes {I_d\over d}\big)\in \sep_d \ \ \text{  where } \epsilon =O\big(\big({d\over t}\big)^2\big);
\end{equation}
that is, by moving $\rho$ in the direction of $\rho_1\otimes I_d/d$ by $\epsilon= O\big(\big({d\over t}\big)^2\big)$, one finds a separable state.  

{An \emph{entanglement witness} for a state $\rho$ is any certificate that certifies $\rho \not \in \SEP_d$.
 One way to obtain such an entanglement witness  is to exhibit one of the constraints defining a relaxation of $\sep_d$ (such as $\dps_{1,t}$) that is violated by $\rho$, like for example, one of the PPT conditions. More generally one can obtain an entanglement witness for $\rho\not\in\sep_d$ by finding a hyperplane separating $\rho$ and $\sep_d$, i.e., a matrix
$W \in \H^d \otimes \H^d$ such that 
\begin{equation}\label{eqhsep}
\Tr(W\rho) > h_{\sep}(W):=\max \{ \Tr(W \sigma)\colon \sigma \in \SEP_d\},
\end{equation}
 which shows again the importance of linear optimization over the set $\SEP_d$ and of designing tractable relaxations for $\sep_d$.} The function   $h_{\sep}(W)$ in \cref{eqhsep} is  known as the {\em support function} of $\sep_d$ in the direction $W$. Analogously define the support function of $\dps_{1,t}$ as 
$$h_{\dps_{1,t}}(W):=\max \{\Tr(W\rho): \rho\in\dps_{1,t}\}.
$$
As an  application of the  quantitative result in \cref{eqrhot} the following is shown in \cite{NOP09}:
\begin{equation}\label{eqhsepdps}
h_{\sep}(W)\le h_{\dps_{1,t}}(W) \le \big(1+O\big(\Big({d\over t}\Big)^2\big)\big) h_{\sep}(W).
\end{equation}
Clearly, either \cref{eqrhot} or \cref{eqhsepdps}  implies  equality $\bigcap_{t\ge 1}\dps_{1,t}= \sep_d$, i.e., completeness of the DPS hierarchy.

Fang and Fawzi \cite{FangFawzi} investigate the DPS hierarchy from the dual sum-of-squares perspective. 
In particular, they show a representation result for matrix polynomials that are nonnegative on the sphere,  which they use to give  an alternative proof for \cref{eqhsepdps}.
Namely, they show that, if $F$ is a polynomial matrix in $d$ variables and degree $2k$ such that $0\preceq F(x)\preceq I$ on $\bbS^{d-1}$ then,  for all $t\ge C_k d$,
 $F(x)+ C'_k\big({d\over t}\big)^2 I$ is a Hermitian sum-of-squares matrix polynomial of degree $2t$ on $\bbS^{d-1}$, where $C_k,C'_k$ are constants depending only $k$.
In addition,  detailed proofs are given in \cite{FangFawzi} for the description of the dual cones of the cones  $\dps_{1,t}$:
while  the dual cone of $\sep_d$ consists of the matrices $W$ for which the polynomial $p_W:=\langle W, xx^*\otimes yy^*\rangle$ is nonnegative on the bi-sphere $\bbS^{d-1}\times \bbS^{d-1}$,  the dual cone of $\dps_{1,t}$ consists of the $W$'s for which the polynomial $\|y\|^{2(t-1)}p_W$ is a sum of Hermitian squares.

For the problem of approximating the support function $h_{\sep}(W)$, Harrow, Natarayan and Wu \cite{HNW17} propose to strengthen the set $\dps_{1,t}$ by adding equality constraints arising from the classical optimality conditions. In this way they obtain a hierarchy of bounds for $h_{\sep}(W)$, stronger than $h_{\dps_{1,t}}(W)$, that converges in finitely many steps to $h_{\sep}(W)$.

Li and Ni \cite{Li_2020} use the moment approach on the bi-sphere for testing separability of a state $\rho$ (in the general multipartite setting).
For this, given a generic sum-of-squares polynomial $F$, they consider the problem of minimizing the expectation $\int Fd\mu$ over  the probability measures $\mu$  on the bi-sphere whose degree-4 moments correspond to the entries of $\rho$, and the corresponding moment relaxations (whose constraints are essentially those in the program defining the set $\set_{t}$ in \cref{eq: dpst}). Then a separability certificate can be obtained at a finite relaxation level when the optimal solution satisfies the so-called flatness condition.  Note that the separability problem only asks for the existence of such a measure $\mu$, thus, it is a feasibility problem.
The optimization approach  in  \cite{Li_2020}, based on optimizing a generic polynomial $F$, relies on the fact that  this `encourages' flatness of an optimal solution (which then permits to get a separable decomposition and thus a certificate of separability).
Indeed Nie~\cite{Nie14} shows  that if both the objective and constraints of a polynomial optimization problem are generic, then flatness occurs at some finite relaxation level.
Dressler, Nie, and Yang~\cite{dressler2020separability} strengthen the approach in \cite{Li_2020}: they use a symmetry argument which permits to  replace the bi-sphere by its subset consisting of the points $(x,y)\in \C^d\times \C^d$ that have $x_1,y_1$ real and nonnegative.  This provides a formulation that uses  less real variables ($2(2d-1)$ instead of $4d$) and leads to stronger and more economical moment relaxations. Separability of real states is considered in  \cite{NieZhang2016}, where a similar reduction is applied, namely by restricting to the vectors $(x,y)$ in  the (real) bi-sphere   satisfying $\sum_{i=1}^d x_i\ge 0$ and $\sum_{i=1}^d y_i\ge 0$.

\subsubsection*{Related literature on factorization ranks} 
Various notions of  ``factorization ranks'' have been studied extensively in the literature such as  (versions of) tensor ranks \cite{Kolda}, nonnegative matrix factorization (NMF) rank \cite{Gillis}, positive semidefinite matrix factorization rank \cite{FGPRT}, completely positive matrix factorization rank \cite{BSM03}; we refer to these references  and further references therein for details. Given the importance of factorizations for applications, designing algorithmic methods for finding a factorization of a given type (when it exists) is a topic of ongoing research (see, e.g., \cite{Gillis,DSSV,Sponsel-Dur} and references therein). The above mentioned factorization ranks are often hard to compute (see \cite{Vav09,Shi16,Shi16b} for nonnegative rank, \cite{Shi17} for positive semidefinite rank, \cite{Hastad90} for tensor rank), which motivates the search for good bounds for a given factorization rank.  Such bounds can be obtained using a variety of techniques. For example, using dedicated combinatorial methods  (see, e.g., \cite{FGPRT} and references therein), optimization methods (see, e.g., \cite{FP16}), or using a moment-based approach as we do here. A moment-based approach has previously been used to derive hierarchies of bounds for the rank of tensors~\cite{TS15}, for the symmetric nuclear norm of tensors~\cite{Nie17}, for the nonnegative rank, the completely positive rank, the positive semidefinite rank,  and the completely positive semidefinite rank of matrices~\cite{GdLL17a}. In this paper, we consider the separable rank, a notion which has been present in the (quantum information theory) literature, although no systematic study of bounds for it has been carried out so far to the best of our knowledge.

\subsubsection*{Contents of the paper}
The paper is organized as follows. In Section \ref{sec:preliminaries} we introduce the preliminaries on polynomial optimization that we will need in the rest of the paper. 
In particular, in Section \ref{sec:matrix-setting}, we introduce some of the main notions in the general setting of sum-of-squares matrix polynomials and matrix-valued linear maps.
In Section \ref{sec:moment} we recall the moment method and present the main underlying results from real algebraic geometry and moment theory. Since some of these results are presented in the literature in the real setting while we need the complex setting, we give arguments on how to extend the results from real to complex in Appendix \ref{sec:AppA}.  Section~\ref{sec: LB} is devoted to the new  hierarchy of bounds for the separable rank.
In Section \ref{sec:extension} we indicate several extensions of our approach, in particular for the real separable rank of  real states
and for getting improved bounds on the completely positive rank. We also present numerical results on examples to illustrate the behavior of the bounds in Section \ref{secnumerics}.
Finally, in Section \ref{secDPS} we revisit the Doherty-Parrilo-Spedalieri hierarchy of relaxations for the set $\sep$ of separable states. In particular,  we provide a new, alternative proof for their completeness, that uses 
the tools from the moment approach previously developed.

\section{Preliminaries on polynomial optimization}\label{sec:preliminaries}

In this section, we group some preliminaries about polynomial optimization that we need in the rest of the paper; for a general reference we refer, e.g., to \cite{Las2001,Las2009,Laurent2009} and further references therein. We will deal with polynomial optimization in real and complex variables, which is the setting needed for the application to the set of separable states and the separable rank treated in this paper, and we  will also need to deal with polynomial matrices and matrix-valued linear maps. 

\subsection{Polynomials, linear functionals and moment matrices}\label{secNotation}
We first fix some notation that we use throughout the paper.
$\N$ denotes the set of nonnegative  integers. We set $[n]=\{1,2,...,n\}$ for an integer $n\ge 1$, $[k,n]=\{k,k+1,\ldots,n-1,n\}$ for  integers $k\le n$, and  $|\alpha| = \sum_{i = 1}^n \alpha_i$ for $\alpha\in \N^n$.

For a complex matrix $X$ we denote its transpose by $X^T$ and its conjugate transpose by $X^*$.
For a scalar $a \in \C$ its conjugate is $a^*=\overline a$ and its modulus is $|a| = \sqrt{a^*a}$.
The vector space $\C^n$ is equipped with the scalar product $\inip{x,y} = x^*y=\sum_{i = 1}^{n} x_i^*y_j$ for $x,y\in \C^n$
and the Euclidean norm of $x\in\C^n$ is  $\|x\| = \sqrt{x^*x}$.
Analogously, $\C^{n\times n}$ is equipped with the trace inner product 
$\langle X,Y\rangle=\text{Tr}(X^*Y)=\sum_{i,j=1}^n \overline X_{ij}Y_{ij}$ and $\|X\|=\sqrt{\langle X,X\rangle}$ for $X\in \C^{n\times n}$. 
A matrix $X\in \C^{n\times n}$ is called Hermitian if $X^*=X$ and we let $\calH^{n}$
denote the space of complex Hermitian $n \times n$ matrices, A matrix  $X \in \calH^{n} $ is  positive semidefinite (denoted $X \succeq 0$) if $v^* A v \geq 0$ for all $v \in \C^n$. We let $\calH^n_+$ denote the cone of Hermitian positive semidefinite matrices.

For a set $S$ in a vector space, we let $\cone(S)$ and $\conv(S)$ denote, respectively,  its conic hull and its convex hull.

\paragraph{Polynomials.}
We consider polynomials in $n$ complex variables $x_1,\ldots, x_n$ and their conjugates $\overline{x_1},\ldots,\overline{x_n}$. For $\alpha,\beta \in \N^n$ we use the short-hand $\bx^\alpha \bbx^\beta$ to denote the monomial 
\[
\bx^\alpha \bbx^\beta = \prod_{i=1}^n x_i^{\alpha_i} \prod_{j=1}^n \overline{x_j}^{\beta_j}.
\]
The degree of this monomial, denoted by $\deg(\bx^\alpha \bbx^\beta)$, is equal to $|\alpha|+|\beta| = \sum_{i=1}^n \alpha_i + \beta_i$. We collect the set of all monomials of degree at most $t \in \N\cup \{\infty\}$ in the vector $\cxbx_t$ (using some given ordering of the monomials) and  also set $[\bx,\bbx]=[\bx,\bbx]_\infty$.  We interpret $\cxbx_t$ as a set when we write $\bx^\alpha \bbx^\beta \in \cxbx_t$.  
Taking the complex linear span of all monomials in $\cxbx_t$ gives the space of polynomials with complex coefficients and  degree at most $t$:
\[
\Ccxbx_t := \Span{m ~|~ m \in \cxbx_t}=\Big\{\sum_{m\in [\bx,\bbx]_t}a_m m: a_m\in \C\Big\}.
\]
For $t=\infty$ we obtain the full polynomial ring in $\bx,\bbx$ over $\C$, also denoted as $\C[\bx,\bbx]$.
So any polynomial $p\in \C[\bx,\bbx]$ is of the form $p=\sum_{\alpha,\beta}p_{\alpha,\beta} \bx^\alpha\bbx^\beta$, where only finitely many coefficients $p_{\alpha,\beta}$ are nonzero; its  {\em degree} 
 is   the maximum degree of the monomials occurring in $p$ with a nonzero coefficient,  i.e., $\deg(p) = \max_{p_{\alpha,\beta} \neq 0} \deg(\bx^\alpha \bbx^\beta)$. 
For convenience let $\C^{\N^n\times \N^n}_0$ denote the set of vectors  $\ba=(a_{\alpha,\beta})_{(\alpha,\beta)\in\N^n\times \N^n} $ that have only finitely many nonzero entries. Then any polynomial $p$ can be written as 
$p=\ba^*[\bx,\bbx]$, where we set $\ba =( \overline p_{\alpha,\beta})\in \C^{\N^n\times \N^n}_0$ (the conjugate of the vector of coefficients of $p$).

  Conjugation on complex variables extends linearly to polynomials: for $p = \sum_{\alpha,\beta} p_{\alpha,\beta} \bx^\alpha \bbx^\beta$ we define its conjugate polynomial  $\overline p
=  \sum_{\alpha,\beta} \overline{p}_{\alpha,\beta} \bbx^\alpha \bx^\beta$. Then, $p$ is called {\em Hermitian} if $p = \overline p$. Hermitian polynomials only take real values: $p(x)\in \R$ for all $x\in \C^n$. 
We denote the space of Hermitian polynomials by $\Ccxbxh$. For instance, the polynomial $p=x+\ox$ is Hermitian as well as $p=\bi x-\bi \ox$,  but $q=x-\ox$ is not Hermitian (note $q(\bi)=2\bi\not\in \R$), where $\bi=\sqrt{-1}\in\C$.

To capture positivity on the ring of polynomials, we work with the cone of Hermitian sums of squares. 
Any polynomial of the form $q \overline q$ (for some $q\in \C[\bx,\bbx]$) is called a {\em Hermitian square} and $\Sigma[\bx,\bbx]$ (or simply $\Sigma$) denotes  the conic hull of Hermitian squares.
For any integer $t\in\N$ we let  $\hS[\bx,\bbx]_{2t} = \cone\{p\qo{p}  ~|~ p \in \cxbx_t \}=\Sigma[\bx,\bbx]\cap \C[\bx,\bbx]_{2t}$ (or simply $\Sigma_{2t}$) denote the cone of Hermitian sums of squares with degree at most $2t$.

\paragraph{The dual space of polynomials.}
The algebraic dual of the ring of polynomials $\Ccxbx$ is the vector space of all linear functionals on $\Ccxbx$. To clarify, a linear functional $L$ on $\Ccxbx$ is a linear map from $\Ccxbx$ to $\C$. For every $t\in \N \cup \{\infty \}$ we denote the dual space of $\Ccxbx_t$ by $\Ccxbxs_t$, defined as 
\[
\Ccxbxs_t = \{ L:\Ccxbx_t \rightarrow \C: L~\text{is~linear}\}.
\]
We again abbreviate $\Ccxbxs_{\infty}$ by $\Ccxbxs$. A linear functional $L \in \Ccxbxs_t$ is called \emph{Hermitian} if $L(\overline p) = \overline{L(p)}$ for all $p \in \Ccxbx_t$. A (Hermitian) linear functional $L\in \C[\bx,\bbx]_{2t}^*$ is called \emph{positive} if it maps Hermitian squares to nonnegative real numbers, i.e., if $L(p\overline p) \geq 0$ for all $p \in \Ccxbx_t$.

\paragraph{Example of linear functionals.} For any $a \in \C^{n}$  we can define the \emph{evaluation functional at $a$}, denoted $L_{a} \in \Ccxbxs$, by
\[
L_{a}(p) = p(a)  ~\text{for~every}~p \in C\cxbx.
\]
It is easy to see that $L_a$ is Hermitian and positive. 

\paragraph{Linear functionals applied to polynomial matrices.} 
It will also be useful to apply linear functionals to polynomial matrices, i.e., matrices whose entries are polynomials, by considering an entrywise action.  
That is, for a polynomial matrix  $G=(G_{ij})_{i,j=1}^m \in \Ccxbx^{m \times m}$ and a linear functional $L \in \Ccxbxs$   we define 
\[
L(G) :=  \Big( L(G_{ij}) \Big)_{i,j \in [m]}\in \C^{m\times m}.
\]

\paragraph{Moment matrices.}\label{MomentMat}
As an example, applying a linear functional to the (infinite) matrix $\cxbx \cxbx^*$ leads to the notion of  moment matrix.
Given $L \in \Ccxbxs_{2t}$,  where $t \in \N \cup \{\infty\}$, we define the {\em moment matrix} of $L$  by
\begin{equation}\label{eqMtL}
	M_t(L) := L( \cxbx_t \cxbx^*_t) =  \big( L(m \overline{m'}) \big)_{m,m' \in \cxbx_t } .
\end{equation}
If $t$ is finite then the moment matrix is said to be  {\em truncated} at {\em order} $t$. Note that $L$ 
is Hermitian if and only if its moment matrix $M_t(L)$  is Hermitian. Similarly, $L$
is positive if and only if its moment matrix $M_t(L)$ is positive semidefinite:
\begin{equation}\label{eq_pos_L_ML}
L(p\overline p)\ge 0 \ \forall p\in \C[\bx,\bbx]\ \Longleftrightarrow \ M_t(L)\succeq 0.
\end{equation}
 Indeed,  for any $p\in \C[\bx,\bbx]_t$, written as $p = \ba^* \cxbx_t \in \Ccxbx$ with $\ba\in \C^{\N^n\times \N^n}_0$,  
we have $\overline p=[\bx,\bbx]_t^*\ba$ and thus 
\begin{equation}\label{eq_pos_herm}
	L( p \overline p) = L(\ba^*\cxbx_t \cxbx_t^* \ba) = \ba^*L(\cxbx_t \cxbx^*_t)\ba = \ba^*M_t(L)\ba.
\end{equation}
More generally, if $p=\ba^*[\bx,\bbx]_t$ and $q=\bb^*[\bx,\bbx]_t$ with $\ba,\bb\in  \C^{\N^n\times \N^n}_0$, then 
$L(p\overline q)=\ba^*M_t(L)\bb.$
 If $t = \infty$ we write $M(L)$ instead of $ M_{\infty}(L)$. 
 
 Observe that the moment matrix of an evaluation functional $L_a$  at $a\in \C^d$ satisfies $M(L_a)=[a,\overline a]_t[a,\overline a]_t^*$ and thus it has rank 1. Hence,  if $L$ is a linear combination of evaluation functionals, then its moment matrix has finite rank.

\paragraph{Polynomial localizing maps $gL$.}
 Given a polynomial $g \in \Ccxbx$ and a linear functional $L \in \Ccxbxs$ we can define a new linear functional $gL \in \Ccxbxs$ by 
\begin{align*}
gL: \Ccxbx &\to \C \\
p &\mapsto L(gp).
\end{align*}
In this way, we can say that $g$ acts on $\Ccxbxs$ by mapping $L$ to $gL$. Constraints are often phrased in terms of the positivity of $gL$. As stated before, positivity of $gL$ can be characterized by positive semidefiniteness of its moment matrix:
\begin{equation}\label{eqgLpos}
gL \ \text{ is positive } \Longleftrightarrow \  L(g \cdot \cxbx \cxbx^* ) = M(gL) \succeq 0.
\end{equation}
If both $g$ and $L$ are Hermitian then $gL$ is Hermitian and hence $M(gL)$ is Hermitian. If $L$ is an evaluation map at a point $a\in\C^n$ for which $g(a) \geq 0$, then $gL$ is a positive map since we have $(gL)(p\overline p) = g(a)|p(a)|^2\ge 0$. In the literature $M(gL)$ is often called a {\em localizing moment matrix}.

\subsection{SoS-polynomial matrices and matrix-valued linear maps}\label{sec:matrix-setting}

There is a natural extension of the previously defined concepts to the matrix-valued setting. This extension will be useful, in particular, to define a matrix analog of localizing moment constraints and to provide a moment approach to the hierarchy by Doherty, Parrilo and Spedalieri \cite{DPS04}.

\paragraph{SoS-polynomial matrices.} 
A polynomial matrix $S\in \C[\bx,\bbx]^{m\times m}$ is called  an {\em SoS-polynomial matrix} if $S=U U^*$ for some polynomial matrix $U\in \C[\bx,\bbx]^{m\times k}$ and some integer $k\in \N$, or, equivalently, if 
$S\in \cone\{ \avec{p}\avec{p}^*: \avec{p} = (p_1,\ldots, p_m) \in \C\cxbx^m\}$.

\paragraph{Matrix-valued linear functionals.}
Consider a matrix-valued linear functional 
\begin{align*}
 \calL: \Ccxbx &\to \C^{m \times m} \\
p &\mapsto \calL(p)=\big( L_{ij}(p) \big)_{i,j \in [m]},
\end{align*}
where  $\calL=(L_{ij})_{i,j=1}^m$ and each $L_{ij}\in \C[\bx,\bbx]^*$ is a scalar-valued linear functional.
Then  $\calL$ is {\em  Hermitian} if $\calL(\overline p)=\calL(p)^*$, i.e., $L_{ij}(\overline p) =\overline {L_{ji}(p)}$ for all $i,j\in [m]$, for all $p\in \C[\bx,\bbx]$.
In  addition $\calL$ is said to be {\em positive} if it maps positive elements 
(i.e., Hermitian squares $p\overline p $) to positive elements (i.e., Hermitian positive semidefinite $m\times m$ matrices), i.e., if the following holds:
\begin{equation}\label{eqGLpositive}
\calL(p\overline p) = (L_{ij}(p\overline p))_{i,j=1}^m  
\succeq 0 \text{ for all } p \in \Ccxbx.
\end{equation}
In analogy to \cref{eqMtL} it is natural to define the {\em moment matrix} $M(\calL)$ as
\begin{equation}\label{eqMML}
M( \calL) := \calL(  \cxbx \cxbx^*)= (L_{ij}( \cxbx \cxbx^*))_{i,j=1}^m=(M(L_{ij}))_{i,j=1}^m,
\end{equation}
which thus acts on $\C^m\otimes \C[\bx,\bbx]$. Clearly, $M(\calL)$ is a Hermitian matrix if $\calL$ is Hermitian.
Note that $M(\calL)$ can be viewed as an $m\times m$ block-matrix whose $(i,j)$th block is the moment matrix $M(L_{ij})$.

When $\calL$ acts on a truncated polynomial space $\C[\bx,\bbx]_{2t}$  its moment matrix $M_t(\calL)$, {\em truncated at order $t$}, is defined in the obvious way by
$$M_t( \calL) := \calL(  \cxbx_t \cxbx_t^*)=(M_t(L_{ij}))_{i,j=1}^m,$$
with $M_t(\calL)=M(\calL)$ if $t=\infty$.

One may also  define the action of $\calL$ on a polynomial matrix $S=(S_{ij})_{i,j=1}^m \in \C[\bx,\bbx]^{m\times m}$ by
\begin{equation}\label{eq:LS}
\langle \calL, S\rangle:= \sum_{i,j=1}^m L_{ij}(S_{ij}).
\end{equation}
If $\calL$ and $S$ are both  Hermitian then $\langle \calL,S\rangle\in \R$.  As before,  given $g\in \C[\bx,\bbx]$ we may define a new (localizing) matrix-valued linear map $g\calL$ by:
\begin{align*}
 g\calL: \Ccxbx &\to \C^{m \times m} \\
p &\mapsto (g\calL)(p)=\calL(gp)=\big( L_{ij}(gp) \big)_{i,j \in [m]}.
\end{align*}

\paragraph{Positivity of  $\calL$ and its moment matrix $M(\calL)$.} The analog of \cref{eqgLpos} does not extend to the matrix-valued  case: If $M(\calL)$ is positive semidefinite, then   $\calL$ is positive, but the reverse implication may not hold in general. In the next two lemmas, we present alternative characterizations for positivity of a matrix-valued map $\calL$ and positivity of its moment matrix $M(\calL)$ that make this more apparent.

\begin{lemma}\label{lemGLpos}
$\calL$ is positive, i.e., \cref{eqGLpositive} holds, if and only if 
any of the following equivalent conditions holds:
\begin{align} 
v^* \calL(p\overline p)v=  \big(\sum_{i,j=1}^m \overline {v_i} v_j L_{ij}\big)(p\overline p) =(v^*\calL v)(p\overline p)
\geq 0 \text{ for all } v \in \C^m \text{ and }  p \in \Ccxbx,
\label{eq: positivity of GL as lin func} \\
M(v^*\calL v)\succeq 0 \ \text{ for all } v\in \C^m,
\label{eqGL3}\\
(v \otimes \ba)^* \, M(\calL) \, (v \otimes \ba) \geq 0 \text{ for all } v \in \C^m \text{ and } \ba
\in\C^{\N^n\times \N^n}_0.
\label{eq: positivity of ML}
\end{align}
\end{lemma}

\begin{proof}
The equivalence of \cref{eqGLpositive} and \cref{eq: positivity of GL as lin func} is clear. The equivalence of \cref{eq: positivity of GL as lin func} and \cref{eqGL3} follows using  \cref{eq_pos_L_ML}  
applied to each   (scalar-valued) map $v^*\calL v$. To see the equivalence of 
\cref{eq: positivity of GL as lin func} and \cref{eq: positivity of ML}, write  a polynomial $p\in \C[\bx,\bbx]$ as 
$p=\ba^*[\bx,\bbx]$ with  $\ba=(a_{\alpha,\beta})\in\C^{\N^n\times \N^n}_0$. 
Then, for any $v\in \C^m$, following \cref{eq_pos_herm},  we have:
$$v^*\calL (p\overline p)v=v^*( L_{ij}(p\overline p))_{i,j=1}^m v= v^*(  \ba^*  M(L_{ij}) \ba)_{i,j=1}^m v=
 (v\otimes \ba)^* M(\calL) v\otimes \ba,$$
using the definition of $M(\calL)$ from \cref{eqMML}.
\end{proof}

\begin{lemma}\label{lemMGLpos} 
$M(\calL)\succeq 0$ if and only if any of the following equivalent conditions holds:
\begin{align}
w^* M(\calL) w \geq 0 \text{ for all } w \in \C^m \otimes  \C^{\N^n\times \N^n}_0,
\label{eqMGpsd}\\
\langle \calL, \avec{p}\avec{p}^*\rangle=  \sum_{i,j=1}^m L_{ij}(p_i \overline{p}_j)
     \geq 0 \text{ for all } \avec{p} = (p_1,\ldots, p_m) \in \C\cxbx^m,
      \label{eq: M(ML) psd} 	\\
     \langle \calL, S\rangle \ge 0 \text{ for all SoS-polynomial matrices } S\in\C[\bx,\bbx]^{m\times m}.
      \label{eq: MMLSoS}
  \end{align}
  	\end{lemma}

\begin{proof}
\cref{eqMGpsd} is clear. To see the equivalence with \cref{eq: M(ML) psd} 
consider a vector $w=(w_{i, (\alpha,\beta)})_{i, (\alpha,\beta)}$ in $ \C^m\otimes \C^{\N^n\times \N^n}_0$ and, for each $i\in [m]$, define the vector 
$\ba_i=(w_{i, (\alpha,\beta)})_{(\alpha,\beta)}\in \C^{\N^n\times \N^n}_0$, the corresponding polynomial $p_i= \ba_i^* [\bx,\bbx]$, and define the polynomial vector $\avec{p}=(p_1,\ldots,p_m)\in \C[\bx,\bbx]^m.$ Then 
$$
w^* M(\calL) w=w^*(M(L_{ij}))_{i,j=1}^m w=
\sum_{i,j=1}^m \ba_i^* (L_{ij}([\bx,\bbx][\bx,\bbx]^*))_{i,j=1}^m \ba_j 
= \sum_{i,j=1}^m L_{ij}(p_i \overline{p}_j),
$$
implying  the equivalence of \cref{eqMGpsd} and \cref{eq: M(ML) psd}.
The equivalence with \cref{eq: MMLSoS}  follows since SoS-polynomial matrices are conic combinations of terms of the form $ \avec{p}\avec{p}^*$. 
\end{proof}

\noindent
Note that \cref{eq: positivity of ML} is the restriction of \cref{eqMGpsd}, where we restrict to vectors $w$ in tensor product form $w=v\otimes \ba$.
In addition, we recover 
\cref{eq: positivity of GL as lin func} if, in \cref{eq: M(ML) psd}, we restrict to polynomials $p_1,\ldots,p_m$  of the form  $p_i = v_i p$ (for $i \in [m]$) for some $p\in\C[\bx,\bbx]$ and  $v =(v_1,\ldots,v_m)\in \C^m$. This shows  again that \cref{eq: positivity of GL as lin func} is more restrictive than \cref{eq: M(ML) psd}. 
Summarizing, we have  the following implication.

\begin{lemma}\label{lemMGLtoGL}
If $M(\calL)\succeq 0$ then $\calL$ is positive. 
\end{lemma}

\begin{remark}\label{remcomplexity}
Note that requiring positivity of the moment matrix $M(\calL)$ not only provides a stronger condition than requiring positivity of $\calL$, but it is also a condition that is computationally easier to check. To make this concrete we consider the truncated case when  $\calL$ is restricted to the subspace $ \C[\bx,\bbx]_{2t}$. 
Then, the condition $M_t(\calL)\succeq 0$ asks whether a single matrix is positive semidefinite, which can be efficiently done. On the other hand, asking whether $\calL$ is positive on sums of squares of degree at most $2t$ amounts to checking whether  $M_t(v^*\calL v) \succeq 0$ for all $v\in \C^m$, i.e.,   positive semidefiniteness of infinitely many matrices.

Note also that \cref{eq: MMLSoS} highlights  the duality relationship which exists between $m\times m$ SoS-polynomial matrices and matrix-valued linear maps $\calL$ with $M(\calL)\succeq 0$. 
\end{remark}
 
\paragraph{Link to complete positivity of $\calL$.}
We now point out a link to the notion of complete positivity. Given a linear map $\calL:\C[\bx,\bbx]\to \C^{m\times m}$ and an integer $k\in \N$ one can define a new linear map
\begin{align*}
I_k\otimes \calL:  \C[\bx,\bbx]^{k\times k} & \to \C^{k\times k}\otimes \C^{m\times m} \\
 (p_{i' j'})_{i',j'=1}^k  &\mapsto  (\calL(p_{i'j'}))_{i',j'=1}^k.
\end{align*}
Then $\calL$ is said to be {\em completely positive} if $I_k\otimes \calL$ is positive for all $k\in \N$. (See, e.g.,  \cite{paulsen_2003} for a general reference about completely positive maps.)

\begin{lemma} \label{lemCPL}
$\calL$ completely positive $\Longrightarrow$ $I_m\otimes \calL$ positive $\Longrightarrow$ $M(\calL)\succeq 0$.
\end{lemma}

\begin{proof}
The first implication is obvious. Assume $I_m\otimes \calL$ is positive, we show that $M(\calL)\succeq 0$. In view of \cref{eq: M(ML) psd} it suffices to show that $\sum_{i,j=1}^m L_{ij}(p_i\overline {p}_j)\ge 0$ for all $\avec{p}=(p_1,\ldots,p_m)\in \C[\bx,\bbx]^m$. As $\avec{p}\avec{p}^*$ is  a SoS-polynomial  matrix (and thus a positive element), it follows that 
$(I_m\otimes \calL)(\avec{p}\avec{p}^*)=  (\calL(p_{i'}\overline{p}_{j'}))_{i',j'=1}^m \succeq 0$. Consider the vector $w=(w_{ii'})_{i,i'\in [m]}$ with entries $w_{ii'}=1$ if $i=i'$ and $w_{ii'}=0$ otherwise.
Then, 
$ \sum_{i,j=1}^m L_{ij}(p_i \overline{p}_j) =w^*  (\calL(p_{i'}\overline{p}_{j'}))_{i',j'=1}^m w \ge 0$, as desired.
\end{proof}

\paragraph{Polynomial matrix localizing maps $G\otimes L$.}
Given a (scalar-valued)  linear map $L\in \C[\bx,\bbx]^*$ there is a   natural generalization of the above  notion of localizing map $gL$, where, instead of considering a scalar  polynomial $g$, we consider a polynomial matrix $G=(G_{ij})_{i,j=1}^m \in \C[\bx,\bbx]^{m\times m}$. Then, we can define  the matrix-valued linear map 
$\calL:=  (G_{ij}L)_{i,j=1}^m$, that we denote by  $G\otimes L$, by
\begin{align*}
G\otimes L:  \C[\bx,\bbx] & \to  \C^{m\times m} \\
p &\mapsto  (G\otimes L)(p):=\big( (G_{ij} L)(p) \big)_{i,j=1}^m= \big(L(G_{ij}p)\big)_{i,j=1}^m= L(Gp).
\end{align*}
Following \cref{eqMML}   the {moment matrix} of $G \otimes L$ is
\begin{equation}\label{eqMGL}
M(G\otimes L) = (G\otimes L) (\cxbx \cxbx^*)=((G_{ij}L)(\cxbx \cxbx^*))_{i,j=1}^m= L( G \otimes \cxbx \cxbx^*).
\end{equation}

\begin{remark} \label{rem: M(GL) product and psd}
	When $L=L_a$ is the (scalar-valued) evaluation map at a vector $a \in \C^n$ the moment matrix $M(G\otimes L_a)$ has indeed a tensor product structure, since we have
	 	\[
	M(G \otimes L_a) = L_a(G \otimes \cxbx \cxbx^*) = G(a) \otimes [a,\overline a][a,\overline a]^* = L_a(G) \otimes L_a(\cxbx \cxbx^*).
	\]
	In particular, if $G(a) \succeq 0$ then we have $M(G \otimes L_a) \succeq 0$. Therefore, $M(G \otimes L) \succeq 0$ when $L$ is a conic combination of evaluation maps at points at which $G$ is positive semidefinite. This property motivates using such a positivity constraint in defining our bounds for the separable rank and the completely positive rank.
\end{remark}

As observed above, $M(G\otimes L)\succeq 0$ implies that $G\otimes L$ is positive. Note that, by \cref{eqGL3}, $G\otimes L$ is positive if and only if $M((v^*Gv)  L)\succeq 0$ for all $v\in\C^m$, while, by \cref{eq: M(ML) psd}, $M(G\otimes L)\succeq 0$ if and only if $L(\avec{p}^*G \avec{p})\ge 0$ for all $\avec{p}\in \C[\bx,\bbx]^m$. In particular, for a truncated linear map $L\in \C[\bx,\bbx]^*_{2t}$,  the condition $M_t(G\otimes L)\succeq 0$ implies any of the following two equivalent conditions (the truncated analogs of \eqref{eq: positivity of GL as lin func} and \eqref{eqGL3}), which characterize positivity of $G\otimes L$ on $\Sigma_{2t}$:
\begin{align}
L(v^*Gv \cdot p\overline p)\ge 0 \ \text{ for all } v\in \C^m \text{ and } p\in \C[\bx,\bbx]_t,\label{eqMtGLa}\\
M_t((v^*Gv) L)\succeq 0\ \text{ for all } v\in \C^m. \label{eqMtGL}
\end{align}
 While it is computationally easy to check whether $M_t(G\otimes L)\succeq 0$, it is not clear how to  check the above conditions efficiently.
  For this reason, we will select the stronger moment matrix positivity condition when defining our new hierarchy of bounds for the separable rank. However, we note  that the weaker positivity condition of the localizing map will be sufficient to establish convergence properties of the bounds.

\subsection{The moment method}\label{sec:moment}
We now  state several widely used definitions and results from polynomial optimization that we will need to design our hierarchy of bounds on the separable rank and for the moment approach to the DPS approximation hierarchy of the set $\SEP$ of separable states.

	Given a set of Hermitian polynomials  $S \subseteq \Ccxbxh$ we define the \emph{positivity domain} of $S$ as
	\begin{equation}\label{poddom}
		\sD(S) := \{u \in \C^{n} ~|~ g(u) \geq 0 ~\text{for~every}~ g \in S\}.
	\end{equation}
Given a Hermitian polynomial matrix $G\in \C[\bx,\bbx]^{m\times m}$ we define 
	 the polynomial set
	\begin{equation}\label{eqSG}
	S_G:=\{v^* Gv: v\in \C^d, \|v\|=1\}\subseteq \C[\bx,\bbx]^h,
	\end{equation}
so that  the set 
\begin{equation}\label{posdomG}
		\sD(S_G) = \{u \in \C^{n} ~|~ G(u) \succeq  0\}
	\end{equation}
	corresponds to the positivity domain of $G$. 
		For $t \in \N \cup \{\infty\}$ and $S \subseteq \Ccxbxh$  the set 
	$$
	\calM(S)_{2t} := \cone\{gp\overline p ~|~ p \in \Ccxbx,~ g \in S \cup \{1\},~\deg(gp\overline p) \leq 2t\}
	$$
	denotes the \emph{quadratic module} generated by $S$, {\em truncated at order $2t$} when $t\in \N$. If $t = \infty$ we simply write $\calM (S)$. The quadratic module $\calM(S)$ is said to be {\em Archimedean} if, for some scalar $R > 0$, 
	\begin{equation}\label{ArchimedianCond}
	R - \sum_{i = 1}^{n}x_i \overline{x_i} \in \calM(S).
	\end{equation}
	Hence a quadratic module is Archimedean if it contains an algebraic certificate of boundedness of the associated positivity domain. The next lemma shows that, in the case when the algebraic certificate  in (\ref{ArchimedianCond}) belongs to  the quadratic module $\calM(S)_{2}$, 
	 the  linear functionals that are nonnegative on $\calM(S)$  are bounded. Its proof is standard (and easy) and thus omitted. 

\begin{lemma}\label{pointconv} 
	Let $S \subseteq \Ccxbxh $ be such that $ R  - \sum_{i = 1}^{n} x_i \overline{x_i} \in  \calM(S)_2$ for some $R> 0$. For any $t \in \N$ assume $L_t \in \Ccxbx_{2t}^*$ is nonnegative on $\calM(S)_{2t}$.  Then we have 
	$$
	|L_t(w)| \leq R^{|w|/2}L_t(1)~\text{for~all}~ w \in  \cxbx_{2t}.
	$$
	Moreover, if 
	\begin{equation}\label{Ltsup}
	\sup_{t\in \N} L_t(1) < \infty,
	\end{equation}
	then $\{L_t\}_{t \in \N}$ has a point-wise converging subsequence in $ \Ccxbxs$.
\end{lemma}	

\paragraph{Linear functionals and measures.}

The following result is central to our approach for approximating matrix factorization ranks. It is a complex analog 
of  results by Putinar \cite{Pu93} and Tchakaloff~\cite{Tchakaloff}. 
For completeness,  we will indicate in \cref{AppProofPutinar}  
how to derive from these results the following complex analog.
\begin{theorem}
\label{theomainTchakaloff}
	Let $S \subseteq  \Ccxbxh$ be a set of Hermitian polynomials such that the quadratic module $\calM(S)$ is Archimedean and consider a Hermitian linear map $L: \C[\bx,\bbx]\to \C$. Assume  that $L$
	is  nonnegative on $\calM(S)$. Then the  following holds.
\begin{itemize}
\item[(i)] (based on \cite{Pu93}) $L$ has a representing measure $\mu$ that is supported by $\sD(S)$, i.e., we have
$L(p)=\int_{\sD(S)}p d\mu$ for all $p\in\C[\bx,\bbx]$.
\item[(ii)] (based on \cite{Tchakaloff}) For any  integer $k \in \N$, there exists a linear functional $\haL: \C[\bx,\bbx]\to \C$ which coincides with $L$ on $\C[\bx,\bbx]_k$ and has a finite atomic representing measure supported by $\sD(S)$, i.e.,	
	we have
	\begin{align}	
	\haL(p) = L(p) \text{ for every } p \in \Ccxbx_k, \label{eqext}\\
	\haL =  \sum_{\ell=1}^{K} \lambda_\ell L_{v_\ell},\label{eqatom}
			\end{align}
	for some  integer $K\ge 1$, scalars  $\lambda_1,\lambda_2,...,\lambda_K>0$ and vectors $v_1,v_2,...,v_K \in \sD(S)$.
\end{itemize}

\end{theorem}
We will often apply the above theorem to a linear functional $L \in \C[\bx,\bbx]^*$ that additionally satisfies the positivity condition: $(G \otimes L) (p\overline p) \succeq 0$ for all $p \in \C[\bx,\bbx]$, for some Hermitian  polynomial matrix $G \in \C[\bx,\bbx]^{m \times m}$. Then, in view of \cref{lemGLpos} (combined with  \cref{eqMtGLa} and \cref{eqMtGL}), one may still apply Theorem \ref{theomainTchakaloff} after replacing the set $S$ by the set $S \cup S_{G}$ so that  the resulting measure $\mu$ will be supported by $\sD(S\cup S_G)\subseteq \{x: G(x)\succeq 0\}$, thus within the positivity domain  of $G$.

\paragraph{Matrix-valued linear functionals and matrix-valued measures.}

We now mention extensions of the previous results in Theorem \ref{theomainTchakaloff} from the scalar-valued  case to the matrix-valued case, that we will use for  the moment approach to the DPS hierarchy.

For the next result we use (a specification of) a result of Cimpric and Zalar \cite[Theorem~5]{CimpricZalar}, which shows an operator-valued version of Theorem \ref{theomainTchakaloff} (i).  Since the latter is stated in the real case we indicate in Appendix \ref{AppProofCimpric} how to derive from it its complex analog that we need for  the implication (ii) $\Rightarrow$ (i) in Theorem \ref{theomainmatrix} below. 
In a nutshell, this implication relies on a version of Riesz' representation theorem for positive operator valued linear maps (see, e.g., \cite{hadwin81}) combined with a density argument (for going from polynomials to continuous functions) and Putinar's Positivstellensatz.

\begin{theorem}[based on \cite{CimpricZalar}]\label{theomainmatrix}
Let $S\subseteq \C[\bx,\bbx]^h$ be a set of Hermitian polynomials such that the quadratic module $\calM(S)$ is Archimedean and let $\calL:\C[\bx,\bbx]\to \H^m$ be a Hermitian matrix-valued linear map. The following assertions are equivalent.
\begin{itemize}
\item[(i)] $\calL$ has a representing measure $\mu$ that is supported by $\sD(S)$ and takes its values in the cone $\H^m_+$ of $m\times m$ Hermitian positive semidefinite matrices.
\item[(ii)] $\calL$ is nonnegative on $\calM(S)$, i.e.,  $\calL(gp\overline p)\succeq 0$ for all $g\in S\cup\{1\}$ and $p\in \C[\bx,\bbx]$.
\item[(iii)] $M(g\calL)\succeq 0$ for all $g\in S\cup\{1\}.$
\item[(iv)] $g\calL$ is completely positive for all $g\in S\cup\{1\}.$
\end{itemize}
\end{theorem}

\begin{proof}
First we show that (i) implies (iv). 
Let $k\in \N$, let $P=(p_{i'j'})_{i',j'=1}^k\in \C[\bx,\bbx]^{k\times k}$ be a polynomial matrix such that $P(x)\succeq 0$ for all $x\in \C^d$, and let $g\in S\cup \{1\}$;
we show that $(I_k\otimes g\calL)(P)\succeq 0$.
For this 
note that
$$(I_k\otimes g\calL)(P)=
(g\calL)(p_{i'j'}))_{i',j'=1}^k
= (\calL(gp_{i'j'}))_{i',j'=1}^k
= \Big(\int_{\sD(S)}  g p_{i'j'} d\mu\Big)_{i',j'=1}^k  \succeq 0.
$$
Here, the last inequality follows (for example) from \cref{theoTchakmatrix} below, using the fact that $g(x)\ge 0$ on $\sD(S)$, $P(x)=(p_{i'j'}(x))_{i',j'=1}^k \succeq 0$ for all $x$, and $\mu$ takes its values in $\H^m_+$. Indeed, say $D$ is an upper bound on the degree of $g(x)P(x)$. Then, by Theorem \ref{theoTchakmatrix} applied to $\calL$ restricted to $\C[\bx,\bbx]_D$, there exist an integer $K\in \N$, matrices $\Lambda_\ell\succeq 0$ and vectors $v_\ell\in \sD(S)$ (for $\ell\in [K])$) such that 
$ (\calL(gp_{i'j'}))_{i',j'=1}^k=\sum_{\ell=1}^K g(v_\ell) \Lambda_\ell \otimes P(v_\ell),$ which proves it is a positive semidefinite matrix.
The implication (iv) $\Longrightarrow$ (iii) follows from Lemma \ref{lemCPL} and (iii) $\Longrightarrow$ (ii) follows from Lemma \ref{lemMGLtoGL}.

Finally, for the implication (ii) $\Longrightarrow$ (i) we refer to the arguments in Appendix \ref{AppProofCimpric}.
\end{proof}

What the above result shows is that, while in general the notions of complete positivity, positivity and having a positive semidefinite moment matrix are not equivalent, these properties become equivalent when considering a linear map $\calL$ acting on an Archimedean quadratic module. We will apply these results  to the case of the quadratic module of the unit sphere (with $S=\{1-\sum_i x_i\overline{x_i}\}$) for the moment approach to the DPS hierarchy in Section \ref{sec:momentDPS}.

Finally, there is also an  analog of 
Theorem \ref{theomainTchakaloff} (ii) for the matrix-valued case.

\begin{theorem}[Kimsey \cite{Kimsey}] \label{theoTchakmatrix}
Let $S\subseteq \C[\bx,\bbx]^h$ be a set of Hermitian polynomials and let $\calL:\C[\bx,\bbx]\to \H^m$ be a Hermitian matrix-valued linear map. Assume $\calL$ has a representing measure supported by $\sD(S)$ and taking values in the cone $\H^m_+$. Then, for any integer $k\in N$, the restriction of $\calL$  to $\C[\bx,\bbx]_k$  has another representing measure that is finitely atomic; that is, there exists $K\in \N$, matrices $\Lambda_1,\ldots,\Lambda_K\in \H^m_+$ and vectors $v_1,\ldots,v_K\in \sD(S)$ such that 
$\calL(p)=\sum_{\ell=1}^K \Lambda_\ell p(v_\ell)$ for all polynomials $p\in \C[\bx,\bbx]_k$.
\end{theorem}

\section{A hierarchy of lower bounds on the separable rank} \label{sec: LB}

In this section, we show how to use the polynomial optimization techniques developed in the previous section in order  to obtain a hierarchy of lower bounds on the separable rank. 

\subsection{The parameter $\tsep$}\label{sec:tausep}

Consider  a separable state  $\rho \in \SEP_d$. As defined earlier,  its separable rank is the smallest integer  $r\in\N$ for which there exist (nonzero)  vectors $a_1,\ldots, a_r, b_1,\ldots,b_r \in \C^d$  such that 
\begin{equation}\label{eqseprank}
\rho = \sum_{\ell=1}^r a_\ell a_\ell^* \otimes b_\ell b_\ell^*.
\end{equation}
We mention several properties that are satisfied by the vectors $a_\ell,b_\ell$ entering such a decomposition.
First of all, the vectors $a_\ell,b_\ell$ clearly satisfy the positivity condition
\begin{equation}\label{eqabpsd}
\rho- a_\ell a_\ell^*\otimes b_\ell b_\ell^*\succeq 0 \quad \text{ for all }\ell\in [r].
\end{equation}
Let
$$\rho_{\max}:=\max_{i,j\in [d]} \rho_{ij,ij}$$
denote the maximum diagonal entry of $\rho$. 
Then, in view of (\ref{eqabpsd}), the vectors $a_\ell,b_\ell$ also satisfy $|(a_\ell)_i|^2 |(b_\ell)_j|^2 \le \rho_{ij,ij}$ for all $i,j\in [d]$,
which implies the following boundedness conditions
\begin{equation}\label{eqbound}
\|a_\ell \|_\infty^2 \cdot \|b_\ell\|^2_\infty \le \rho_{\max} \quad \text{ and }
\quad 
\|a_\ell \|_2^2 \cdot \|b_\ell \|^2_2\le \Tr(\rho)\quad \text{ for all } \ell\in [r].
\end{equation}
Note that we may rescale the vectors $a_\ell,b_\ell$ so that additional properties can be assumed. For instance we may rescale them so that $\|a_\ell\|_\infty=\|b_\ell\|_\infty$, in which case we may assume without loss of generality that 
\begin{equation}\label{eqabmax}
\|a_\ell\|^2_\infty, \|b_\ell\|^2_\infty \le \sqrt{\rho_{\max}} \quad \text{ for all } \ell\in [r].
\end{equation}
Another possibility is  rescaling so that $\|a_\ell\|_2=\|b_\ell\|_2$, in which case we could instead assume  that
\begin{equation}\label{eqab2}
\|a_\ell\|_2^2 = \|b_\ell\|_2^2 \le \sqrt {\Tr(\rho)} \quad \text{ for all } \ell\in [r].
\end{equation}
Yet another possibility would be to rescale  so that $\|b_\ell\|_2=\sqrt{\Tr(\rho)}$ for all $\ell$, in which case we would have
\begin{equation}\label{eqbnorm1a}
 \|a_\ell\|_2^2 \le \sqrt{\Tr(\rho)}, \  \|b_\ell\|_2=\sqrt{\Tr(\rho)}\ \text{ for all } \ell\in [r]
\end{equation}
or, equivalently (up to rescaling), we may assume that 
\begin{equation}\label{eqbnorm1}
 \|a_\ell\|_2^2 \le \Tr(\rho), \ \|b_\ell\|_2=1\ \ \text{ for all } \ell\in [r].
\end{equation}
To fix ideas we will now  apply the first rescaling (\ref{eqabmax}), so that  each $(a_\ell,b_\ell)$ belongs to the set 
\begin{equation} \label{eq:set}
\calV_\rho:= \Big\{ (x,y) \in \C^d \times \C^d ~|~ xx^* \otimes yy^* \preceq \rho, \ \|x\|_\infty, \|y\|_\infty \leq \rho_{\max}^{1/4} \Big\}.
\end{equation}
We will consider the impact of doing other rescalings  as in \cref{eqab2}, \cref{eqbnorm1a} or \cref{eqbnorm1} later on in the paper in numerical examples. 
However, as will be noted in Remark \ref{rembound}, the localizing constraints corresponding to the  
scaling (\ref{eqabmax})
already imply the localizing constraints corresponding to the inequalities in~(\ref{eqbound}). 

From \cref{eqseprank} we  have
$${1\over r}\rho ={1\over r}\sum_{\ell=1}^r a_\ell a_\ell^*\otimes b_\ell b_\ell^* \in \conv\{xx^*\otimes yy^*: (x,y)\in \calV_\rho\},
$$
which motivates defining the following parameter
\begin{equation} \label{eq:tausep}
\tausep(\rho) := \inf\Big\{\lambda: \lambda>0,  {1\over \lambda} \rho \in \conv \{xx^*\otimes yy^*: (x,y)\in \calV_\rho\}\Big \}.
\end{equation}
From the above discussion, this parameter gives a lower bound on the separable rank.

\begin{lemma}\label{tauleqsep}
	For any  $\rho \in \SEP_d$, we have $\tsep(\rho) \leq \seprank(\rho) $. Moreover, if $\rho\not\in\SEP_d$ then $\tsep(\rho)=\seprank(\rho)=\infty$.
\end{lemma}
The parameter $\tsep(\rho)$ does not seem any easier to compute than the separable rank.
It, however, enjoys an additional  convexity property that the  combinatorial parameter $\seprank(\rho)$ does not have.
In the next section, we will present a hierarchy  of lower bounds on $\seprank(\rho)$, constructed using  tools from polynomial optimization.
These bounds arise from convex (semidefinite) programs, they in fact also lower bound the (weaker) parameter $\tsep(\rho)$ and will be shown to asymptotically converge to it.

\subsection{Polynomial optimization approach  for $\tausep$ and $\seprank$} \label{sec: poly sep}

As above, let $\rho\in \SEP_d$ be given, together with a decomposition (\ref{eqseprank}) with $r=\seprank(\rho)$, where we assume that the points $(a_\ell,b_\ell)$ belong to the set $\calV_\rho$ in (\ref{eq:set}).
We explain how to define bounds for $\seprank(\rho)$ by using the moment method from Section \ref{sec:moment}.

For this let us consider the linear functional
\begin{equation} \label{eq:defL}
L = \sum_{\ell=1}^r L_{(a_\ell,b_\ell)},
\end{equation}
the sum of the evaluation functionals at the points entering the decomposition (\ref{eqseprank}). Then $L$  acts on the polynomial space $\C[\bx,\by,\bbx,\bby]$, where it is now convenient to denote the $2d$  variables as  $\bx=(x_1,\ldots,x_d)$ and $\by=(y_1,\ldots,y_d)$, 
corresponding to the `bipartite' structure in \cref{eqseprank}.
By construction,  $L$ corresponds to a finite atomic measure supported on the set  $\calV_\rho$. Moreover we have 
$$L(1) = \sum_{\ell=1}^r L_{(a_\ell,b_\ell)}(1) = \sum_{\ell=1}^r 1 = r=\seprank(\rho)$$ and the fourth-degree moments are given by the entries of $\rho$:
$$L(\bx\bx^* \otimes \by\by^*)=\rho.
$$
In addition, since each $(a_\ell,b_\ell)$ belongs to the set $\calV_\rho$,  it follows that
$$M(G_\rho \otimes L)=L(G_\rho\otimes \monoV \monoV^*) \succeq 0 \quad \text{ and } \quad L \ge 0 \text{ on } \calM(S_\rho),$$
after defining  the  Hermitian polynomial matrix
\begin{equation}\label{eqGrho}
\Gr(\bx,\by) := \rho - \bx \bx^* \otimes \by \by^* \in \C[\bx,\by,\bbx,\bby]^{d^2\times d^2}_4
\end{equation} 
and the localizing set of Hermitian polynomials 
\begin{equation}\label{eqSrho}
S_\rho=\Big\{ \sqrt{\rho_{\max}}-x_i\overline x_i, \sqrt{\rho_{\max}}-y_i\overline y_i : i\in [d]\Big\}\subseteq \C[\bx,\by,\bbx,\bby]^h_2.
\end{equation}
To see that $M(G_\rho\otimes L)\succeq 0$ we use Remark \ref{rem: M(GL) product and psd}.
Recall also the definition of the  set $S_{G_\rho}$ of localizing polynomials corresponding to the polynomial matrix $G_\rho$ in (\ref{eqGrho}):
 $$S_\Gr = \{v^* \Gr v: v \in \C^d \otimes \C^d\}\subseteq \C[\bx,\by,\bbx,\bby]^h_4.$$
Then, by construction,   the combined positivity domains of the sets $S$ and $ S_{\Gr}$ recover  $\calV_\rho$:
\[
\sD(S_\Gr \cup S_\rho)  = \calV_\rho=\Big\{ (x,y) \in \C^d \times \C^d ~|~ xx^* \otimes yy^* \preceq \rho, \ \|x\|_\infty, \|y\|_\infty \leq \rho_{\max}^{1/4} \Big\}.
\] 
\begin{remark}\label{rembound}
Note that the localizing constraints for the inequalities in (\ref{eqbound}) are implied by the localizing constraints for $S_\rho \cup S_{G_\rho}$.
This follows from the following two identities:
$$\rho_{\max}-x_i\ol x_i y_j\ol y_j= (\sqrt{ \rho_{\max}} -x_i\ol x_i)y_j\ol y_j +\sqrt{\rho_{\max}}(\sqrt{\rho_{\max}}-y_j\ol y_j) \in \calM(S_\rho)_4,$$
$$\Tr(\rho)-(\sum_i x_i\ol x_i)(\sum_j y_j\ol y_j)= \sum_{i,j} (\rho_{ij,ij} -x_i\ol x_i y_j\ol y_j ) \in \calM(S_{G_{\rho}})_4.$$
\end{remark}

Moreover, let us recall for future reference that 
\begin{equation}\label{eqMGLtoGL}
M(G_\rho\otimes L)\succeq 0\  \Longrightarrow\  L\ge 0 \text{ on } \calM(S_{G_\rho}),
\end{equation}
which follows from \cref{lemGLpos} and the characterization of positivity of $G_\rho\otimes L$ from \cref{eqMtGL}.
The above observations  motivate introducing the following parameters. For   $t \in \N \cup \{\infty \}$ with $t\ge 2$,  define the parameter	
	\begin{equation} \label{eqxit}
	\begin{split}
	\xidsep{t} := \inf \Big\{ L(1) ~|~ &L: \polyS_{2t}\to\C \ \text{ Hermitian\ \ s.t. }  \\
	& L(\bx\bx^*\otimes \by\by^*) = \rho, \\ 
	& L \geq 0 \text{ on }  \calM(S_\rho)_{2t}, \\
	& M_{t-2}(G_\rho\otimes L)=L(\Gr \otimes \monoV_{t-2} \monoV_{t-2}^*) \succeq 0\Big\}.
	\end{split}
	\end{equation}
For $t=\infty$ the  parameter $\xidsep{\infty}$ involves linear functionals acting on the full polynomial space  $\polyS$.
In addition, we let $\xidsep{*}$ denote  the parameter obtained by adding the constraint $\rank(M(L)) < \infty$ to the definition of  $\xidsep{\infty}$. 
One can show that the function  $\rho \mapsto \xidsep{t}$ is lower semicontinuous, the proof is analogous to that of \cite[Lemma~7]{GdLL17a} and thus omitted.  {In addition, as we will see in \cref{rem: level 2 implies ppt}, if the program defining $\xidsep{t}$ ($t\ge 2$) is feasible then $\rho$ satisfies the PPT criterion, i.e., $\rho^{T_B}\succeq 0$, where $\rho^{T_B}$ is obtained by taking the partial transpose of $\rho$ on the second register (see  (\ref{eqPPTB})). }

As is well-known, for finite $t \in \N$, the bound $\xidsep{t}$ can be expressed as a semidefinite program since  nonnegativity of $L$ on the truncated quadratic module $\calM(S_\rho)_{2t}$ can be encoded through positive semidefiniteness  of the moment matrix $M_t(L) = L([\bx,\by,\bbx,\bby]_{t}[\bx,\by,\bbx,\bby]_{t}^*)$ and of the localizing moment matrices $M_{t-1}(gL)=L(g[\bx,\by,\bbx,\bby]_{t-1}[\bx,\by,\bbx,\bby]_{t-1}^*)$ for all $g \in S_\rho$. 

By the above discussion, for any $\rho\in \SEP_d$  we have the following chain of inequalities:
\begin{equation}\label{eqchain}
\xidsep{2} \leq \xidsep{3} \leq \cdots \leq  \xidsep{\infty} \leq \xidsep{*}  \leq \seprank(\rho)  < \infty.
\end{equation}
We will now  show that the bounds $\xidsep{t}$ in fact converge to the parameter $\tausep(\rho)$. In a first step we observe that the parameters $\xidsep{t}$ converge to $\xidsep{\infty}$ and after that we show that $\xidsep{\infty}=\xidsep{*}=\tsep(\rho)$.

\begin{lemma}\label{lemconv}
Let $\rho\in\SEP_d$.  The infimum is attained in problem \eqref{eqxit} for any integer $t\ge 2$ or $t=\infty$, and we have $\lim_{t\to\infty}\xidsep{t}=\xidsep{\infty}$. 
\end{lemma}

\begin{proof}
First we show that problem \eqref{eqxit} attains its optimum. For this note that, in view of  \cref{eqchain}, we may restrict the optimization  to linear functionals $L$ satisfying $L(1) \le \seprank(\rho)$. 
By the definition of $S_\rho$ in \eqref{eqSrho}, the quadratic module $\calM(S_\rho)$ is Archimedean since, with $R=2d\sqrt{\rho_{\max}}$,  $R-\sum_{i=1}^d (x_i\overline x_i + y_i\overline y_i )\in  \calM(S_\rho)_2$. As  $L$ is nonnegative on $\calM(S_\rho)_{2t}$, we can apply \cref{pointconv}  
	and  conclude that 	$$
	|L(w)| \leq R^{|w|/2} L(1) ~~\text{for any}~~ w \in \monoV_{2t}.
	$$
Hence we are optimizing a linear objective function over a compact set, and thus the optimum is attained.
So, for each integer $t\ge 2$, let $L_t$ be an optimum solution of problem \eqref{eqxit}. As $\sup_t L_t(1)\le \seprank(\rho)<\infty$, we can conclude from \cref{pointconv} that there exists a linear functional $L\in \C[\bx,\by,\bbx,\bby]^*$ which is the limit of a subsequence of the sequence $(L_t)_t$. Then  $L$ is feasible for $\xidsep{\infty}$, which implies
$\xidsep{\infty}\le L(1) =\lim_{t\to\infty}L_t(1)=\lim_t\xidsep{t}$. Note that this $L$ is optimal for $\xidsep{\infty}$.
\end{proof}

\begin{lemma}\label{lemallequal}
For any $\rho\in \calH^d\otimes \calH^d$ we have $\xidsep{\infty} = \xidsep{*}= \tausep(\rho)$. 
\end{lemma}

\begin{proof}
As  $\xidsep{\infty}\le \xidsep{*}$  it suffices to show that $\xidsep{*}\le \tsep(\rho)$ and $\tsep(\rho)\le \xidsep{\infty}$.

First we show $\xidsep{*}\le \tsep(\rho)$. If $\tsep(\rho)=\infty$ there is nothing to prove. So assume we have a feasible solution: $\rho=\lambda \sum_{\ell=1}^K \mu_l a_\ell a_\ell^*\otimes b_\ell b_\ell^*$, where $\lambda>0$, $(a_\ell,b_\ell)\in \calV_\rho$, $\mu_\ell>0$ and $\sum_\ell \mu_\ell=1$.
Define the linear functional $L=\lambda\sum_{\ell=1}^K\mu_\ell L_{(a_\ell,b_\ell)}$. Then $L$ is feasible for $\xidsep{*}$ with $L(1)=\lambda$. Hence, $\xidsep{*}\le L(1)=\lambda$, which shows $\xidsep{*}\le \tsep(\rho)$.

Now we show $\tausep(\rho) \leq \xidsep{\infty}$. If $\xidsep{\infty}=\infty$ there is nothing to prove. 
So assume $L$ is a feasible solution to $\xidsep{\infty}$.
Then, in view of \cref{eqMGLtoGL},  $L\ge 0$ on $\calM(S_\rho \cup S_{G_\rho})$. As $\calM(S_\rho)$ is Archimedean we can apply \cref{theomainTchakaloff} (with $k=4$) and conclude that the restriction of $L$ to  $\C[\bx,\by,\bbx,\bby]_4$ is a conic combination of evaluations at points in $\sD(S_\rho\cup S_{G_\rho})=\calV_\rho$. In other words, there exist $(a_\ell,b_\ell)\in \calV_\rho$ and scalars $\mu_\ell>0$ such that 
$L(p)=\sum_{\ell=1}^K\mu_\ell p(a_\ell,b_\ell)$ for any $p\in \C[\bx,\by,\bbx,\bby]_4$.
In particular, we have $L(1)=\sum_{\ell=1}^K\mu_\ell$ and $\rho=L(\bx\bx^*\otimes \by\by^*)=\sum_{\ell=1}^K \mu_\ell \ a_\ell a_{\ell}^*\otimes b_\ell b_\ell^*$. This implies that ${1\over L(1)}\rho$ belongs to $\conv\{xx^*\otimes yy^*: (x,y)\in \calV_\rho\}$ and thus $\tsep(\rho)\le L(1)$, showing $\tsep(\rho)\le \xidsep{\infty}$.
\end{proof}

As observed earlier already, since $\SEP_d$ is a $d^4$-dimensional cone, by  Carath\'eodory theorem we have $\seprank(\rho)\le d^4$ for any $\rho\in \SEP_d$ (or, even stronger, $ \seprank(\rho)\le \rank(\rho)^2$). Based on this one can also use the bounds $\xidsep{t}$ to test (non-)membership in $\SEP_d$. The bound $\rank(\rho)^2$ in \cref{lemmembership} below can of course be replaced by any other valid upper bound on the separable rank. Such a valid bound can be obtained, e.g.,  using the {\em birank} of $\rho$, defined as the pair $(\rank(\rho),\rank(\rho^{T_B}))$. 
Indeed, as 
 $\seprank(\rho)=\seprank(\rho^{T_B})$, we have
\begin{equation}\label{eqbirank}\max\{\rank(\rho),\rank(\rho^{T_B})\}\le \seprank(\rho)\le (\min \{\rank(\rho),\rank(\rho^{T_B})\})^2.
\end{equation}

\begin{lemma}\label{lemmembership}
Let $\rho\in \calH^d\otimes \calH^d$. Then, $\rho \in \SEP_d$ if and only if $\xidsep{t}\le \rank(\rho)^2$ for all integers $t\ge 2$.
\end{lemma}

\begin{proof}
The `only if' part follows from  $\xidsep{t}\le \seprank(\rho)\le \rank(\rho)^2$ when $\rho \in\SEP_d$.
Conversely, assume $\xidsep{t}\le  \rank(\rho)^2$ for all integers $t\ge 2$. 
Then, we can use the same argument as in the proof of Lemma \ref{lemconv} and conclude the existence of $L\in\C[\bx,\by,\bbx,\bby]^*$ feasible for $\xidsep{\infty}$, so that $\xidsep{\infty}\le L(1)<\infty$.
Then, by Lemma \ref{lemallequal}, we have $\tsep(\rho)<\infty$, which shows  $\rho$ is separable.
\end{proof}

\begin{remark}\label{remweakbound}
Note that all the results in this section remain valid if, in the definition (\ref{eqxit}) of the parameter $\xidsep{t}$, we omit the `tensor-type' constraint $M_{t-2}(G_\rho\otimes L)\succeq 0.$
Using this additional constraint permits however to define stronger bounds on the separable rank. The results also remain valid if,   instead of the polynomials  in the set $S_\rho$, we use either of the following sets of polynomials: 
$\{\pm(\|x\|^2-\|y\|^2), \sqrt{\Tr(\rho)}- \|y\|^2\}$ corresponding to (\ref{eqab2}), or $\{\sqrt{\Tr(\rho)}-\|x\|^2, \pm 1(\sqrt{\Tr(\rho)}-\|y\|^2)\}$ corresponding to (\ref{eqbnorm1a}) (or, equivalently, $\{\Tr(\rho)-\|x\|^2, \pm (1-\|y\|^2)\}$ corresponding to (\ref{eqbnorm1})). 
\end{remark}

\subsection{Block-diagonal reduction for the parameter $\xib{sep}{t}(\cdot)$} \label{sec: block diagonalization}

In this section we indicate how to rewrite the program (\ref{eqxit}) defining $\xib{sep}{t}(\rho)$ in a more economical way. Observe that all the terms of each of the localizing polynomials $g\in S_\rho$ and the matrix $G_\rho$ have the same degree in $\bx$ and in $\bbx$, and also the same degree in $\by$ and in $\bby$. This enables us to show (see \cref{lemL0}) that we may restrict the optimization in (\ref{eqxit}) to linear functionals $L$ that satisfy the condition
\begin{equation}\label{eqL0}
L(\bx^\alpha\bbx^{\alpha'}\by^\beta\bby^{\beta'})=0 \ \text{ if } |\alpha|\ne| \alpha' |\text{ or } |\beta| \ne |\beta'|.
\end{equation}
Note that this implies in particular that $L(\bx^\alpha\bbx^{\alpha'}\by^\beta\bby^{\beta'})=0$ if $|\alpha+\alpha'|$ or $|\beta+\beta'|$ is odd.

The computational advantage is  that, if $L$ satisfies (\ref{eqL0}), then the moment matrix $M_t(L)$ and the localizing moment matrices $M_{t-1}(gL)$ and $M_{t-2}(G_\rho\otimes L) $ have a block-diagonal form. 
To see this consider first the matrix $M_t(L)$, which is indexed by the set 
\begin{equation} \label{eq: index set}
I^t:=\{(\alpha,\alpha',\beta,\beta')\in ( \N^d)^4 :  |\alpha+\beta+\alpha'+\beta'|\le t\}
\end{equation}
(where the tuple $(\alpha,\alpha',\beta,\beta')$ corresponds to the monomial $\bx^\alpha \bbx^{\alpha'}\by^\beta\bby^{\beta'}$).
Let us partition $I^t$  into sets depending on two integers $r = |\alpha|-|\alpha'|$ and $s = |\beta|-|\beta'|$. For $r,s \in\{-t,-t+1,\ldots, t\}$ let   
\begin{equation} \label{eq: partition}
I^t_{r,s} := \{(\alpha,\alpha',\beta,\beta') \in I^t: |\alpha|- |\alpha'| =r, \ |\beta|-|\beta'| = s\},
\end{equation}
then we have 
\[
I^t = \bigcup_{r,s=-t}^t I^t_{r,s}. 
\]
Then, with respect to this partition of its index set, the matrix $M_t(L)$ is block-diagonal and thus $M_t(L)\succeq 0$ if and only if its principal submatrices indexed by the sets $I^t_{r,s}$ are positive semidefinite.  The analogous reasoning  applies to each localizing moment matrix $M_{t-1}(gL)$ for $g\in S_\rho$ (indexed by $I^{t-1}$) and to $M_{t-2}(G_\rho\otimes L)$ (indexed by $I^{t-2}$).

\begin{lemma}\label{lemL0}
In the definition of the parameter  $\xidsep{t}$ we may restrict the optimization to linear functionals satisfying the additional condition (\ref{eqL0}).
\end{lemma}
\begin{proof}
Assume $L$ is feasible for $\xidsep{t}$; we construct another feasible solution $\tilde L$ with the same objective value: $\tilde L(1)=L(1)$, and satisfying (\ref{eqL0}).
For this define $\tilde L(\bx^\alpha \bbx^{\alpha'} \by ^\beta\bby^{\beta'})=L(\bx^\alpha \bbx^{\alpha'} \by ^\beta\bby^{\beta'})$ if $|\alpha|=|\alpha'|$ and $|\beta|=|\beta'|$,  and 
$\tilde L(\bx^\alpha\bbx^{\alpha'}\by^\beta\bby^{\beta'})=0$ otherwise. Then, $\tilde L(1)=L(1)$ and, by construction,  $\tilde L$ satisfies (\ref{eqL0}). We claim that $\tilde L$ is feasible for program (\ref{eqxit}).  Clearly, we have $\tilde L(\bx\bx^*\otimes \by\by^*)=\rho$.
We now show that $M_t(\tilde L)\succeq 0$, $M_{t-1}(g\tilde L)\succeq 0$ for $g\in S_\rho$, and 
$M_{t-2}(G_\rho\otimes \tilde L)\succeq 0$. 

We first show that $M_t(\tilde L)\succeq 0$. We use the partitioning $I^t = \cup_{r,s=-t}^t I^t_{r,s}$ of the row/column indices. 

As the principal submatrix of $M_t(\tilde L)$ indexed by $I^t_{r,s}$ only involves evaluations of $L$ at monomials of the form $\bx^\gamma\bbx^{\gamma'}\by^\delta\bby^{\delta'}$ with $|\gamma|=|\gamma'|$ and $|\delta|=|\delta'|$,  it coincides with the principal submatrix of $M_t(L)$ indexed by  $I^t_{r,s}$ and thus it is positive semidefinite.
Hence, by construction, the matrix $M_t(\tilde L)$ is block-diagonal with respect to the partition $I_t=\cup_{r,s=0}^t I_{r,s}$ of its index set, with positive semidefinite diagonal blocks, which implies  $M_t(\tilde L)\succeq 0$.

Consider now a localizing polynomial $g \in S_\rho$. Note that all its terms have the same degree in $\bx$ and $\bbx$ and also the same degree in $\by$ and $\bby$ (equal to 0 or 1). We consider the  partition of the index set of $M_{t-1}(gL)$ as $I^{t-1}=\cup_{r,s=-t+1}^{t-1} I^{t-1}_{r,s}.$ Again, the principal submatrix of $M_{t-1}(gL)$ indexed by $I^{t-1}_{r,s}$ involves only values $\tilde L(\bx^\gamma \bbx^{\gamma'}\by^\delta \bby^{\delta'})$ with $|\gamma|=|\gamma'|$ and $|\delta|=|\delta'|$ and thus it coincides with the principal submatrix of $M_{t-1}(gL)$ indexed by $I^{t-1}_{r,s}$.
Hence, the matrix $M_{t-1}(g\tilde L)$ is block-diagonal with respect to the partition $I_{t-1}=\cup I^{t-1}_{r,s}$ of its index set, with positive semidefinite diagonal blocks, which implies  $M_{t-1}(g\tilde L)\succeq 0$.

The analogous reasoning applies to showing that $M_{t-2}(G_\rho\otimes \tilde L)\succeq 0$. For this we consider the partition of its index set $[d]^2\times I_{t-2}$ into $\cup_{r,s=-t+2}^{t-2} ([d]^2\times I^{t-2}_{r,s})$ and observe that $M_{t-2}(G_\rho\otimes \tilde L)$ is block-diagonal with respect to this partition, with positive semidefinite diagonal blocks.
\end{proof}

\paragraph{An alternative way to arrive at \cref{eqL0} by exploiting sign symmetries.}
Let $\mathbb T$ be the \emph{circle group}, the multiplicative group of all complex numbers of modulus $1$:
\[
\mathbb T = \{z \in \C \colon \abs{z} = 1\}.
\]
The set $\SEP_d$ is naturally invariant under the action of $(w_x, w_y) \in \mathbb T \times \mathbb T$ on vectors $(x,y) \in \C^d \times \C^d$ given by $(w_x,w_y) (x,y) = (w_x x, w_y y)$ (and its extension to states). Indeed, we have 
\[
(w_x,w_y) \cdot (xx^* \otimes yy^*) = (w_x x)(w_x x)^* \otimes (w_y y) (w_y y)^* = xx^* \otimes yy^*.
\]
Likewise, the localizing constraints are invariant under this group action, and this group action extends to the linear functionals $L$ used as variables in the definition of
$\xidsep{t}$. Since $\mathbb T \times \mathbb T$ admits a Haar measure, in the derivation of $\xidsep{t}$ we may therefore restrict to linear functionals that are invariant under this group action.
That is, we may assume that 
\[
L(\bx^\alpha \bbx^{\alpha'} \by^\beta \bby^{\beta'}) = w_x^{\abs{\alpha} - \abs{\alpha'}} w_y^{\abs{\beta} - \abs{\beta'}} L(\bx^\alpha \bbx^{\alpha'} \by^\beta \bby^{\beta'})\quad \text{ for all } (w_x,w_y)\in \mathbb T\times\mathbb T.
\]
This implies that 
\begin{equation*}
L(\bx^\alpha\bbx^{\alpha'}\by^\beta\bby^{\beta'})=0 \ \text{ if } |\alpha|\ne| \alpha' |\text{ or } |\beta| \ne |\beta'|.
\end{equation*}
Indeed, suppose for example that $\abs{\alpha} - \abs{\alpha'} =: r \neq 0$. Then using the above with $w_x = e^{\i \pi /r}$ shows that $L(\bx^\alpha \bbx^{\alpha'} \by^\beta \bby^{\beta'}) = - L(\bx^\alpha \bbx^{\alpha'} \by^\beta \bby^{\beta'})$ and hence $L(\bx^\alpha \bbx^{\alpha'} \by^\beta \bby^{\beta'})=0$. 

Note that Dressler, Nie and Yang  \cite{dressler2020separability} used this same group action to argue that, alternatively,  one may restrict to  $(x,y)\in \C^d\times \C^d$ having leading coordinates that are real nonnegative:  $x_1,y_1\ge 0$. While this permits to eliminate variables (and work with $2(2d-1)$ instead of $4d$ real variables), this reduction does not permit to block-diagonalize the moment matrices as indicated above. We also refer to \cite{GP} and the recent paper \cite{TSOS} for more details about exploiting sign symmetries.

\paragraph{Block-diagonal reduction example.}
To illustrate the effect of the block-diagonalization we consider an example with $\rho\in \mathcal H^3\otimes \mathcal H^3\simeq \mathcal H^9$ (i.e., $d_1=d_2=3$) 
and relaxation order $t=3$.  In Table \ref{tableblockdiag} we indicate the respective sizes of the matrices involved in the program for $\xisep{3}$ with  and without block-diagonalization (in column `block' and `non-block', respectively). There, `\# entries' stands for $\sum_i m_i^2$, where $m_i$ are the sizes of the matrices involved in the program, and `\# variables' indicates the total number of variables  in each case.
The last line indicates the typical run time for such an instance, we collect the computational details later in \cref{secnumerics}.  Note that the full program cannot be solved and thus  block-diagonalization is crucial to enable computation. For the next case $(d_1,d_2)=(2,6)$ or $(4,4)$ one can compute the bound of order $t=2$ but not the bound of order $t=3$ even after block-diagonalization.

\begin{table}[!htbp]
	\centering
	\caption{Matrix sizes block- vs. non-block-diagonalized.}\label{tableblockdiag}
	\begin{tabular}{| c | c | c |} 
		\toprule
		Matrix & block  & non-block \\
		\hline \hline 
		$M_3(L)$  & $25 \times (12 \times 12 \text{ to } 96 \times 96)$ & $455 \times 455$\\
		$M_2(gL)$  & $78 \times (6 \times 6 \text{ to } 38 \times 38)$ & $6 \times (91 \times 91)$ \\
		$M_1(G_{\rho} \otimes L)$ &  $5 \times (36 \times 36 \text{ to } 108\times 108)$  & $234\times 234$ \\
		\hline 
		\# entries & $110480$ & $286624$  \\
		\# variables & 6952 & 18564 \\
		\hline
	run time & {4.6 min}  & memory error \\
			\bottomrule
	\end{tabular}
\end{table}

\begin{remark}
{As observed above, using the block-diagonalized version of the program for $\xisep{3}$ is crucial to be able to compute the bounds for some larger matrix sizes. We note however that the optimal solution to this program will not satisfy the flatness condition $\rank\, M_t(L)=\rank\, M_{t-1}(L)$ 
(with $t=2,3$). Indeed one can check that this flatness condition  can hold only in the trivial case $\rho=0$. 
Intuitively this can be (roughly) explained by noting that, due to its symmetric structure,~$L$  tends to lie within the interior of the feasible region.   Hence our approach, which produces lower bounds on $\seprank(\rho)$, can be viewed as being complementary to the approach in, e.g., \cite{dressler2020separability,Li_2020,NieZhang2016}, which uses flatness to produce separable decompositions of $\rho$ and thus upper bounds on $\seprank(\rho)$.} \end{remark}

\section{Extensions and connections to other matrix factorization ranks}\label{sec:extension}

Here   we explain some simple extensions of the approach given in the previous section to related notions of factorization ranks.

Without going into details let us mention that the approach 
generalizes in a straightforward way to the separable rank of \emph{multipartite} separable quantum states. In that case we have  an $n$-partite quantum state $\rho$ acting on $(\C^{d})^{\otimes n}$, and separability means that $\rho$   belongs to the set 
\[\mathrm{cone}\{x_1 x_1^* \otimes x_2 x_2^* \otimes \ldots \otimes x_n x_n^*: x_1, \ldots, x_n \in \C^d, \|x_i\| = 1 \ (i \in [n])\}.\]
In addition, one can use a different local dimension $d_i$ for each part (i.e., $x_i \in \C^{d_i}$). 

The approach also extends to an alternative (but equivalent) definition of separability, which  uses mixed states instead of pure states, i.e., where one requires $\rho$ to be of the form $\rho = \sum_{\ell=1}^r A_\ell \otimes B_\ell$ with $A_\ell, B_\ell \in \mathcal H_+^d$. Analogously, the smallest such integer $r$ is called the \emph{mixed separable rank} of $\rho$. 
This notion  has been considered, e.g.,  in~\cite{CDN19,de_las_Cuevas_2020,dressler2020separability} and mixed separable decompositions are called $S$-decompositions in \cite{NieZhang2016} (which deals with real states).
To define bounds on the mixed separable rank one can follow the same approach as in \cref{sec: LB} but one here has to introduce more variables. Indeed, we now need variables $\bx=(x_{ij})_{1\le i\le j\le d}$ and $\by=(y_{ij})_{1\le i\le j\le d}$  to model the entries of the matrices $A_\ell\in \H^d_+$ and  $B_\ell\in\H^d_+$ (while we previously only needed variables $(x_i)_{i\in [d]}$ and $(y_i)_{i\in [d]}$  to model the vectors $a_\ell\in\C^d$ and $b_\ell\in \C^d$) and one should assume that the corresponding  Hermitian matrices $X=(x_{ij})$ and $Y=(y_{ij})$ are  positive semidefinite. One may again scale the variables so that they satisfy a boundedness condition $|x_{ij}|,|y_{ij}|\le \sqrt {\rho_{\max}}$. This enables to design hierarchies of lower bounds that converge to the mixed separable analog of the parameter $\tau_{\text{sep}}(\rho)$. The details are analogous and thus omitted.

In what follows we mention two other possible extensions, for the real separable rank and for the completely positive rank, where we give some more details as well as some numerical results.

\subsection{Specialization to bipartite real states} \label{sec: real separable rank}
The treatment in Section \ref{sec: LB} for the separable rank can be adapted in an obvious manner to the case of the {\em real} separable rank. Here we are given a real symmetric bipartite state $\rho\in \calS^d\otimes \calS^d$, where $\calS^d$ is the set of real symmetric $d\times d$ matrices. 
Then $\rho$ is called {\em real separable} if it admits a decomposition (\ref{eqseprank}) with all vectors $a_\ell,b_\ell\in \R^d$ real valued,
and the  smallest $r$ for which such a decomposition exists is the {\em real separable rank}, denoted 
$\seprankR(\rho)$.
{Note that it can be that a real state is separable but not real separable; this is the case for the state Sep3 discussed in Section \ref{secnumerics}.}
One can define in an analogous manner the corresponding parameter $\tau_{\mathrm sep}^\R(\rho)$ and the hierarchy of bounds $\xidsepR{t}$ that converge asymptotically to $\tau_{\mathrm sep}^\R(\rho)$. The difference in the formulation of these parameters is that we now replace the complex conjugate by the real transpose operation and work with linear functionals $L$ acting on the real polynomial space $\R[\bx,\by]_{2t}$. So the parameter $\xidsepR{t}$ reads
	\begin{equation} \label{eqxitR}
	\begin{split}
	\xidsepR{t} := \inf \Big\{ L(1) ~|~ &L: \R[\bx,\by]_{2t}\to\R \ \text{  s.t. }  \\
	& L(\bx\bx^T\otimes \by\by^T) = \rho, \\ 
	& L \geq 0 \text{ on }  \calM(S_\rho)_{2t}, \\ 
	& M_{t-2}(G_\rho\otimes L)=L(\Gr \otimes [\bx,\by]_{t-2}[\bx,\by]_{t-2}^T) \succeq 0\Big\}.
	\end{split}
	\end{equation}
Again we may impose an additional block-diagonal structure on the positive semidefinite matrices entering this program. Indeed, since the polynomials involved in the constraints leading to the above program have the property that all their terms have an even degree in $\bx$ and an even degree in $\by$, we may assume that the variable $L$ satisfies the condition
\begin{equation}\label{eqL0R}
L(\bx^\alpha\by^\beta)=0 \ \text{ if } |\alpha| \text{ or } |\beta| \text{ is odd.}
\end{equation}
Note that this is the real analog of condition (\ref{eqL0}) in the complex case. The additional constraint (\ref{eqL0R}) permits to replace each of the positive semidefinite  constraints for the matrices $M_t(L)$, $M_{t-1}(gL)$ for $g\in S_\rho$,  and $M_{t-2}(G_\rho\otimes L)$ by four smaller positive semidefinite constraints, each of size roughly 1/4 of the original size. For this let $I^t$ denote the index set of the matrix $M_t(L)$ which we partition into $I^t=\cup_{a,b\in \{0,1\}}I^t_{a,b}$, where $I^t_{a,b}$ consists of the pairs $(\alpha,\beta)\in I^t$ with given parity $|\alpha|\equiv a$, $|\beta|\equiv b$ modulo 2. Then, with respect to this partition of its index set, the matrix $M_t(L)$ is block-diagonal and thus $M_t(L)\succeq 0$ if and only if  $M_t(L)[I^t_{a,b}]\succeq 0$ for $a,b\in \{0,1\}$. The same block-diagonalization applies to the matrices $M_{t-1}(gL)$ for $g\in S_\rho$. For the matrix $M_{t-2}(G_\rho\otimes L)$ we consider the block-diagonalization obtained by partitioning its index set as
$\cup_{a,b\in \{0,1\}} ([d^2] \times I^{t-2}_{a,b})$.

Some numerical results on the behaviour of the bounds will be given in the next section.

\subsection{Numerical results for bipartite complex and real states}\label{secnumerics}

Here we collect some numerical results that illustrate the behaviour of the bounds $\xidsep{t}$ and $\xidsepR{t}$ for different choices of localizing constraints, see \cref{table2,table3,table4} for examples at order $t =2,3,4$ respectively. Computations were made in Windows using Julia \cite{bezanson2017julia}, JuMP \cite{DunningHuchetteLubin2017} and MOSEK \cite{Mosek}  with hardware specifications: i7-8750 CPU with 32 Gb Memory.\footnote{The code is available at: \url{https://github.com/JAndriesJ/sep-rank}}
For our examples we will use the separable states Sep1, Sep2, and Sep3, and the entangled state Ent1 that we describe now. For numerical stability we do the computations with a scaling of these states so that they have trace equal to 1. We present the examples in matrix form with lines drawn to indicate the block structure $\rho=\big(\big(\rho_{ij,i'j'} \big)_{j,j'\in [d_2]}\big)_{i,i' \in [d_1]}$. Zero-valued entries are left blank.
\begin{equation*}
	\mathrm{Sep1} := 
	\left[\begin{array}{cc|cc}
		1   &  &  &   \\
		&  &  &   \\
		\hline
		&  &  & \  \\
		&  &  &    1  \\
	\end{array}\right]
	~;~
	\mathrm{Sep2} := 
	\left[\begin{array}{cc|cc}
		2   &   1 & 1 & 1  \\
		1   &   1 & 1 & 1  \\
		\hline
		1   &   1 & 1 & 1  \\
		1   &   1 & 1 & 2  \\
	\end{array}\right]
\end{equation*}	

\begin{equation*}
	\mathrm{Sep3} := 
	\left[\begin{array}{ccc|ccc}  
		4     &    &    &    &    &     \\
		& 4  & 2     &    &    & 2     \\
		& 2     & 2     &1      & -1    &     \\
		\hline
		&    & 1     & 2     & 1     & -1     \\
		&    & -1    & 1     & 5     & 1      \\
		& 2     &    & -1    & 1 & 2          \\
	\end{array}\right]
	~;~
	\mathrm{Ent1} = 
	\left[\begin{array}{ccc|ccc|ccc}
		1& & & &1& & & &1\\
		&2& &1& & & & & \\
		& &\frac{1}{2}& & & &1& & \\
		\hline
		&1& &\frac{1}{2}& & & & & \\
		1& & & &1& & & &1\\
		& & & & &2& &1& \\
		\hline
		& &1& & & &2& & \\
		& & & & &1& &\frac{1}{2}& \\
		1& & & &1& & & &1
	\end{array}\right].
\end{equation*}	
The separable states Sep1, Sep2, and Sep3 were previously studied for example in \cite{Chen_2012}, where it was moreover shown that for a separable state $\rho$ with local dimensions $(d_1,d_2)=(2,3)$ and birank $(r,s)$ one has $\seprank(\rho)=\max\{r,s\}$.  The entangled state Ent1 was constructed by Choi in~\cite{Choi82} as the first example in dimension $(d_1,d_2) = (3,3)$ of an entangled state $\rho$ that satisfies the PPT condition. 

In \cref{sec:tausep} we provided three different choices of localizing constraints in (\ref{eqabmax}), (\ref{eqab2}) and (\ref{eqbnorm1}),  that we denote here as S1, S2 and S3, respectively.  The examples show that the different choices lead to incomparable bounds. Indeed, let us the notation S1~$<$~S2 as short hand for ``there exists a $\rho$ such that $\xidsep{t}$ (using scaling $S1$) $< \xidsep{t}$ (using scaling $S2$)''. Then at level $t=2$ the state Sep1 demonstrates  both S3~$<$~S1 and S2~$<$~S1, and at level $t=3$ Sep2 demonstrates both S2~$<$~S3 and S1~$<$~S3 and Sep3 demonstrates both S1~$<$~S2 and S3~$<$~S2. A case where the various constraints differ in ability to detect entanglement is provided by the state Ent1 at order $ t=2$.	

As mentioned in \cref{sec: real separable rank}, there exist real states $\rho \in \calS^d\otimes \calS^d$ that are separable but do not admit a decomposition using real vectors $a_\ell, b_\ell \in \R^d$. Our bound $\xidsepR{2}$ provides a proof of the latter for the state Sep3: its real separable rank is infinity since our lower bound is infeasible (i.e., there exists a dual certificate that proves $\seprankR(\text{Sep3}) = \infty$). 

Finally we note that one sometimes needs to go beyond level $t=2$ (and thus beyond the PPT criterion) to reveal entanglement: with the localizing constraints S3 the bound for Ent1 is feasible at $t=2$, but infeasible at $t=3$.

\medskip
In addition, we show in \cref{scat} a scatter plot of the bound $\xi_{3}^{\mathrm{sep}}(\rho)$  vs.~its computation time in seconds for 100 random complex matrices $\rho$ grouped and colored by the respective scalings S1, S2 and S3.
These  matrices  are defined by
$\rho = \sum_{j=1}^{5} a^{(j)} {({a}^{(j)})}^* \otimes b^{(j)} {({b}^{(j)})}^*, $
where $a^{(j)},b^{(j)} \in \C^3$ are random vectors whose entries are of the form $ x + \i y$ with  $x,y \in \mathcal{N}(0,1)$. (We also normalize the trace here for numerical stability.) This construction guarantees separability and provides the upper bound $\seprank(\rho)\le 5$. Such states also satisfy the reverse inequality $\seprank(\rho) \geq 5$ almost surely (since $\rank(\rho)=5$ almost surely). We use this class of examples merely to test the quality of the bounds.
From the figure we can draw the following observations. First, the bounds are concentrated around the means 2.7, 3.4 and 3.3 for the scalings S1, S2 and S3, respectively.
Second, in this class of examples the S1 rescaling yields inferior bounds as compared to S2 and S3.
Third, out of the hundred examples and for the three different scalings considered, no bound exceeded the value 4.

\begin{table}[!htbp]
	\centering
	\caption{Examples and numerical bounds level $t=2$}\label{table2}
	\begin{tabular}{|c c c | c c c | c c c |c|c|}
		\toprule
		$\rho$ & $(d_1,d_2)$ & $\birank(\rho)$   &  \multicolumn{3}{c}{$\xidsep{2}$}  &  \multicolumn{3}{c}{$\xidsepR{2}$} & $\seprank(\rho)$  & time \\
		\hline\hline
		{}&{}&{}&S1&S2&S3&S1&S2&S3&&\\	
		\hline
		Sep1\cite{Chen_2012}&(2, 2)&(2, 2)&\textbf{2.0}&1.0&1.0&\textbf{2.0}&1.0&1.0&2&$<1$\\
		Sep2\cite{Chen_2012} &(2, 2)&(3, 3)&1.421&1.0&1.0&1.421&1.0&1.0&3&$<1$ \\
		Sep3\cite{Chen_2012} &(2, 3)&(4, 6)&1.333&1.0&1.0&*&*&*&6&$<1$  \\
		Ent1\cite{Choi82}&(3, 3)&(4, 4)&2.069&\textbf{*}&1.525&2.069&\textbf{*}&1.525&$\infty$&$<1$ \\
		\bottomrule
	\end{tabular}
	
	\caption{Examples and numerical bounds level $t=3$}\label{table3}
	\begin{tabular}{|c c c | c c c | c c c |c|c|}
		\toprule
		$\rho$ & $(d_1,d_2)$ & $\birank(\rho)$   &  \multicolumn{3}{c}{$\xidsep{3}$}  &  \multicolumn{3}{c}{$\xidsepR{3}$} & $\seprank(\rho)$ & time \\
		\hline \hline
		 {}&{}&{}&S1&S2&S3&S1&S2&S3&&\\	
		\hline			
		Sep1&(2, 2)&(2, 2)&\textbf{2.0}&\textbf{2.0}&\textbf{2.0}&\textbf{2.0}&\textbf{2.0}&\textbf{2.0}&2&$<1$	\\
		Sep2&(2, 2)&(3, 3)&1.909&2.0&\textbf{2.178}&1.909&2.0&2.178&3&$2$	\\
		Sep3&(2, 3)&(4, 6)&2.423&3.0&2.790&*&*&*&6&25	\\
		Ent1&(3, 3)&(4, 4)&-&-&\textbf{*}&-&\textbf{*}&\textbf{*}&$\infty$&267 \\
		\bottomrule
	\end{tabular}
		
		\caption{Examples and numerical bounds level $t=4$}\label{table4}
		\begin{tabular}{|c c c | c c c | c c c |c|c|}
			\toprule
			$\rho$ & $(d_1,d_2)$ & $\birank(\rho)$   &  \multicolumn{3}{c}{$\xidsep{4}$}  &  \multicolumn{3}{c}{$\xidsepR{4}$} & $\seprank(\rho)$ & time \\
			\hline \hline
			{}&{}&{}&S1&S2&S3&S1&S2&S3&&\\	
			\hline			
			Sep1&(2, 2)&(2, 2)&\textbf{2.0}&\textbf{2.0}&\textbf{2.0}&\textbf{2.0}&\textbf{2.0}&\textbf{2.0}&2&105	\\
			Sep2&(2, 2)&(3, 3)&\textbf{3.0}&\textbf{3.0}&\textbf{3.0}&\textbf{3.0}&\textbf{3.0}&\textbf{3.0}&3& 332 \\
			\bottomrule
			\multicolumn{11}{@{}p{5in}}{\footnotesize Run time given in seconds}\\
			\multicolumn{11}{@{}p{5in}}{\footnotesize * : Infeasibility certificate returned}\\
			\multicolumn{11}{@{}p{5in}}{\footnotesize - : Solver could not reach a conclusion (not a memory error)} \\
	\end{tabular}

\end{table}	

\begin{figure}[ht]
	\centering
	\begin{tikzpicture} [scale=1]
		\begin{axis}[enlargelimits=false,axis on top,ylabel =time (s), xlabel = $\xi_{3}^{\mathrm{sep}}(\rho)$, 
			xtick={2.00, 2.25, 2.50, 2.75, 3.00, 3.25, 3.50, 3.75, 4.00},ytick={100, 200, 300, 400, 500, 600, 700},
			xticklabel style={tick align = outside},
			yticklabel style={tick align = outside}
			]
			\addplot graphics
			[xmin=2.00,xmax=4.10,ymin=80,ymax=740]
			{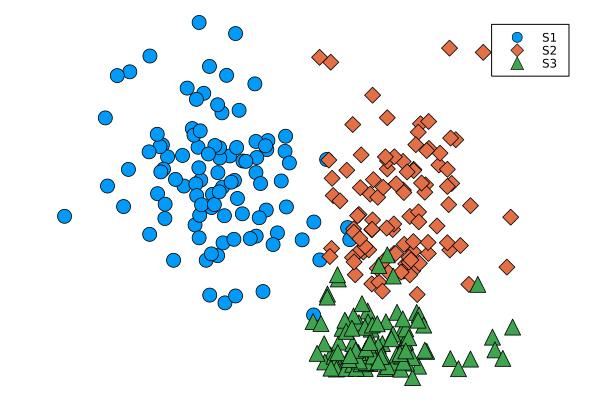};
		\end{axis}
	\end{tikzpicture}
	\caption{Scatter plot of $\xi_{3}^{\mathrm{sep}}(\rho)$  vs computation time (seconds) for 100 random matrices, grouped and colored by rescalings S1, S2 and S3.}
	\label{scat}
\end{figure}

\subsection{Stronger bounds for  the completely positive rank} \label{sec:CP}

For a given integer $d \in \N$, the cone of completely positive $d \times d$ matrices is defined as 
\[
\CP_d := \mathrm{cone}\{ xx^T :  x \in \R^d_{+} \}. 
\]
The cone of completely positive matrices and its dual, the cone of copositive matrices, are well known for their expressive power. For example, many NP-hard problems can be formulated as linear optimization problems over these cones~\cite{dKP02,Bur09}. We refer to~\cite{BSM03} for many structural properties about the cone $\CP_d$. As in the case of separable states, given a completely positive matrix $A$ one can ask what is the smallest integer $r \in \N$ such that $A$ admits a decomposition of the form 
\begin{equation}\label{eqAcp}
A = \sum_{\ell=1}^r a_\ell a_\ell^T
\end{equation}
for entrywise nonnegative vectors $a_\ell \in \R^d_{+}$ ($\ell \in [r]$). The smallest such $r$ is called the \emph{completely positive rank} of $A$ and denoted as $\cprank(A)$. In~\cite{FP16} the authors defined the parameter $\taucp(A)$ as
\begin{equation} \label{eq:taucp}
\taucp(A) := \inf\Big\{\lambda: \lambda>0,  {1\over \lambda} \rho \in \conv \{xx^T: x \in \R^d_{+},\ xx^T \leq A,\ xx^T \preceq A\} \Big \}
\end{equation}
to lower bound the completely positive rank (as well as  an SDP-based bound $\taucp^{\mathrm{sos}}(A)$).  In~\cite{GdLL17a} the authors studied (among others) the completely positive rank from the polynomial optimization perspective and  derived a hierarchy of semidefinite programming bounds, denoted here as  {$\xib{cp}{t,(2019)}(A)$.} There the fact was  used that, if $xx^T \preceq A$, then also $(xx^T)^{\otimes \ell} \preceq A^{\otimes \ell}$ for all $\ell \in \N$ and therefore the following constraints are valid
\begin{equation} \label{eq: tensor}
L((\bx\bx^T)^{\otimes \ell}) \preceq A^{\otimes \ell} \qquad \text{for all } \ell \in \N
\end{equation}
for the linear functional arising from the atomic decomposition (\ref{eqAcp}). 
Based on this the following bounds are defined in~\cite{GdLL17a} and shown to converge to $\taucp(A)$ as $t \to \infty$:
	\begin{equation} \label{eqxitcp}
	\begin{split}
	{\xib{cp}{t,(2019)}(A) }:= \inf \Big\{ L(1): &\ L \in [\bx]_{2t}^*,  \\
	& L(\bx\bx^T) = A, \\ 
	& L \geq 0 \text{ on }  \calM(\{\sqrt{A_{ii}} x_i - x_i^2 : i \in [d]\})_{2t}, \\ 
	& L \geq 0 \text{ on }  \calM(\{A_{ij} - x_i x_j : i,j \in [d], i \neq j\})_{2t}, \\
	& L((\bx\bx^T)^{\otimes \ell}) \preceq A^{\otimes \ell} \text{ for all } \ell \in [t] \Big\}.
	\end{split}
	\end{equation}
The same convergence result holds if we replace the last constraint in (\ref{eqxitcp}) with the constraint
\begin{equation} \label{eq: v CP}
L \geq 0 \qquad \text{ on } \calM(\{v^T (A-\bx\bx^T) v: v \in \R^d\})_{2t}.
\end{equation}
Using the same reasoning as in \cref{sec: poly sep}, we see that we can strengthen the parameter $\xib{cp}{t,(2019)}(A)$ by adding the constraint
\begin{equation} \label{eq: GL cp}
M_{t-1}((A-\bx\bx^T)\otimes L) = L((A-\bx\bx^T) \otimes [\bx]_{t-1} [\bx]_{t-1}^T) \succeq 0.
\end{equation}
{Let $\xib{cp}{t}(A)$ denote the parameter defined in this way, so that $\xib{cp}{t,(2019)}(A)\le \xib{cp}{t}(A)$.} Note that   \cref{lemGLpos,lemMGLtoGL} show that \cref{eq: GL cp} implies \cref{eq: v CP}. 
We now show that \cref{eq: GL cp} in fact implies \cref{eq: tensor}, which means that adding \cref{eq: GL cp} strengthens both approaches provided in \cite{GdLL17a}; we present below numerical examples that illustrate this.
To do so, we introduce the following notation. Let $\langle x \rangle$ denote the vector of noncommutative monomials in the variables $x_1,\ldots, x_d$. Then we can define the \emph{noncommutative localizing matrix}
\begin{equation} \label{eq: Mnc}
\Mnc((A-xx^T)\otimes L) := L( (A-xx^T) \otimes \langle x\rangle \langle x\rangle^T).
\end{equation}
Note that $M((A-xx^*)\otimes L) \succeq 0$ if and only if $\Mnc((A-xx^*)\otimes L) \succeq 0$ (since the latter is obtained by duplicating rows/columns of the former). 

\begin{lemma}
	Consider $A \in \R^{d \times d}$ and $L \in \R[\bx]^*$. If $L(\bx\bx^T) = A$ and $M((A-\bx\bx^T)\otimes L) \succeq 0$, then \cref{eq: tensor} holds, i.e., 
	\[
	L((\bx\bx^T)^{\otimes \ell}) \preceq A^{\otimes \ell} \text{ for all } \ell \in \N.
	\]
\end{lemma}

\begin{proof}
	As observed above, $M((A-\bx\bx^T)\otimes L) \succeq 0$ if and only if $\Mnc((A-\bx\bx^T)\otimes L) \succeq 0$. Note that for each $\ell \in \N$, the matrix $\Mnc((A-\bx\bx^T)\otimes L)$ contains $L( (A-\bx\bx^T) \otimes (\bx\bx^T)^{\otimes \ell-1})$ as a principal submatrix. To see this write  the vector $\langle x\rangle$ of noncommutative monomials  as 
	$
	1 \oplus_{\ell \in \N} x^{\otimes \ell}
	$
	by grouping the monomials according to their degree.
	With respect to this partition of its index set  
	the matrix $\Mnc((A-xx^*)\otimes L)$ 
	has  the matrices $L( (A-\bx\bx^*) \otimes (\bx\bx^*)^{\otimes \ell-1})$ as its   diagonal blocks. 
Since $\Mnc((A-\bx\bx^*)\otimes L) \succeq 0$, we  obtain
	\[
	A \otimes L((\bx\bx^T)^{\otimes \ell-1}) \succeq L((\bx\bx^T)^{\otimes \ell}) \qquad \text{ for all } \ell \in \N.
	\]
	Combined with $L(\bx\bx^*) = A$ this permits to show \cref{eq: tensor}:
	\[
	 L((\bx\bx^T)^{\otimes \ell} \preceq A \otimes  L((\bx\bx^T)^{\otimes (\ell-1)} \preceq A^{\otimes 2} \otimes  L((\bx\bx^T)^{\otimes (\ell-2)} \preceq \cdots \preceq A^{\otimes( \ell-1)} \otimes L(\bx\bx^T) = A^{\otimes \ell}. \qedhere
	\]
\end{proof}

We conclude this section with some numerical results. To demonstrate the impact of the constraints (\ref{eq: GL cp}) we compare our bounds $\xib{cp}{3}(A)$ to the bounds 
$\xib{cp}{3,(2019)}(A)$ from \cite{GdLL17a} on the cp-rank of some matrices $A$ known to have a high cp-rank, taken from \cite{BOMZE2014208}. The boldface entries in Table \ref{CP_bounds} show a strict improvement in the bounds. For these computations we used the high precision solver SDPA-GMP \cite{5612693} because MOSEK \cite{Mosek} and SDPA \cite{Yamashita10ahigh-performance, Yamashita03implementationand} could not certify solutions.\footnote{The code is available at: \url{https://github.com/JAndriesJ/ju-cp-rank}}

\begin{table}[!htbp]\label{CP_bounds}
	\centering
	\caption{Bounds for completely positive rank at level t=3.}
	\begin{tabular}{|c c c c c c c  |} 
		\toprule
		$A$ &  $\rank(A)$ &$n$& $\lfloor \frac{n^2}{4}\rfloor$ & $\xib{cp}{3,(2019)}(A)$ & $\xib{cp}{3}(A)$ & $\cprank(A)$ \\ [0.5ex] 
		\hline\hline
		$M_7$ & 7&7&12&10.5&\textbf{11.4} & 14    \\ 
		\hline
		$\wtl M_7$  &7&7&12&10.5&10.5 &  14    \\  
		\hline 
		$\wtl M_8$   &8&8&16&13.82 &\textbf{14.5} & 18     \\ 
		\hline  
		$\wtl M_9$   &9&9&20&17.74&\textbf{18.4}&  26    \\  
		\bottomrule
	\end{tabular}
\end{table}

\section{Entanglement witnesses}\label{secDPS}

The moment approach we have developed in the previous section for bounding the separable rank of a state $\rho$ can be viewed as searching for a (non-normalized) measure on the product of two balls, with the additional property that, for any point $(x,y)$ in its support, we have $\rho - xx^* \otimes yy^* \succeq 0$.  
We will first observe in Section \ref{sec:mem} how  this approach can also be used to detect entanglement, i.e., non-membership in the set $\SEP$.

As mentioned earlier  
one can  also capture the set $\SEP$ by viewing it as a moment problem 
on the bi-sphere (the product of two unit spheres). In the rest of this section we will show that this second moment approach  corresponds exactly to the well-known state extension perspective  that leads to the Doherty-Parrilo-Spedalieri hierarchy of approximations of $\sep$ from \cite{DPS04}.

\subsection{Entanglement witnesses based on the hierarchy of parameters $\xidsep{t}$}\label{sec:mem}
Our approach to design lower bounds on the separable rank also directly leads to a way to detect non-membership in the set $\SEP$ or,  in other words,  to a way to witness entanglement of a state. Indeed, as shown in \cref{lemmembership}, a state $\rho$ is separable if and only if $\xidsep{t} \leq \rank(\rho)^2$ for all $t \ge 2$. In other words, $\rho$ is entangled if and only if $\xidsep{t}>\rank(\rho)^2$ for some integer $t\ge 2$ (which includes $\xidsep{t}=\infty$ in case the program defining 
$\xidsep{t}$ is infeasible).

In order to get a certificate of entanglement it is therefore convenient to consider the dual semidefinite program to the program (\ref{eqxit}) defining the parameter 
$\xidsep{t}$, which reads:
	\begin{equation} \label{eqxit dual}
	\begin{split}
\sup \Big\{ \langle \rho, \Lambda\rangle  ~|~  \Lambda \in \C^{d \times d} \otimes \C^{d \times d} &  \text{ Hermitian  s.t. }  \\
	 1-\langle \Lambda, xx^*\otimes yy^*\rangle \in 
	  &\ \calM(S_\rho)_{2t}
	 + \cone\{ \langle G_\rho, \vec{p} \vec{p}^*\rangle: \vec p \in (\C[\bx,\bbx,\by,\bby]_{t-2})^{d^2} \} \Big\}.
	\end{split}
	\end{equation}
	
	\begin{lemma}
	For any integer $t \geq 2$, the matrix $\Lambda = 0$ is a strictly feasible solution  for \eqref{eqxit dual}.
	\end{lemma}
	
\begin{proof}
	First we observe that, for small  $\lambda>0$, the matrix $\Lambda = \lambda \cdot I_{d^2}$ is a feasible solution for (\ref{eqxit dual}).
	For this we show that  the polynomial $1-\lambda \langle I, xx^*\otimes yy^*\rangle=1-\lambda \sum_i\sum_j x_i\ol x_iy_j\ol y_j$ lies in the quadratic module $\calM(S_\rho)_{2t}$ for small $\lambda>0$.  We know that there exists a scalar $R>0$  such that $R - \sum_{i} x_i \ol x_i - \sum_j y_j \ol y_j\in \cal M(S_\rho)$. Then also $R-\sum_{i} x_i \ol x_i$ and 
	$R- \sum_iy_i \ol y_i$ lie in $\calM(S_\rho)_{2t}$, as well as 
	$(R-\sum_{i} x_i \ol x_i)(R+ 	\sum_iy_i \ol y_i)$ and 
	$(R+\sum_{i} x_i \ol x_i)(R- 	\sum_iy_i \ol y_i)$. Adding up the latter two polynomials we obtain that the polynomial 
	$R^2- \sum_i\sum_j x_i\ol x_i y_j\ol y_j$ belongs to $\calM(S_\rho)_{2t}$, which shows that 		
	$\Lambda = \lambda \cdot I_{d^2}$ is feasible for all $0<\lambda\le R^{-2}$. 
	
	We now show that any $\Lambda$ satisfying  $\|\Lambda\|\leq R^{-2}$ is feasible (which shows there is a ball contained in the feasible region of~\eqref{eqxit dual}). 
	For this write
	\[
	1- \langle \Lambda,xx^*\otimes yy^*\rangle = \underbrace{1- \langle \|\Lambda\| I_{d^2}, xx^*\otimes yy^*\rangle}_{(a)} +  
	\underbrace{\langle \|\Lambda\| I_{d^2} - \Lambda, xx^*\otimes yy^* \rangle}_{(b)}.
	\]
	In the first part of the proof we have shown that term $(a)$ belongs to $\calM(S_\rho)_{2t}$ if $\|\Lambda\| \leq R^{-2}$. In addition,  term $(b)$ is a sum of squares since $\|\Lambda\| I_{d^2} - \Lambda$ is positive semidefinite. Together, this shows $1-\langle \Lambda, xx^*\otimes yy^*\rangle  \in \calM(S_\rho)_{2t}$ and therefore $\Lambda$ is feasible. 
	\end{proof}
	
As a consequence, strong duality holds between the program (\ref{eqxit}) defining $\xidsep{t}$ and its dual~\eqref{eqxit dual}. 
That is, if the program (\ref{eqxit dual}) is bounded then its optimal value is finite and equal to $\xidsep{t}$ and, otherwise, its optimal value is equal to $\infty$ and thus $ \xidsep{t}$  is infeasible.
Therefore, we obtain that $\rho$ is entangled if and only if, for some integer $t\ge 2$,
  there exists a matrix $\Lambda \in \C^{d \times d} \otimes \C^{d \times d}$ which is feasible for \eqref{eqxit dual} and satisfies $\langle \rho, \Lambda\rangle>\rank(\rho)^2$. In that case such matrix $\Lambda$ provides a certificate that the state $\rho$ is entangled. 

\subsection{The Doherty-Parrilo-Spedalieri hierarchy: moment perspective}\label{sec:momentDPS}

Recall definition (\ref{SEP}) of the set of separable states $\SEP_d$,
so $\rho \in\SEP_d$ if and only if it is of the form
\begin{equation}\label{decrho}
\rho=\sum_{\ell=1}^r \lambda_\ell \ a_\ell a_\ell^*\otimes b_\ell b_\ell^*,
\end{equation}
where $\lambda_\ell>0$, $a_\ell,b_\ell\in \C^d$ with $\|a_\ell \|=1=\|b_\ell\|$. To this decomposition we can associate a linear functional  on $\C[\bx,\bbx,\by,\bby]$ that is a conic combination of evaluation functionals at points on the bi-sphere: $L = \sum_{\ell=1}^r \lambda_\ell L_{(a_\ell,b_\ell)}$. By construction, this linear functional is positive on Hermitian squares, it vanishes  on the ideal generated by $1-\|\bx\|^2$ and $1-\|\by\|^2$ (called the {\em bi-sphere ideal} for short) and it satisfies $L(xx^* \otimes yy^*) = \rho$. This naturally suggests a hierarchy of outer approximations to the set $\SEP$: a state $\rho$ belongs to the $t$-th level of this hierarchy if there exists an $L$ that satisfies these constraints for polynomials of degree at most~$2t$. Formally, we consider the set 
\begin{equation} \label{eq: dpst}
\begin{split}
	\set_{t} := \{ \rho \in \calH^d\otimes \calH^d :\ &\exists L:\C[\bx, \qo \bx, \by, \qo \by]_{2t} \to \C \text{ Hermitian s.t.  }\\
	& L(\bx\bx^*\otimes \by\by^*)=\rho,\\
	& L = 0 \text{ on } \mathcal I(1-\|\bx\|^2,1-\|\by\|^2 )_{2t},\\
	& M_t(L)\succeq 0 
	\}.
\end{split}
\end{equation}
We will show in Section \ref{sec: state extension perspective}  that this set is in fact closely related to the DPS hierarchy of outer approximations to the set $\sep_d$: if we introduce separate degree-bounds on the  $\bx,\bbx$ variables and the $\by,\bby$ variables, then we recover the original formulation from \cite{DPS04}. 

First we note that we can easily show that the sets $\set_{t}$ converge to $\sep$, i.e.,   $\SEP_d=\bigcap_{t\ge 2}\set_{t}$, using the tools from polynomial optimization (\cref{theomainTchakaloff}).

\begin{proposition} \label{lem: convergence set t}
We have: $\SEP_d=\bigcap_{t\ge 2} \set_{t}.$
\end{proposition}

\begin{proof}
Assume $\rho\in \bigcap_{t\ge 2}\set_{t}$, we show  $\rho\in \SEP_d$. For any $t\ge 2$ let $L_t$ be an associated certificate of membership in $\set_{t}$.
Then we have $L_t(1)=L_t(\|\bx\|^2\|\by\|^2)= \Tr(\rho)$. Hence it follows from \cref{pointconv} that the sequence $(L_t)_t$ has a pointwise converging subsequence, with limit $L\in \C[\bx,\bbx,\by,\bby]^*$. 
Then $L\ge 0$ on $\Sigma$ and $L=0$ on $\mathcal I(1-\|\bx\|^2,1-\|\by\|^2)$. Using \cref{theomainTchakaloff} we can conclude that there exists scalars $\mu_\ell>0$ and points $(a_\ell,b_\ell)\in \C^d\times \C^d$ with $\|a_\ell\| = \|b_\ell\|=1$ such that 
$L(p)=\sum_{\ell=1}^K \mu_\ell p(a_\ell,b_\ell)$ when $p$ has degree at most 4. In particular, we obtain 
$$\rho=L(\bx\bx^*\otimes \by\by^*) =\sum_{\ell=1}^K \mu_\ell \ a_\ell a_\ell^*\otimes b_\ell b_\ell^*,$$
which shows that $\rho\in \SEP_d$.
\end{proof}

Next, we reformulate the positivity condition $M_t(L)\succeq 0$ in a way that will be useful for making the link to the DPS hierarchy.
As observed in \cref{sec: block diagonalization}, we may additionally require the linear functionals in \cref{eq: dpst} to satisfy the constraint \cref{eqL0}, which we repeat here for convenience:
\begin{equation}\label{eqL02}
L(\bx^\alpha\bbx^{\alpha'}\by^\beta\bby^{\beta'})=0 \ \text{ if } |\alpha|\ne| \alpha' |\text{ or } |\beta| \ne |\beta'|.
\end{equation}
This permits to block-diagonalize the associated moment matrix $M_t(L)$ according to the partition given in \cref{eq: index set,eq: partition}, thus permitting to replace the constraint $M_t(L)\succeq 0$ by $M_t(L)[I^t_{r,s}]\succeq0$ for $r,s\in [-t,t]$.
 In fact, using the bi-sphere ideal constraint, one can reduce the size of these matrices  
 even further  and replace the matrices $M_t(L)[I^t_{r,s}]$ by their submatrices $M_t(L)[I^{=t}_{r,s}]$, where the sets  $I^{=t}_{r,s}\subseteq I^t_{r,s}$ are defined by  \begin{equation} \label{eq: block diag hom t}
I^{=t}_{r,s} :=\left\{(\alpha,\alpha',\beta,\beta') \in (\N^d)^4: |\alpha+\alpha'+\beta+\beta'|=t, \ |\alpha|-|\alpha'| = r, \ |\beta|-|\beta'|=s\right\}. 
\end{equation}
In other words we can show
the following  reformulation of the set $\set_t$: 
\begin{equation} \label{eq: dpst SDP}
\begin{split}
	\set_{t} = \big\{ \rho \in \calH^d\otimes \calH^d :\ &\exists L:\C[\bx, \qo \bx, \by, \qo \by]_{2t} \to \C \text{ Hermitian s.t.  }\\
	& L(\bx\bx^*\otimes \by\by^*)=\rho,\\
	& L = 0 \text{ on } \mathcal I(1-\|\bx\|^2,1-\|\by\|^2 )_{2t},\\
	& M_{t}(L)[I^{=t}_{r,s}] \succeq 0 \text{ for all } r,s \in \{-t,-t+1,\ldots,t\} \big\}.
\end{split}
\end{equation}
We will show this result in a slightly different setting (closer to that of the original  formulation of the DPS hierarchy). Similar arguments as those used in the proof of Lemma \ref{prop: restrict to hom} below can be used to show the equivalence between \eqref{eq: dpst} and \eqref{eq: dpst SDP}. 

\medskip
In order to connect the moment approach on the bi-sphere to the original formulation of the DPS hierarchy we need to introduce a separate degree bound on the $\bx,\bbx$ variables and the $\by,\bby$ variables. For integers $k,t \ge 1$ we let $\C[\bx,\bbx,\by, \bby]_{k,t}$ (resp., $\C[\bx,\bbx,\by,\bby]_{=k,=t}$) denote the set of polynomials that have degree at most $k$ (resp., equal to $k$) in $\bx,\bbx$ and degree at most $t$ (resp., equal to $t$) in $\by,\bby$, and we set
\[
\Sigma_{2k,2t}= \cone\{ p \ol p: p\in \C[\bx,\bbx,\by,\bby]_{k,t}\}=\Sigma \cap \C[\bx,\bbx,\by, \bby]_{2k,2t},\quad \Sigma_{=2k,=2t}=\Sigma\cap \C[\bx,\bbx,\by,\bby]_{=2k,=2t}.
\]
We define the sets 
\begin{equation} \label{eq: dps kl}
\begin{split}
	\set_{k, t} := \{ \rho \in \calH^d\otimes \calH^d :\ &\exists L:\C[\bx, \qo \bx, \by, \qo \by]_{2k,2t} \to \C \text{ Hermitian s.t.  }\\
	& L(\bx\bx^*\otimes \by\by^*)=\rho,\\
	& L = 0 \text{ on } \mathcal I(1-\|\bx\|^2,1-\|\by\|^2 )_{2k,2t},\\
	& L \geq 0 \text{ on } \Sigma_{2k,2t} \}.
\end{split}
\end{equation}
Note the inclusion $\set_{k,t}\subseteq \set_{k+t}$. The two regimes that we will be interested in are $k=1$ and $k=t$
since, as we will show in Section \ref{sec: state extension perspective}, the sets $\set_{1,t}$ and $\set_{t,t}$ (for $t\in \N$) coincide with the approximation hierarchies $\dps_{1,t}$ and $\dps_{t,t}$ from \cite{DPS04}.

We will give a more economical reformulation for the positivity condition on $L$ in \cref{propRkl}. For this we first show that for linear functionals $L$ that vanish on the bi-sphere ideal  and satisfy \eqref{eqL02} the following two positivity conditions are equivalent: $L \geq 0$ on $\Sigma_{2k,2t}$ and $L \geq 0$ on $\Sigma_{=2k,=2t}$. That is, we only need to require positivity on \emph{homogeneous} polynomials. 

\begin{lemma} \label{prop: restrict to hom}
Let $L \in \C[\bx,\bbx,\by,\bby]_{2k,2t}^*$ be such that $L =0$ on $\mathcal I(1-\|\bx\|^2,1-\|\by\|^2 )_{2k,2t}$ and $L$ satisfies \cref{eqL02}. Then we have $L \geq 0$ on $\Sigma_{2k,2t}$ if and only if $L \geq 0$ on $\Sigma_{=2k,=2t}$.
\end{lemma}

\begin{proof}
Assume $L \geq 0$ on $\Sigma_{=2k,=2t}$. We show that $L$ is positive on Hermitian squares in $\Sigma_{2k,2t}$. Let  $p  \in \C[\bx,\bbx,\by,\bby]_{k,t}$, we want to show that $L(p \ol p)\ge 0$. For this we decompose $p$ as  $p=p_{00}+p_{01}+p_{10}+p_{11}$, where, for $a,b\in \{0,1\}$,  we group in $p_{ab}$ the terms of $p$ that involve a monomial $\bx^\alpha\bbx^{\alpha'}\by^\beta\bby^{\beta'}$ with
	$|\alpha+\alpha'| \equiv k-a$ modulo 2 and $|\beta+\beta'| \equiv t-b$ modulo 2. 
	Then we have 
	\[
	L(p \ol p) = \sum_{a,b \in \{0,1\}} L(p_{ab}\ol p_{ab}),
	\]
	where  we use \cref{eqL02} to see that $L(p_{ab}\ol p_{a'b'})=0$ if $(a,b)\ne (a',b')$.
Hence it remains to show that $L(p_{ab}\ol p_{ab})\ge 0$ for $a,b\in\{0,1\}$. Write $p_{ab}=\sum c_{\alpha\alpha'\beta\beta'} 	\bx^\alpha\bbx^{\alpha'}\by^\beta\bby^{\beta'}$. We define the polynomial
\[
q_{ab}= \sum c_{\alpha\alpha'\beta\beta'} 	\bx^\alpha\bbx^{\alpha'}\by^\beta\bby^{\beta'}  \|\bx\|^{k - a - |\alpha + \alpha'|}\|\by\|^{t-b-|\beta+\beta'|}.
\]
Note that in each term the powers of $\|\bx\|$ and $\|\by\|$ are by construction both even and nonnegative and therefore the polynomial $q_{ab}$ is homogeneous of degree $k-a$ in $\bx,\bbx$ and of degree $t-b$ in $\by,\bby$. Since $L$ vanishes on the truncated ideal generated by $1-\|\bx\|^2$ and $1-\|\by\|^2$ we have 
\[
L(p_{ab}\ol p_{ab})= L(q_{ab} \ol q_{ab})= L(q_{ab} \ol q_{ab}\|\bx\|^{2a}\|\by\|^{2b})\ge 0,
\]
where the last inequality follows from the fact that $q_{ab} \ol q_{ab}\|\bx\|^{2a}\|\by\|^{2b} \in \Sigma_{=2k,=2t}$.	
\end{proof}

We now proceed to define the analog of \cref{eq: block diag hom t} for the $(k,t)$-setting. Given two integers $r \in \{-k,-k+2,-k+4,\ldots,k\}$ and $s \in \{-t, -t+2,-t+4,\ldots,t\}$ define the set of (exponents of) monomials 
\begin{equation} \label{eq: partition kl}
I^{=k,=t}_{r,s} :=\left\{(\alpha,\alpha',\beta,\beta') \in (\N^d)^4: |\alpha+\alpha'|=k, \ |\alpha|-|\alpha'|=r, \ |\beta+\beta'| = t, \ |\beta|-|\beta'|=s \right\}. 
\end{equation}
Note that we restrict our attention to $r \equiv k \bmod 2$ and $s \equiv t \bmod 2$. If $r,s$ do not satisfy these  conditions  then $I^{=k,=t}_{r,s}=\emptyset$.
  We then have the following semidefinite representation of $\set_{k,t}$.

\begin{proposition}\label{propRkl}
For $k,t \in \N$ we have 
\begin{equation} \label{eq: dpskl SDP}
\begin{split}
	\set_{k,t} = \big\{ \rho \in \calH^d\otimes \calH^d :\ &\exists L:\C[\bx, \qo \bx, \by, \qo \by]_{2k,2t} \to \C \text{ Hermitian s.t.  }\\
	& L(\bx\bx^*\otimes \by\by^*)=\rho,\\
	& L = 0 \text{ on } \mathcal I(1-\|\bx\|^2,1-\|\by\|^2 )_{2k,2t},\\
	& M_{k,\ell}(L)[I^{=k,=t}_{r,s}] \succeq 0 \text{ for all } r \in [-k,k], s\in [-t,t] 
 \big\}.
\end{split}
\end{equation}
\end{proposition}
\begin{proof}
As mentioned above, we may add the constraint \eqref{eqL02} to the program  \eqref{eq: dps kl}. It then follows from \cref{prop: restrict to hom} that we may replace the condition $L \geq 0$ on $\Sigma_{2k,2t}$ with $L \geq 0$ on $\Sigma_{=2k,=2t}$. Finally we observe that the index sets $I^{=k,=t}_{r,s}$ block-diagonalize $M_{=k,=t}(L)$. Indeed, let $p \in \C[\bx,\bbx,\by,\bby]_{=k,=t}$ and write $p = \sum_{r,s} p_{r,s}$ where $p_{r,s}$ is the polynomial corresponding to the terms of $p$ with exponents in $I^{=k,=t}_{r,s}$. Then $p_{r,s} \ol p_{r',s'}$ is a linear combination of monomials of the form 
\[
\bx^{\alpha}\bbx^{\alpha'} \by^{\beta} \bby^{\beta'} \bbx^{\gamma}\bx^{\gamma'} \bby^{\delta} \by^{\delta'},
\]
where we have the following for the degrees in $\bx,\bbx$. By assumption $|\alpha|-|\alpha'|=r$ and $|\gamma|-|\gamma'|=r'$, and therefore the degree in $\bx$ minus the degree in $\bbx$ is $(|\alpha| +|\gamma'|) - (|\alpha'|+|\gamma|) = r-r'$. Similarly, the degree in $\by$ minus the degree in $\bby$ equals $s-s'$. Hence, if $(r,s) \neq (r',s')$, then \cref{eqL02} shows that $L(p_{r,s} \ol p_{r',s'})=0$.
\end{proof}

Finally, we observe the following alternative formulation of the positivity conditions in \cref{eq: dpskl SDP} in terms of the noncommutative moment matrices (cf.~\cref{eq: Mnc}):   
\begin{equation} \label{eq: ppt on L}
M_{k,t}(L)[I^{=k,=t}_{r,s}] \succeq 0
\Longleftrightarrow L\big( (\bx\bx^*)^{\otimes (k+r)/2} \otimes (\bbx\, \bbx^*)^{\otimes (k-r)/2} \otimes (\by\by^*)^{\otimes (t+s)/2} \otimes (\bby \, \bby^*)^{\otimes (t-s)/2}\big) \succeq 0.
\end{equation}
Although less efficient, this  reformulation will permit to connect the program \eqref{eq: dpskl SDP} to the original formulation of the DPS hierarchy $\dps_{1,t}$ (see the proof of Proposition \ref{prop: equality}).

The analog of \cref{lem: convergence set t} holds  for the sets  $\set_{k,t}$:
$$\bigcap_{k,t \ge 1}\set_{k,t}= \bigcap_{t\ge 2} \set_{t,t} =\SEP_d;$$  the argument is similar, based on  standard tools from polynomial optimization (\cref{theomainTchakaloff}).
In fact,   
even the (weaker) sets $\set_{1,t}$ already converge to $\SEP$, i.e.,  we have
\begin{equation}\label{eqR1t}
\bigcap_{t\ge 2} \set_{1,t}=\SEP_d;
\end{equation}
 in other words, in the moment approach it is sufficient to let only the degree in $\by,\bby$ grow. We will show this  in \cref{theoR1t} below, using 
 the tools about matrix-valued polynomial optimization (\cref{theomainmatrix}).

\subsection{Convergence of the sets $\set_{1,t}$ to $\SEP$}\label{sec:convergenceR1t}

We first reformulate the set $\set_{1,t}$ from \cref{eq: dps kl,eq: dpskl SDP} (case $k=1$) in terms of matrix-valued linear functionals $\cL$ on the polynomial space $\C[\by,\bby]$.

\begin{lemma}
For $t \in \N$ we have 
\begin{equation} \label{eq: dps1l SDP}
\begin{split}
	\set_{1,t} = \big\{ \rho \in \calH^d\otimes \calH^d :\ &\exists \cL:\C[\by, \bby]_{2t} \to \H^d \text{ Hermitian s.t.  }\\
	& \cL(\by\by^*)=\rho,\\
	& \cL = 0 \text{ on } \mathcal I(1-\|\by\|^2 )_{2t},\\
	& M_{t}(\cL) \succeq 0 \big\}.
\end{split}
\end{equation}
\end{lemma}

\begin{proof}
Let us use $\widehat \set_{1,t}$ to denote the set defined in \eqref{eq: dps1l SDP}. We show that $\set_{1,t} = \widehat \set_{1,t}$ using the formulation of $\set_{1,t}$ given in \cref{eq: dps kl}.

First consider $\rho \in \set_{1,t}$ and let $L:\C[\bx, \bbx, \by, \bby]_{2,2t} \to \C$ be an associated certificate. Define $\cL:\C[\by,\bby]_{2t} \to \H^d$ by 
\[
\cL(p) = L(\bx \bx^* p) = \big(L(x_i \overline x_j p \big)_{i,j=1}^d  \qquad \text{ for all } p \in \C[\by,\bby]_{2t}.
\]
So $\calL=(L_{ij})_{i,j=1}^d$ with $L_{ij}(p)=L(x_i\ol x_j p)$.
By construction $\cL(\by\by^*)=\rho$. To see that $\cL = 0$ on $\mathcal I(1-\|\by\|^2)_{2t}$, it suffices to observe that, for any $p \in \mathcal I(1-\|\by\|^2)_{2t}$ and any $i,j \in [d]$, the polynomial $x_i \overline x_j p$ lies in $\mathcal I(1-\|\bx\|^2,1-\|\by\|^2)_{2,2t}$. 
To show that $M_t(\cL) \succeq 0$ we use  (the degree truncated version of) \cref{lemMGLpos}. That is, we use that $M_t(\cL) \succeq 0$ is equivalent to $\sum_{i,j \in [d]} L_{ij}(p_i \overline p_j) \geq 0$ for all $ (p_1,\ldots, p_d) \in (\C[\by,\bby]_t)^d$. We have $\sum_{i,j \in [d]} L_{ij}(p_i \overline p_j) = L((\sum_i x_i p_i)(\sum_i x_i p_i)^*) \geq 0$, where the last inequality follows from the fact that $(\sum_i x_i p_i)(\sum_i x_i p_i)^* \in \Sigma_{2,2t}$. This shows that if $\rho \in \set_{1,t}$, then $\rho \in \widehat \set_{1,t}$. 

Conversely, let $\rho \in \widehat \set_{1,t}$ and let $\cL:\C[\by,\bby]_{2t} \to \H^d$ be an associated certificate. We write $\cL(p) = \big(L_{ij}(p)\big)_{i,j \in [d]}$ with $L_{ij} \in \C[\by,\bby]_{2t}^*$ for all $i,j \in [d]$. We define a linear functional $L$ on $\C[\bx,\bbx,\by,\bby]_{2,2t}$ as follows. For a polynomial $p \in \C[\by,\bby]_{2t}$ we set $L(p) = \sum_{i\in[d]} L_{ii}(p)$ and, for each $i,j \in [d]$, we set $L(x_i \overline x_j p) = L_{ij}(p)$. We extend $L$ to $\C[\bx,\bbx,\by,\bby]_{2,2t}$ by setting $L(x^{\alpha} \overline x^{\beta} p) = 0$ for all $\alpha, \beta \in \N^d$ with $(|\alpha|,|\beta|) \not \in \{(0,0), (1,1)\}$, and then extending by linearity. We show that $L$ is a certificate for $\rho \in \set_{1,t}$. First observe that $L(\bx \bx^* \otimes \by \by^*) = \big(L(x_i \overline x_j \by \by^*)\big)_{i,j\in [d]} = \cL(\by \by^*) = \rho$. By construction we have $L((1-\sum_{i} x_i \overline x_i) p)=0$ for all $p \in \C[\by,\bby]_{2t}$. Moreover, by assumption, $\cL((1-\sum_i y_i \overline y_i) p)=0$ for all $p \in \C[\by,\bby]_{2t-2}$. Using the construction of $L$, this implies that $L((1-\sum_i y_i \overline y_i) p) = 0$ for all $p \in \C[\bx,\bbx,\by,\bby]_{2,2t-2}$. Together, this shows that $L=0$ on $\mathcal I(1-\|\bx\|^2,1-\|\by\|^2)_{2,2t}$.

 It remains to show that $L \geq 0$ on $\Sigma_{2,2t}$. Let $p \in \C[\bx,\bbx,\by,\bby]_{1,t}$, we show that $L(p\ol p)\ge 0$. 
 For this, write $p = p_0 + p_1 + p_2$, where $p_0$ has degree $0$ in $\bx,\bbx$, $p_1$ has degree $(1,0)$  in $(\bx,\bbx)$, and $p_2$ has degree  $(0,1)$ in $(\bx,\bbx)$. By definition, $L(p_a \ol p_b)=0$ if $a \neq b$ and thus $L(p\ol p)=L(p_0\ol p_0)+L(p_1\ol p_1)+L(p_2\ol p_2)$. We have $L(p_0 \ol p_0) = \sum_{i=1}^d L_{ii}(p_0\ol p_0)\ge 0$, since $M_t(L_{ii})\succeq 0$ for each $i\in [d]$ as $M_t(\calL)\succeq 0$. Next we show that $L(p_1 \ol p_1) \geq 0$. To do so, write $p_1 = \sum_{i=1}^d x_i q_i$ where $q_i \in \C[\by,\bby]_{t}$ for $i \in [d]$. It then follows that $L(p_1 \ol p_1) = \sum_{i,j=1} ^d L(x_i \overline x_j q_i \ol q_j) = \sum_{i,j=1}^d L_{ij}(q_i\ol q_j)=\langle \cL, \vec {q} \vec{q}^*\rangle \geq 0$ where $\vec{q} = (q_1,\ldots, q_d) \in (\C[\by,\bby]_{t})^d$ and the last inequality follows from $M_{t}(\cL) \succeq 0$ (using~\cref{lemMGLpos}). This also directly implies that $L(p_2\ol p_2)\ge 0$. It follows that $L \geq 0$ on $\Sigma_{2,2t}$ and thus $\rho \in \set_{1,t}$. 
\end{proof}

We can now show the convergence of the sets $\calR_{1,t}$ to $\SEP$. The proof is analogous to that of \cref{lem: convergence set t}, except that it now relies on the results for matrix-valued polynomial optimization (\cref{theomainmatrix,theoTchakmatrix}).

\begin{theorem}\label{theoR1t}
We have $\SEP_d = \bigcap_{t \geq 2} \set_{1,t}$.
\end{theorem}

\begin{proof}
Assume $\rho \in  \bigcap_{t \geq 2} \set_{1,t}$ and for each $t \geq 2$ let $\cL_t$ be a corresponding certificate for membership in $\calR_{1,t}$. Using \cref{pointconv} one can show that the sequence $(\cL_t)_t$ has a pointwise converging subsequence. Let $\cL$ be its limit. It then follows from \cref{theomainmatrix,theoTchakmatrix} that there exists a $K \in \N$, matrices $\Lambda_1,\ldots, \Lambda_K \in \mathcal S_+^d$, and vectors $v_1,\ldots, v_K \in \C^d$ with $\|v_i\|=1$ such that 
\[
\cL(p) = \sum_{k=1}^K \Lambda_k p(v_k) \qquad \text{ for all } p \in \C[\by,\bby]_2.
\]
In particular, 
\[
\rho = \cL(\by \by^*) = \sum_{k=1}^K \Lambda_k \otimes v_k v_k^*,
\]
which shows that $\rho \in \SEP_d$. 
\end{proof}

\subsection{The Doherty-Parrilo-Spedalieri hierarchy: state extension perspective} \label{sec: state extension perspective}

In the previous section we  introduced the sets  $\set_{k,t}$ for  integers $k,t \ge 1$ and we mentioned that there are two regimes of interest: $k=1$ and $k=t$,  leading to the two hierarchies $\set_{1,t}$ and $\set_{t,t}$.  We now show that these hierarchies   in fact coincide with the DPS hierarchies, denoted here as  $\dps_{1,t}$ and $\dps_{t,t}$, that are defined in terms of (one-sided and two-sided) state extensions \cite{DPS04}.
 For ease of notation, we will mostly focus on the first regime and show the equality $\set_{1,t}=\dps_{1,t}$; the arguments naturally adapt to the second regime to show $\set_{t,t}=\dps_{t,t}$. This permits to recover the convergence of the DPS hierarchy to $\sep$ from the corresponding convergence result for the sets $\set_{1,t}$ (Theorem \ref{theoR1t}) obtained via the moment approach.

We begin with  giving the original formulation of the DPS hierarchy $\dps_{1,t}$ in terms of (one-sided) state extensions. To do so, we require a few definitions. 

Given a bipartite state $\rho\in \calH^d\otimes \calH^d$ it is convenient to denote the two vector spaces (aka registers) composing the tensor product space on which $\rho$ acts as $A$ and $B$. Then we may also denote $\rho$ as $\rho_{AB}$. The \emph{partial trace} of $\rho_{AB}$ with respect to the second register is the operator $\rho_A=\Tr_B(\rho)$ that acts on the first register and is defined by tracing out the second register. In the same way, $\rho_B=\Tr_A(\rho)$ is the second partial trace, which acts on the second register and is obtained by tracing out the first one.
Concretely, say $\rho=(\rho_{ij,i'j'})_{i,j,i',j'\in [d]}$ after fixing a basis of $\C^d\otimes\C^d$. Then we have
$$\rho_A=\Tr_B(\rho)=\Big(\sum_{j=1}^d \rho_{ij,i'j}\Big)_{i,i'=1}^d\quad \text{ and } \quad 
\rho_B=\Tr_A(\rho)=\Big(\sum_{i=1}^d\rho_{ij,ij'}\Big)_{j,j'=1}^d.
$$
{The {\em partial transpose} $\rho^{T_B}$ of $\rho$ with respect to the second register $B$ is 
defined by 
\begin{equation}\label{eqPPTB}
(\rho^{T_B})_{ij,i'j'}= \rho_{ij',i'j} \ \text{ for all } i,i'\in [d],\ j,j'\in [d]
\end{equation}
and  $\mathsf{T_B}$ denotes the corresponding transpose operator that acts on $\mathcal H^d\otimes \mathcal H^d$ by taking the partial transpose on the second register, so that  $\mathsf{T_B}(\rho)=\rho^{T_B}$.
 The partial transpose $\rho^{T_A}$ with respect to the first register is  defined analogously by
$(\rho^{T_A})_{ij,i'j'}=\rho_{i'j,ij'}$ for all $i,i',j,j'\in [d]$. Note that 
$\rho^{T_A}= (\rho^{T_B})^T= \overline {\rho^{T_B}}$ if $\rho$ is Hermitian, and thus $\rho^{T_B}\succeq 0$ implies $\rho^{T_A}\succeq 0$.}

Given an integer $t\ge 2$ the construction of the  relaxation $\dps_{1,t}$  relies on the following observation: If $\rho_{AB}$ has a decomposition as in (\ref{decrho}) then one may introduce $t$ copies of  the second register and define the following extended state $\rho_{A B_{[t]} }$ acting on $\C^d\otimes (\C^d)^{\otimes t}$:
\begin{equation}\label{decrhot}
\rho_{A B_{[t]} }:=  \sum_{\ell=1}^r \lambda _l \ x_\ell x_\ell^*\otimes ( y_\ell y_\ell^*)^{\otimes t}.
\end{equation}
There is a natural action of the  symmetric group $\Sym(t)$ on $(\C^d)^{\otimes t}$, defined by $\sigma(v_1\otimes \ldots \otimes v_t)=v_{\sigma(1)}\otimes \ldots \otimes v_{\sigma(t)}$ for $v_1,\ldots,v_t\in\C^d$ and $\sigma\in \Sym(t)$, and extended to the space $(\C^d)^{\otimes t}$ by linearity. Let $\Sym((\C^d)^{\otimes t})$ denote the invariant subspace of $(\C^d)^{\otimes t}$ under this action and  let $\Pi_t$ denote the projection from $ (\C^d)^{\otimes t}$ onto its invariant subspace $\Sym((\C^d)^{\otimes t})$, defined by
$$\Pi_t(w) ={1\over t!}\sum_{\sigma\in \Sym(t)} \sigma(w) \quad \text{ for } w\in  (\C^d)^{\otimes t}.$$
Then, $I_A\otimes \Pi_t$ acts onto $\C^d\otimes  (\C^d)^{\otimes t}$.

We  now present some natural properties that the extended state $\rho_{A B_{[t]}}$ from \eqref{decrhot} satisfies:
\begin{description}
\item[(1)] $\rho_{A B_{[t]}}$ is positive semidefinite.
\item[(2)] 
$\rho_{AB} =\Tr_{B_{[2:t]}}(\rho_{A B_{[t]}})$, where, in $\Tr_{B_{[2:t]}}(\rho(A B_{[t]})$, we trace out the last $t-1$ copies of the second register $B$. 
\item[(3)] $(I_A\otimes \Pi_t)\rho_{A B_{[t]}} (I_A\otimes \Pi_t)=\rho_{AB_{[t]}}$, i.e., $\rho_{A B_{[t]}}$ is symmetric in the last $t$ registers.
\item[(4)]
$I_A\otimes \mathsf{T}_B^{\otimes s} \otimes I_B^{\otimes (t-s)} (\rho_{A B_{[t]}}) \succeq 0$ for any $1\le s\le t$. 
\end{description}
For property (2) we use the fact that each vector $y_\ell$ lies in the unit sphere and the
 last property (4) follows from the fact that $\mathsf{T}_B(yy^*)=(yy^*)^T=\overline y\, \overline y^*$ and thus
$$I_A\otimes \mathsf{T}_B^{\otimes s} \otimes I_B^{\otimes (t-s)} (\rho_{A B_{[t]}}) =
\sum_{\ell=1}^r \lambda_\ell x_\ell x_\ell^* \otimes (\overline y_\ell \ \overline y_\ell^*)^{\otimes s}\otimes (y_\ell y_\ell^*)^{\otimes (t-s)}\succeq 0
$$ if $\rho_{A B_{[t]}}$ satisfies (\ref{decrhot}). Property (4) is  known as the {\em positive partial transpose} (PPT) criterion or as the Peres-Horodecki criterion \cite{Horodecki_1996}.
{Clearly, in view of the symmetry property (3), taking the partial transpose of {\em any} $s$ copies (thus not only the first $s$ ones) among the $t$ copies of the second register preserves positivity.}
The above properties are used to define the  hierarchy $\dps_{1,t}$.

\begin{definition}\label{defDPS}
For an integer $t\ge 2$ the  DPS relaxation of order $t$ is defined as 
\begin{align}
\dps_{1,t} := \big\{ \rho_{AB}\in\calH^d\otimes \calH^d:\ &\exists \rho_{1,t} \text{ Hermitian linear map acting on } \C^d\otimes (\C^d)^{\otimes t} \text{ s.t. } \label{eq:positive} \\
&\mathrm{Tr}_{B_{[2:t]}}(\rho_{1,t}) = \rho_{AB}, \label{eq:partialtrace} \\
&(I_A \otimes \Pi_t) \rho_{1,t} (I_A \otimes \Pi_t) = \rho_{1,t}, \label{eq:symmetry}\\
&I_A \otimes \mathsf T_B^{\otimes s} \otimes I_B^{\otimes (t-s)}(\rho_{1,t}) \succeq 0 \text{ for all } s \in \{0\} \cup [t] \label{eq:PPT}  \big\}.
\end{align}
\end{definition}

\begin{remark}
In the definition of the set $\dps_{1,t}$ only one part of the system is extended, which is why we  refer to this as a {\em one-sided} state  extension. One can define a stronger relaxation of $\sep$ by considering a {\em two-sided} state extension. Given two integers $k,t \ge 2$ one can define $\dps_{k,t}$
as the set of states $\rho_{AB}$ that have  an extension $\rho_{k,t}$ acting on
$(\C^d)^{\otimes k}\otimes (\C^d)^{\otimes t}$, which satisfies the appropriate analogs of the above properties (1)-(4).
One may consider in particular the case $k=t$, leading to the sets $\dps_{t,t}$ that satisfy $$\SEP_d\subseteq \dps_{t,t}\subseteq \dps_{1,t}.$$
\end{remark}

Doherty, Parrilo and Spedaglieri \cite{DPS04} show that the  relaxations $\dps_{1,t}$ converge to $\sep$. 

\begin{theorem}[{\cite{DPS04}}]\label{theoDPS}
We have $\SEP_d\subseteq \dps_{1,t+1}\subseteq \dps_{1,t}$ and $\SEP_d=\bigcap_{t\ge 1} \dps_{1,t}$. As a consequence, we also have $\SEP_d=\bigcap_{t\ge 1}\dps_{t,t}$.
\end{theorem}

We now    show  that equality $\set_{1,t} = \dps_{1,t}$ holds for all $t \in \N$.  Therefore, Theorem  \ref{theoDPS} follows directly from Theorem \ref{theoR1t}.
Using similar arguments one can also show that $\dps_{k,t}=\set_{k,t}$ and thus $\dps_{t,t}=\set_{t,t}$.

\begin{proposition} \label{prop: equality}
For any integer $t \geq 2$ we have $\set_{1,t} = \dps_{1,t}$. 
\end{proposition}

\begin{proof}

Assume first  $\rho_{AB}\in \set_{1,t}$, with certificate $L$ satisfying \cref{eq: dps kl} (with $k=1,\ell=t$). 
	We claim that  $\rho_{1,t}:=  L(\bx\bx^* \otimes (\by\by^*)^{\otimes t})$ is a certificate for membership of $\rho_{AB}$ in $ \dps_{1,t}$. Indeed, \cref{eq:partialtrace} holds since $\mathrm{Tr}_{B_{[2:t]}}(\rho_{1,t}) =L(\bx\bbx^*\otimes \by\bby^*)= \rho_{AB}$ 	follows using the bi-sphere ideal condition on $L$. 		The symmetry condition in \cref{eq:symmetry} holds for  $\rho_{1,t}$ since  $L$ acts on commutative polynomials, and  the   PPT condition in \cref{eq:PPT} 
		holds for  $\rho_{1,t}$ as a consequence of the positivity condition: $L\ge 0$ on $\Sigma_{2,2t}$.
		
Conversely, assume  that $\rho_{AB}\in \dps_{1,t}$, with state $\rho_{1,t}$ as certificate  satisfying (\ref{eq:positive})-(\ref{eq:PPT}). We construct a linear functional $L$ acting on $\C[\bx,\bbx,\by,\bby]_{2,2t}$ that certifies membership of $\rho_{AB}$ in $\set_{1,t}$, i.e., satisfies the program (\ref{eq: dps kl}) (with $k=1$).
In a first step we set
	\begin{equation}\label{eqrho1t}
	L(\bx\bx^* \otimes (\by\by^*)^{\otimes t}) := \rho_{1,t}.
	\end{equation}
	In other words we set
		$$L(x_i\overline x_{i'} y_{j_1}\overline y_{j'_1} \cdots y_{j_t}\overline y_{j'_t}):= (\rho_{1,t})_{i \avec{j}, i'\avec{j'}}$$ 
		for any $i,i'\in [d]$ and 
	$\avec{j}=(j_1,\ldots,j_t), \avec{j'}=(j'_1,\ldots,j'_t)\in [d]^t$.
Using the symmetry condition (\ref{eq:symmetry}),  it follows that this definition does not depend on the order of the variables $y_j$ (or $\overline y_j$).	

	Indeed, by \cref{eq:symmetry} we know that 
	\begin{equation*} \label{eq: perm invariance}
	(\rho_{1,t})_{i\avec{j}, i' \avec{j'}} = (\rho_{1,t})_{i\sigma(\avec{j}), i' \tau(\avec{j'})} \text{ for all } i,i' \in [d], \avec{j},\avec{j'} \in [d]^t
	\end{equation*}
	for all permutations $\sigma, \tau\in\Sym(t)$, where  $\sigma(\avec{j}) =  (j_{\sigma(1)}, j_{\sigma(2)},\ldots, j_{\sigma(t)})$ for  $\avec{j}=(j_1,\ldots,j_t)$. 
	which shows that 
	$$
	L(x_i\overline x_{i'} y_{j_1}\overline y_{j'_1} \cdots y_{j_t}\overline y_{j'_t})
	= 	L(x_i\overline x_{i'} y_{j_{\sigma(1)}}\overline y_{j'_{\tau(1)}} \cdots y_{j_{\sigma(t)}}\overline y_{j'_{\tau(t)}}).
	$$
	This shows that $\rho_{1,t}$ defines a linear functional $L$ acting on polynomials with degree $1$ in $\bx$, degree $1$ in $\bbx$, degree $t$ in $\by$, and degree $t$ in $\bby$. We now show how to extend this linear functional $L$ to $\C[\bx,\bbx,\by,\bby]_{2,2t}$ in such a way that it becomes a certificate for $\rho_{AB} \in \set_{1,t}$. 
	
	First we extend $L$ to all monomials $\bx^\alpha\bbx^{\alpha'}\by^\beta\bby^{\beta'}$ with degree at most 2 in $\bx,\bbx$ and degree at most $2t$ in $\by,\bby$. For this we set  
	\begin{equation}\label{eq:extend L}
	L(\bx^\alpha\bbx^{\alpha'}\by^\beta\bby^{\beta'}):= 0 \ \text{ if } |\alpha|\ne |\alpha'| \text{ or } |\beta|\ne |\beta'|.
	\end{equation}
Otherwise, $|\alpha+\alpha'|,|\beta+\beta'|$ are even and we set
\[
L(\bx^\alpha \bbx^{\alpha'} \by^{\beta} \bby^{\beta'}) := L(\|\bx\|^{2-|\alpha+\alpha'|} \|\by\|^{2t-|\beta+\beta'|} \bx^\alpha \bbx^{\alpha'} \by^{\beta} \bby^{\beta'}).
\]
By construction, $L$ is Hermitian (since $\rho_{1,t}$ is Hermitian) and $L$ vanishes on  $\mathcal I(1-\|\bx\|^2,1-\|\by\|^2)_{2,2t}$.

It remains to show that $L\ge 0$ on $\Sigma_{2,2t}$.  In view of Lemma \ref{prop: restrict to hom}, it suffices to show that  $L \geq 0$ on $\Sigma_{=2,=2t}$, or, equivalently, that the moment matrix $M_{=1,=t}(L)$, indexed by monomials $\bx^\alpha\bbx^{\alpha'}\by^\beta\bby^{\beta'}$ with  $|\alpha+\alpha'|=1$ and $|\beta+\beta'|=t$, is positive semidefinite.
	 In view of \eqref{eq:extend L} the matrix  $M_{=1,=t}(L)$ is  block-diagonal with respect to the partition of its index set according to the value of $(|\alpha|, |\beta|)$, i.e., according to the partition of $I^{=1,=t}=\bigcup_{r,s} I^{=1,=t}_{r,s}$ defined in \cref{eq: partition kl} with $-1\le r\le 1, -t\le s\le t$ and $r\equiv 1, s\equiv t$ modulo 2.
	  So we are left with the task of showing that all diagonal blocks $M_{=1,=t}(L)[I^{=1,=t}_{r,s}]$ are positive semidefinite.
For this we use \cref{eq: ppt on L} to obtain that 
$$
M_{=1,=t}(L)[I^{=1,=t}_{r,s}]\succeq 0\Longleftrightarrow L({\bx\bx^*}^{\otimes ({r+1\over 2}) } \otimes {\bbx\bbx^*}^{\otimes ({1-r\over 2})}    \otimes {\by\by^*}^{\otimes ({t+s\over 2})} \otimes {\bby\bby^*}^{\otimes ({t-s\over 2})}) \succeq 0.
$$
This holds for all $r,s$ such that $-1\le r\le 1, -t\le s\le t, r\equiv 1, s\equiv t$ modulo 2 if and only if 
$$
L(\bx\bx^* \otimes {\by\by^*}^{\otimes (t-s')} \otimes {\bby\bby^*}^{\otimes s'}) \succeq 0,\ \ 
L(\bbx\bbx^* \otimes {\by\by^*}^{\otimes (t-s')} \otimes {\bby\bby^*}^{\otimes s'} )\succeq 0$$
for all $s'\in \{0\}\cup [t]$ (setting $s'={t-s\over 2}$). In view of \cref{eqrho1t}  we obtain that 
\begin{align} \label{eq: PPT on L}
L(\bx\bx^* \otimes {\by\by^*}^{\otimes (t-s')} \otimes {\bby\bby^*}^{\otimes s'}) =I_A\otimes I_B^{\otimes (t-s')} \otimes \mathsf T_B^{\otimes s'}(\rho_{1,t})
\end{align}
{and, since $L$ is Hermitian,} \begin{align*}
L(\bbx\bbx^* \otimes {\by\by^*}^{\otimes (t-s')} \otimes {\bby\bby^*}^{\otimes s'} )
= \overline{L(\bx\bx^* \otimes {\bby\bby^*}^{\otimes (t-s')} \otimes {\by\by^*}^{\otimes s'}) }
=\overline{ I_A\otimes \mathsf T_B^{\otimes (t-s')} \otimes I_B^{\otimes s'} (\rho_{1,t})}.
\end{align*}
Therefore, the positive semidefiniteness of all the diagonal blocks composing the matrix $M_{=1,=t}(L)$ follows from the PPT condition (\ref{eq:PPT}) combined with the symmetry condition (\ref{eq:symmetry}) and the fact that the conjugate of a Hermitian positive semidefinite matrix remains positive semidefinite.
\end{proof}

\begin{remark}\label{rem: level 2 implies ppt}
Note that it follows from relation \eqref{eq: PPT on L} in the above proof that, 
 if $\rho = L(\bx \bx^* \otimes \by \by^*)$ where the linear functional $L$ satisfies  $L \geq 0$ on $\Sigma_{2,2}$, then $\rho$ satisfies the PPT condition \eqref{eq:PPT}. In particular this implies that the PPT condition is contained in the definition of the parameter $\xidsep{t}$: if the program (\ref{eqxit}) defining $\xidsep{t}$ is feasible (for $t \geq 2$), then $\rho$ satisfies the PPT condition. 
\end{remark}

\appendix 

\section{Deriving the complex results from their real analogs} \label{sec:AppA}
In this appendix we show how the proofs of \cref{theomainmatrix,theomainTchakaloff} can be obtained from their real versions in \cite{Pu93,Tchakaloff,CimpricZalar}. We begin with recalling in Section \ref{AppPrel} the links between the main properties of the complex objects introduced in the paper and their real analogs. Then we give the proof of Theorem \ref{theomainTchakaloff}  in Section \ref{AppProofPutinar} and of Theorem \ref{theomainmatrix} in Section \ref{AppProofCimpric}.

\subsection{Preliminaries on changing variables from complex to real}\label{AppPrel}
\paragraph{Vectors and matrices.} Throughout we set $\i=\sqrt{-1}\in \C$. 
Then any complex scalar $x \in \C$ can be written (uniquely) as  $x = x_\Re + \i x_\Im$, where $x_\Re:=\Re(x)$ and $x_\Im:=\Im(x)$ denote, respectively, the real and imaginary parts of $x$.
This notation  extends  to vectors and matrices by letting the maps $\Re(\cdot)$ and $\Im(\cdot)$ act entrywise.
Any vector 
$x \in \C^n$ can be written $x =  x_\Re+ \i x_\Im$ with  $x_\Re:=\Re(x) , x_\Im := \Im(x)\in \R^n$. This gives a bijection 
\begin{equation} \label{CRbij}
	\phi:\C^n \to \R^n\times \R^n ~;~  \bx \mapsto (\xr,\xim).
\end{equation}
Similarly, for a complex matrix $G \in \C^{m \times m'}$ set $G_\Re := \Re(G), G_\Im := \Im(G)\in \R^{m\times m'}$ and define the $2m\times 2m'$ real matrix
\begin{equation}\label{RCMats}
G^\R:=
\begin{bmatrix}  
	G_\Re&  -G_\Im\\
	G_\Im& G_\Re
\end{bmatrix}.
\end{equation}
Then $G\in \C^{m\times m}$ is  Hermitian, i.e.,  $G^* = G$, if and only if  $G_\Re  = G_\Re^T$ and $G_\Im^T = -G_\Im$. Moreover, for $G\in \C^{m\times m}$  Hermitian and $w\in \C^m$  we have the  identity 
\begin{equation}\label{RCMatsw}
w^*G w = (w_\Re - \i w_\Im )^T(G_\Re + \i G_\Im)(w_\Re + \i w_\Im ) = 
\begin{bmatrix}  
	w_\Re^T & w_\Im ^T   
\end{bmatrix}
\begin{bmatrix}  
	G_\Re &  -G_\Im\\
	G_\Im & G_\Re
\end{bmatrix}
\begin{bmatrix}  
	w_\Re \\
	w_\Im    
\end{bmatrix},
\end{equation}
which implies the well-known equivalence
\begin{equation*} \label{RCMatsequiv}
G \succeq 0 
\iff
G^\R=
\begin{bmatrix}  
	G_\Re&  -G_\Im\\
	G_\Im& G_\Re
\end{bmatrix}
\succeq 0.
\end{equation*}

\paragraph{Polynomials.} Polynomials in $\Cxox$ with complex variables $\bx \in \C^n$ can be transformed into polynomials in $\Rxrxi$ with real variables $\xr,\xim\in \R^n$, via the change of variables $\bx = \xr + \i\xim$. In this way, any $p \in \Cxox$ corresponds to a unique pair of real polynomials
\begin{align*}
p_{\Re}(\xr,\xim) :=&  \Re(p(\xr + \i \xim,\xr - \i \xim))   \in \Rxrxi,\\ 
p_{\Im}(\xr,\xim) :=&  \Im(p(\xr + \i \xim,\xr - \i \xim))  \in \Rxrxi
\end{align*}
satisfying the identity
\begin{equation}\label{pCR}
p(\bx,\bbx)=p(\xr+\i \xim, \xr-\i \xim)= p_\Re(\xr,\xim)+ \i p_\Im(\xr, \xim).
\end{equation}
 Note that the degrees are preserved: $\deg_{\bx,\bbx}(p) = \max\{\deg_{\xr,\xim}(p_\Re), \deg_{\xr,\xim} (p_\Im) \}$.
A polynomial  $p$ is Hermitian, i.e., $\qo p = p$, if and only if $p_{\Im} = 0$.
Hence, the map 
\begin{align}
\Re : & \ \Cxoxh \to \Rxrxi ~;~ p(\bx,\bbx) \mapsto  p_{\Re}(\xr,\xim) \label{RCHpolybij}
\end{align}
is injective. 
This map is also surjective: take any $f \in \Rxrxi$ and define the polynomial $p(\bx,\bbx):= f(\frac{\bx +\bbx}{2},\frac{\bx -\bbx}{2\i})\in \C[\bx,\bbx]$, then $p$ is Hermitian and satisfies $f = p_\Re$. Finally, since any $p\qo p$ is Hermitian we have 
\begin{equation*}\label{eqsos}
\Re(p\qo p) = p_\Re^2 + p_\Im^2.
\end{equation*} 
Hence sums of Hermitian squares in $\Cxox$ are mapped to real sums of squares in $\Rxrxi$ and vice versa. 
 
\paragraph{Polynomial matrices.} For vectors and matrices with polynomial entries in $\Cxox$, the maps $\Re(\cdot)$ and $\Im(\cdot)$ act entrywise.
Additionally, for a  polynomial matrix $G \in \Cxox^{m\times m'}$,  we can define the real polynomial matrix $G^{\R} \in \Rxrxi^{2m\times 2m'}$  using relation (\ref{RCMats}), where $G_\Re,G_\Im$ are defined entrywise: if $G=(G_{ij}) $ then
$G_\Re=((G_{ij})_{\Re})$ and $G_\Im=((G_{ij})_\Im)$. Then $G$ is Hermitian if and only if $G^{\R}$ is symmetric and as we next observe this correspondance extends  to sums of squares.

\begin{lemma}\label{RCGpsd}
	Let $G \in \Cxox^{m\times m}$ be a polynomial matrix and let $G^\R\in \R[\xr,\xim]^{2m\times 2m}$ be the corresponding real polynomial matrix defined via (\ref{RCMats}).  Then $G$ is a Hermitian SoS-polynomial matrix if and only if $G^{\R}$ is a (real) SoS-polynomial matrix. 
\end{lemma}
\begin{proof}
	Assume $G$ is a Hermitian  SoS-polynomial matrix. Let  $G = UU^*$ with $U \in \Cxox^{m\times k}$.
	Applying the change of variables from complex to real we get 
	$$
	G(\bx,\bbx) = G_\Re(\xr,\xim) + \i G_\Im(\xr,\xim)
	=	U(\xr + \i \xim,\xr - \i \xim)U^*(\xr + \i \xim,\xr - \i \xim) 
	$$
	$$
	= (U_{\Re} + \i U_{\Im})(U_{\Re}^T - \i U_{\Im}^T) \\
	= U_{\Re} U_{\Re}^T + U_{\Im}U_{\Im}^T + \i \big(U_{\Im}U_{\Re}^T - U_{\Re}U_{\Im}^T \big).
	$$
	This implies $G_\Re = U_{\Re} U_{\Re}^T + U_{\Im}U_{\Im}^T$ and $G_\Im = U_{\Im}U_{\Re}^T - U_{\Re}U_{\Im}^T$.
	and thus
	\begin{align*}
	G^\R :=&	
	\begin{bmatrix}  
	G_{\Re} &  -G_{\Im} \\
	G_{\Im} &  G_{\Re}  
	\end{bmatrix}
	=
	\begin{bmatrix}  
	U_{\Re} U_{\Re}^T + U_{\Im}U_{\Im}^T & -(U_{\Im}U_{\Re}^T - U_{\Re}U_{\Im}^T)\\
	U_{\Im}U_{\Re}^T - U_{\Re}U_{\Im}^T & U_{\Re} U_{\Re}^T + U_{\Im}U_{\Im}^T  
	\end{bmatrix}\\
	=&
	\begin{bmatrix}  
	U_{\Re} &  -U_{\Im} \\
	U_{\Im} &  U_{\Re}  
	\end{bmatrix}
	\begin{bmatrix}  
	U^T_{\Re} &  U^T_{\Im} \\
	-U^T_{\Im} &  U^T_{\Re}  
	\end{bmatrix}
	=: U^\R(U^\R)^T,
	\end{align*}
	which shows $G^\R$ is an SoS-polynomial matrix. 
	The converse result follows from retracing the above steps.
\end{proof}

\paragraph{Quadratic modules.}
Given a set $S \subseteq \Cxoxh$ of Hermitian polynomials we define  its real analog  by applying the map $\Re(\cdot)$ from (\ref{RCHpolybij}) elementwise to the set $S$ and set
\begin{equation}\label{eqSRe}
\Sre := \Re(S) = \{p_{\Re}: p \in S \}  \subseteq \Rxrxi.
\end{equation}
Given a Hermitian  
polynomial matrix $G \in \Cxox^{m \times m}$ we  define the  set of Hermitian polynomials
\begin{equation*}
S^{G} := \{ w^* G  w: w \in \C^m \} \subseteq \Cxoxh
\end{equation*}
and, for  the corresponding real symmetric matrix $G^\R\in \R[\xr,\xim]^{2m\times 2m}$ defined  via (\ref{RCMats}), we define the set of real polynomials   
\begin{equation*}
S^{G^\R}  :=  \{ (w_\Re, w_\Im)^T G^{\R} (w_\Re, w_\Im): w \in \C^{m} \} \subseteq \Rxrxi.
\end{equation*}
These two sets satisfy the expected correspondance:
 $$ 
 S^{G^\R}= \Re(S^{G}),$$ since,   in view of relation (\ref{RCMatsw}),  we have 
 $\Re(w^*Gw) = (w_\Re, w_\Im)^T G^{\R} (w_\Re, w_\Im)$ for all $w\in\C^m$.

This  correspondance extends to   the (real part of the) truncated complex quadratic module $\mathcal M(S)_{2t}$ generated by $S\subseteq \C[\bx,\bbx]^h$  and the truncated real quadratic module generated by the corresponding set $S_\Re\subseteq  \R[\xr,\xim]$ via (\ref{eqSRe}), which is denoted here as $\mathcal M^\R(S_\Re)_{2t}$ and defined by
$$
\calM^\R(\Sre)_{2t} := \text{cone}\{g_{\Re}f^2:  f \in \Rxrxi ~,~ g \in S,~~\deg(g_{\Re}f^2) \leq 2t  \}.
$$
Namely, we have 
$$ \Re(\calM(S)_{2t})= \calM^\R(\Sre)_{2t}.$$
Indeed we have  $\Re(g p \qo p) = g_{\Re}(p_\Re^2 + p_\Im^2)$ and  the next relation, collected for further reference:
\begin{equation} \label{RCHquadbij} 
g p \qo p \in  \calM_{2t}(S)  \iff \Re(g p \qo p) =g_\Re(p_\Re^2+p_\Im^2) \in \calMr[2t](\Sre) \text{ for all } p\in \Cxox,~g \in S.
\end{equation}

\begin{lemma} \label{RCArch}
	For any $S \subseteq \Cxoxh$, the (complex) quadratic module $\calM(S)$ is Archimedean if and only if the real quadratic module $\calMr[](\Sre)$ is Archimedean.
\end{lemma}
\begin{proof}
	Directly from \cref{RCHquadbij} since, for any scalar $R\in\R$, 
	$R^2 -\bx^*\bx \in \calM(S)$  if and only if 	$\Re(R^2-\bx^* \bx)=R^2 - \xr^T\xr +\xim^T\xim  \in \calMr[](\Sre)$.
\end{proof}

\paragraph{Positivity domains and measures.}
There is  a natural correspondance between the complex positivity domain $\sD(S)$  of a set $S \subseteq \Cxoxh$ and the real positivity domain of the corresponding set  $S_\Re\subseteq \R[\xr,\xim]$, which is denoted $ \sDr(\Sre)$ and defined by
$$
\sDr(\Sre) := \{ (w_\Re,w_\Im) \in \mathbb{R}^{2n}: g_{\Re}(w_\Re,w_\Im) \geq 0 ~\forall~ g\in S \}.
$$

Indeed, in view of \cref{RCHpolybij} and using the complex/real bijection map $\phi$ from (\ref{CRbij}), we have  $$\sDr(\Sre) = \phi(\sD(S)).$$
{Given a measure $\mu^{\R}$  on $\R^{2n}$ we  define the complex measure $\mu$ on $\C^n$ as $\mu=\mu^{\R}\circ \phi$, the push-forward of $\mu^{\R}$ by the map $\phi^{-1}$, so that  
\begin{equation}\label{eqmu}
\int_{\C^n} p(\bx)d\mu =\int_{\R^{2n}} p\circ \phi^{-1}(\xr,\xim) d\mu^{\R}
= \int_{\R^{2n}} p_\Re(\xr,\xim)d\mu^\R + \i \int_{\R^{2n}}p_\Im(\xr, \xim)d\mu^\R
\end{equation}
for any $p\in\C[\bx,\bbx]$ (using (\ref{pCR})).}

If $\mu^{\R}$ is supported by  $\sDr(\Sre)$ (i.e., $\mu^\R(\R^{2n} \setminus \sDr(\Sre)) = 0$),  then $\mu$ is  supported by $\sD(S)$ (i.e., $\mu(\C^n \setminus \sD(S)) = 0$).
This follows from the fact that $\phi(\C^n \setminus \sD(S)) = \R^{2n} \setminus \sDr(\Sre)$.
	 
\paragraph{Linear functionals.}
For a linear functional $L: \Cxox \to \C$ we have $L(p) = \Re(L(p)) + \i \Im(L(p))$ for all $p\in \Cxox$. Recall that $L$ is Hermitian if $\qo{L(p)} = L(\qo p)$. For any Hermitian $L$,   we can define a real linear functional $\Lr : \Rxrxi \to \R$ by
\begin{equation}\label{llrdef}
 \Lr (f) := L\big(f\big(\frac{\bx + \bbx}{2},\frac{\bx - \bbx}{2\i}\big)\big) ~\text{ for any }~ f\in \Rxrxi.
\end{equation}
For a Hermitian polynomial $p\in \Cxoxh$,  by \cref{RCHpolybij} we have  $p_\Re(\frac{\bx + \bbx}{2},\frac{\bx - \bbx}{2\i}) = p(\bx,\bbx)$ and thus 
\begin{equation}\label{llr}
	L(p) = \Lr (p_{\Re})\quad \text{ for any } p\in \C[\bx,\bbx]^h.
\end{equation}
Then for any   $p\in \Cxox$ we have
\begin{equation}\label{eqLLRe}
L(p) = L\big(p_\Re\big(\frac{\bx + \bbx}{2},\frac{\bx - \bbx}{2\i}\big)\big) + \i L\big(p_\Im\big(\frac{\bx + \bbx}{2},\frac{\bx - \bbx}{2\i}\big)\big) = \Lr (p_{\Re}) + \i \Lr (p_{\Im}).
\end{equation}
 In particular, we have $L(p \qo p) = \Lr (p_{\Re}^2 + p_{\Im}^2)$ for any $p\in\C[\bx,\bbx]$. This implies that $L$ is positive (on Hermitian sums of squares) if and only if $\Lr$ is positive (on real sums of squares). Since $\Re(\cdot)$ preserves degrees, the restriction of $\Lr$ to $\Rxrxi_t$ corresponds to the restriction of $L$ to $ \Cxox_t$.
 This gives the following correspondance for truncated quadratic modules.
 
\begin{lemma}\label{RCPosQuad}
	Given  $S \subseteq \Cxoxh$,  a Hermitian linear map $L\in\Cxox \to \C$, the corresponding  set $S_\Re\subseteq \R[\xr,\xim]$ and the corresponding real linear map $L^\R\in \R[\xr,\xim]\to\R$ we have 
	$$L \geq 0 ~\text{on}~ \calM(S)_{2t} \iff \Lr \geq 0 ~\text{on}~ \calMr[](\Sre)_{2t}\quad \text{ for any } t\in \N\cup\{\infty\}.$$
\end{lemma}
\begin{proof} 
This  follows form the linearity of $L$ and $\Lr$	since, by \cref{RCHquadbij}, $g p \qo p \in  \calM(S)$ if and only if $ \Re(gp  \qo p)=g_{\Re}(p_\Re^2 + p_\Im^2)  \in \calMr[](\Sre)$ and, by \cref{llr}, $L(g p \qo p) = \Lr(g_{\Re}(p_\Re^2 + p_\Im^2))$.
\end{proof}

Finally note that an  evaluation functional $L_{w}$ at a point $w\in \C^d$ corresponds to the  evaluation functional $L_{(w_\Re,w_\Im)}$ at the point $(w_\Re,w_\Im)\in \R^{2d}$ since,  for every $p\in \Cxox$, we have
  \begin{equation*}\label{eqLLw}
 L_{w}(p) = p(w,\qo w) = p_\Re(w_\Re,w_\Im) + \i p_\Im(w_\Re,w_\Im) = \Lr_{(w_\Re,w_\Im)}(p_\Re) + \i \Lr_{(w_\Re,w_\Im)}(p_\Im).
 \end{equation*}

\paragraph{Matrix-valued linear functionals.}
Consider a complex matrix-valued linear map 
$$
\calL:\Cxox \to \C^{m\times m}, \quad p\mapsto \calL(p) := \big( L_{ij}(p) \big)_{i,j\in[m]},
$$
where  each $L_{ij}:\Cxox \to \C$ is  scalar-valued. Then  $\calL$ is Hermitian if and only if, for all $p\in \C[\bx,\bbx]$, we have  $\calL(\qo p) = \calL(p)^*$, i.e., 
$$
\Big( \Re(L_{ij}(\qo p)) + \i \Im(L_{ij}(\qo p))  \Big)_{i,j=1}^m 
=
\Big( \Re(L_{ji}(p)) - \i \Im(L_{ji}( p))  \Big)_{i,j=1}^m
$$
or, equivalently, 
$\Re(L_{i,j}(\qo p)) = \Re(L_{j,i}(p))$ and $\Im(L_{i,j}(\qo p)) = -\Im(L_{j,i}( p))$  for all $i,j \in [m]$. This implies that if $\calL$ is Hermitian and $p$ is Hermitian then the complex matrix $\calL(p)$ is Hermitian. 
 
Assume $\calL$ is Hermitian. Then we define  the real matrix-valued linear functional 
$$
\calLr: \Rxrxi \to \R^{2m \times 2m}, \quad f\in \R[\xr,\xim]\mapsto \calL^\R(f)
$$
\begin{equation} \label{calLrdef} 
\calL^\R(f):=\Big( \calL\big(f\big(\frac{\bx +\bbx}{2},\frac{\bx -\bbx}{2i}\big)\big)  \Big)^\R 
=
\begin{bmatrix}  
	\Re( \calL(f(\frac{\bx +\bbx}{2},\frac{\bx -\bbx}{2i})) ) &  -\Im( \calL(f(\frac{\bx +\bbx}{2},\frac{\bx -\bbx}{2\i})) ) \\
	\Im( \calL(f(\frac{\bx +\bbx}{2},\frac{\bx -\bbx}{2\i})) ) & \Re( \calL(f(\frac{\bx +\bbx}{2},\frac{\bx -\bbx}{2\i})) )  
\end{bmatrix}.
\end{equation}
Since $f(\frac{\bx +\bbx}{2},\frac{\bx -\bbx}{2\i})$ is Hermitian it follows that  
$-\Im( \calL(f(\frac{\bx +\bbx}{2},\frac{\bx -\bbx}{2\i})) )
 =
\Im( \calL(f(\frac{\bx +\bbx}{2},\frac{\bx -\bbx}{2\i})) )^T$. Hence $\calLr$ takes its values in the cone $\calS^{2m}$ of symmetric matrices.

\begin{lemma} \label{calLgpp}
Given a Hermitian linear map $\calL:\Cxox \to \C^{m\times m}$ and the corresponding map $\calL^\R$ from (\ref{calLrdef}),  $g \in \Cxoxh$ and  $p\in \Cxox$ we have the following equivalence
	$$
	\calL(g p \qo p) \succeq 0 \iff  	\calLr(g_\Re(p_\Re^2 +p_\Im^2)) \succeq 0.
	$$
\end{lemma}

\begin{proof}
From \cref{RCMats,RCHpolybij,calLrdef} we obtain that 
$$
0 \preceq
\calL(g p \qo p)  \iff 
0 \preceq
\begin{bmatrix}  
	\Re(\calL(g p \qo p)) & -\Im(\calL(g p \qo p))\\
	\Im(\calL(g p \qo p)) & \Re(\calL(g p \qo p))   
\end{bmatrix}
= \calLr(g_\Re (p_\Re^2 +p_\Im^2)) , 
$$
because $g p \qo p$ is Hermitian.
\end{proof}

\begin{corollary} \label{calLgppcor}
Given $S\subseteq \C[\bx,\bbx]^h$, a	 Hermitian linear map $\calL:\Cxox \to \C^{m\times m}$ is positive on $\calM(S)$ if and only if the corresponding real linear map $\calLr$ from (\ref{calLrdef}) is positive on $\calMr[](\Sre)$.
\end{corollary}

\subsection{Deriving Theorem \ref{theomainTchakaloff} from its real analog} \label{AppProofPutinar}
We can now derive Theorem \ref{theomainTchakaloff}, which  we stated for complex polynomials, from the following well-known results for real polynomials  from \cite{Pu93} and \cite{Tchakaloff}.

\begin{theorem} \label{theomainTchakaloffREAL}
	Let $S \subseteq  \R[\bx]$ such that
	$\calMr[](S)$ is Archimedean and 
	let $L : \R[\bx] \to \R$ be a linear map that is nonnegative on $\calMr[](S)$.
	Then the  following holds.
	\begin{itemize}
		\item[(i)] (Putinar \cite{Pu93}) There exists a measure  $\mu^\R$ that is supported on $\sD^\R(S)$, the positivity domain of $S$ 	defined by
		$$
		\sD^\R(S) = \{a \in \R^n: g(a) \geq 0 \text{ for all } g \in S \},
		$$
		 such that $L(f)=\int f d\mu$ for all $f \in  \R[\bx]$.
		\item[(ii)] (Tchakaloff~\cite{Tchakaloff})		
		For any  integer $k \in \N$ there exists a linear map $ \wh L: \R[\bx] \to \R$ 
		such that 	
	$$	\wh L (f)= L(f)\ \forall f\in \R[\bx]_k  \quad \text{ and } \quad \wh L =  \sum_{\ell=1}^{K} \lambda_\ell L_{a^\ell}$$
		for some  integer $K\ge 1$, scalars  $\lambda_1,\ldots, \lambda_K>0$ and vectors $a^1,\ldots,a^K\in\sD^\R(S)$.
	\end{itemize}
	
\end{theorem}

We now  indicate how to derive Theorem \ref{theomainTchakaloff} from Theorem \ref{theomainTchakaloffREAL}. 
For this consider  $S\subseteq \C[\bx,\bbx]^h$ and a linear map $L:\C[\bx,\bbx]\to\C$. Assume $\calM(S)$ is Archimedean and $L\ge 0$ on $\calM(S)$.
We consider the  set $S_\Re\subseteq \R[\xr,\xim]$ of real polynomials defined via (\ref{eqSRe}) and the associated linear map $L^\R:\R[\xr,\xim]\to\R$ defined via (\ref{llrdef}).
By Lemma \ref{RCArch} the quadratic module $\calM^\R(S_\Re)$ is Archimedean and, by Lemma \ref{RCPosQuad}, $L^\R\ge 0$ on $\calM^\R(S_\Re)$.
Hence we can apply Theorem \ref{theomainTchakaloffREAL} to $S_\Re$ and $L^\R$.

By Theorem \ref{theomainTchakaloffREAL} (i), there exists a (real) measure $\mu^\R$ that is supported by $\sD^\R(S_\Re)$ and satisfies $L^\R(f)=\int fd\mu^\R$ for all $f\in\R[\xr,\xim]$.
Consider the (complex) measure $\mu$ defined by relation (\ref{eqmu}), which is therefore supported by the set $\sD(S)$.
We claim that $\mu$ is a representing measure for $L$. Indeed, for $p\in\C[\bx,\bbx]$, using (\ref{eqLLRe})  we have
$$L(p) = L^\R(p_\Re)+\i L^\R(p_\Im) = \int p_\Re d\mu^\R +\i \int p_\Im d\mu^\R = \int pd\mu.$$
This completes the proof of Theorem \ref{theomainTchakaloff} (i).
We now derive its part (ii). 

Fix an integer $k\in\N$. By Theorem \ref{theomainTchakaloffREAL} (ii), there exists $\wh L:\R[\xr,\xim]\to\R$ such that $\wh L(f)=L^\R(f)$ for all $f\in \R[\xr,\xim]_k$ and $\wh L=\sum_{\ell=1}^K\lambda_\ell L_{a^\ell}$ for some $K\in\N$, $\lambda_\ell>0$ and $a^\ell \in\sD^\R(S_\Re)$.
Define the complex linear map
$\widetilde L:\C[\bx,\bbx]\to\C$ by $\widetilde L(p) :=\wh L(p_\Re)+\i \wh L(p_\Im)$ for any $p\in\C[\bx,\bbx]$.
Then, in view of (\ref{eqLLRe}), we have $\widetilde L(p)=L(p)$ for any $p\in\C[\bx,\bbx]_k$.
For each $\ell\in [K]$ let $w^\ell$ be the complex vector such that $(w^\ell_\Re,w^\ell_\Im)=a^\ell$. Then each $w^\ell$ belongs to $\sD(S)$ and we have 
$$\widetilde L(p) = \wh L(p_\Re)+\i \wh L(p_\Im) = \sum_\ell \lambda_\ell (p_\Re(a^\ell) +\i p_\Im(a^\ell)) =\sum_\ell \lambda_\ell p(w^\ell),$$
which shows $\widetilde L=\sum_\ell\lambda_\ell L_{w^\ell},$ and thus concludes the proof of Theorem \ref{theomainTchakaloff} (ii).

\subsection{Deriving Theorem \ref{theomainmatrix} from its real analog }\label{AppProofCimpric}
In this section we prove the implication (ii) $\Longrightarrow$ (i) in Theorem \ref{theomainmatrix}  from its  real analog  in \cite{CimpricZalar}, which we restate below for convenience. 

\begin{theorem}\cite{CimpricZalar} \label{theomainmatrixREAL}
	Let $S\subseteq \R[\bx]$ be a set of polynomials  such that the quadratic module  $\calMr[](S)$ is Archimedean.
	Let $\calL: \R[\bx] \to \mathcal S^m$ be a matrix-valued linear functional  that is positive on $\calM^\R(S)$, i.e., 
	 $\calL(g f^2) \succeq 0$ for all $g \in S \cup \{1\} $ and $ f \in \R[\bx]$.
	 Then 
	there exists a matrix-valued measure $\mu$ that is supported on $\sDr(S)$ and takes values in the cone $\calS^{m}_+$ of $m \times m$ positive semidefinite matrices such that $\calL(f) =\int fd\mu$ for all $f\in \R[\bx].$
\end{theorem}

We now indicate how to derive the implication (ii) $\Longrightarrow$ (i) in Theorem  \ref{theomainmatrix} from Theorem \ref{theomainmatrixREAL}.
For this let $S\subseteq \C[\bx,\bbx]^h$ such that $\calM(S)$ is Archimedean and let $\calL:\C[\bx,\bbx]\to\calH^m$ which is Hermitian and satisfies $\calL(gp\overline p)\succeq 0$ for all $g\in S$ and $p\in\C[\bx,\bbx]$.
Then the set $S_\Re\subseteq \R[\xr,\xim]$ from (\ref{eqSRe}) has a Archimedean quadratic module $\calM^\R(S_\Re)$ by
\cref{RCArch}. 
Consider  the linear map
$\calL^\R:\R[\xr,\xim]\to\mathcal S^{2m}$ defined via (\ref{calLrdef}).
 Then, by
  \cref{calLgpp}, $\calLr$ is positive on $\calMr[](\Sre)$. Hence we can apply \cref{theomainmatrixREAL} and conclude that $\calLr$ has a representing measure $\mu^\R$, which is supported on $\sDr(\Sre)$ and  takes values in the cone $\calS^{2m}_+$.
	We will now construct a (complex) measure $\mu$, which represents $\calL$ and is supported on the set $\sD(S)$,  using the following two claims.
	
	\begin{claim} \label{calLrincalW}
	The map 	$\calLr$ takes  values in the set 
		$
		\calW := 
		\Big{\{} 
		\begin{bmatrix}  
		A & B^T \\
		B & C
		\end{bmatrix} 
		\in \calS^{2m}: \ A = C,\ B^T = -B
		\Big{\}}.
		$
	\end{claim}

	\begin{proof}
		Since $\calLr$ takes values in $\calS^{2m}$ it has the following block-form
		$$
		\calL^\R
		=
		\begin{bmatrix}  
		\calL^\R_{11} & (\calL^\R_{21})^T  \\
		\calL^\R_{21} & \calL^\R_{22}
		\end{bmatrix},
		$$
		where $\calL^\R_{11}$ and $\calL^\R_{22}$ take values in $\calS^m$, $\calL^\R_{21}$ takes values in $\R^{m\times m}$ and,
by construction, 
		\begin{equation}\label{eqLa}
		\calL^\R_{11}(f) = \calL^\R_{22}(f) = \Re\big(\calL\big(f\big(\frac{\bx +\bbx}{2},\frac{\bx -\bbx}{2\i}\big)\big)\big),\  \ \ 
		\calL^\R_{21}(f) = \Im\big(\calL\big(f\big(\frac{\bx +\bbx}{2},\frac{\bx -\bbx}{2\i}\big)\big)\big)
		\end{equation}
		and $(\calL^\R_{21}(f))^T = -\Im(\calL(f(\frac{\bx +\bbx}{2},\frac{\bx -\bbx}{2\i})))$
		for any $f \in \Rxrxi$. Hence $(\calLr_{21})^T =  -\calLr_{21}$ and thus $\calL^\R$ takes values in $\calW$ as desired.
		\end{proof}

	\begin{claim} \label{muRTakesW}
		Without loss of generality we may assume  the measure $\mu^\R$ takes values in 
		 $\calW \cap \mathcal S^{2m}_+$.
	\end{claim}
	
	\begin{proof}
		We can write the measure  $\mu^\R$ in block-form as   
		$$
		\mu^\R=\begin{bmatrix}  
		\mu^\R_{11} & (\mu^\R_{21})^T \\
		\mu^\R_{21} & \mu^\R_{22}
		\end{bmatrix},
		$$
		where each block is a measure taking its values in $\R^{m\times m}$.
		Then we can define the following new matrix-valued measure
		$$
		\mu' :=
		\frac{1}{2}
		\begin{bmatrix}  
		\mu^\R_{11} + \mu^\R_{22} & -(\mu^\R_{21}-(\mu^\R_{21})^T) \\
		(\mu^\R_{21}-(\mu^\R_{21})^T) &  \mu^\R_{11} + \mu^\R_{22}
		\end{bmatrix}
		 =: 
		 \begin{bmatrix} \mu'_{11} & -(\mu'_{21})^T \\ \mu'_{21} & \mu'_{11}\end{bmatrix}.
		 		$$
		First, by construction, $\mu'$ takes its values in the set $\calW$. Second, $\mu'$ takes its values in  $\calS_+^{2m}$. Indeed, by Theorem \ref{theomainmatrixREAL}, $\mu^\R$ takes values in $\calS_+^{2m}$ and we have
		$$
		\mu' = 
		\frac{1}{2}
		\begin{bmatrix}  
		0 & -I_m \\
		I_m  & 0
		\end{bmatrix}
		\mu^\R
		\begin{bmatrix}  
		0 & I_m \\
		-I_m  & 0
		\end{bmatrix}
		 + 
		 \frac{1}{2}\mu^\R.
		$$
		Finally, $\mu'$ also represents $\calLr$. Indeed, for any $f \in \Rxrxi$ we have $\calLr_{11}(f) = \calLr_{22}(f)$ and $-\calLr_{21}(f)= (\calLr_{21}(f))^T$  by \cref{calLrincalW}. This implies		
		\begin{align*}
		\calLr_{11}(f) &=  \frac{1}{2}(\calLr_{11}(f) + \calLr_{22}(f)) = {1\over 2} \int f (d\mu^\R_{11} + d\mu^\R_{22}) =\int fd\mu'_{11},\\
		\calLr_{21}(f) &=  \frac{1}{2}(\calLr_{21}(f) - (\calLr_{21}(f))^T)  ={1\over 2} \int f (d\mu^\R_{21} -  d(\mu^\R_{21})^T)= \int f d\mu'_{21},
		\end{align*}
and thus 		$\calLr(f) = \int f d\mu'.$
Therefore we may replace the measure $\mu^\R$ by $\mu'$, which shows the claim.
	\end{proof}

	We now define the  complex measure $\mu$ by setting
	\begin{equation} \label{mucomp}
	d\mu := d\mu^\R_{11} \circ\phi  + \i d\mu^\R_{21}\circ\phi,
	\end{equation}
	where $\phi$ is the complex/real bijection in \cref{CRbij}. So $\mu$ 	takes values in $\C^{m\times m}$.
	 	As shown above in Claim \ref{muRTakesW},   $\mu^\R$ takes values in the set $\calW \cap \calS^{2m}_+$. Hence, in view of  \cref{RCMats}, we can conclude that $\mu$ takes its values in $\calH^m_+$. 
	In addition, as $\mu^\R$ is supported by $\sDr(\Sre)$, it follows that $\mu$ is supported by $\sD(S)$.		
	Finally, we verify that $\mu$ represents $\calL$. Indeed, for any $p\in \Cxox$, using (\ref{eqLa}) we obtain  
		\begin{align*}
	\calL(p) &= \calL\big(p_\Re\big(\frac{\bx +\bbx}{2},\frac{\bx -\bbx}{2i}\big)\big) + \i \calL\big(p_\Im\big(\frac{\bx +\bbx}{2},\frac{\bx -\bbx}{2i}\big)\big) \\ 
	&= ( \calLr_{11}(p_\Re(\xr,\xim)) + \i \calLr_{21}(p_\Re(\xr,\xim)) )
	+  \i \big( \calLr_{11}(p_\Im(\xr,\xim) ) + \i\calLr_{21}(p_\Im(\xr,\xim)) \big) \\
	&= 
	\int p_\Re  d\mu^\R_{11}
	+ \i \int p_\Re  d\mu^\R_{21}
	+ \i \int p_\Im d\mu^\R_{11}
	- \int p_\Im d\mu^\R_{21} \\
	&=
	\int (p_\Re + \i p_\Im)  d\mu^\R_{11}
	+\i  \int (p_\Re+  \i p_\Im)  d\mu^\R_{21}\\
	&=
	\int (p_\Re + \i p_\Im)  (d\mu^\R_{11} +  \i d\mu^\R_{21})\\
	&= \int p d\mu.
	\end{align*}
This concludes the proof of the implication (ii) $\Longrightarrow$ (i) in 	Theorem  \ref{theomainmatrix}.

\end{document}